\documentclass[11pt, reqno]{amsart}

\usepackage[margin=0.9in]{geometry}
\usepackage[utf8]{inputenc}
\usepackage[T1]{fontenc}

\usepackage{subfiles}
\usepackage{comment}
\usepackage{float}
\usepackage{marginnote}
\usepackage{euscript}
\usepackage[dvipsnames]{xcolor}   		             		
\usepackage{graphicx}			
\usepackage{amssymb}
\usepackage{mathrsfs}
\usepackage{amsthm}
\usepackage{amsmath, amsrefs}
\usepackage[foot]{amsaddr}
\usepackage{mathtools}
\usepackage[cal=cm]{mathalfa}
\usepackage{bbm}
\usepackage{xypic, dynkin-diagrams}
\usepackage{longtable}
\usepackage[hypertexnames=false, hyperfootnotes=false, colorlinks=true,linktocpage=true, allcolors=highlight, bookmarksdepth = 2]{hyperref}
\usepackage{cleveref}
\usepackage{url}
\usepackage{float}
\usepackage[full]{textcomp}
\makeatletter
\newcommand{\citecomment}[2][]{\citen{#2}#1\citevar}
\newcommand{\citeone}[1]{\citecomment{#1}}
\newcommand{\citetwo}[2][]{\citecomment[,~#1]{#2}}
\newcommand{\citevar}{\@ifnextchar\bgroup{;~\citeone}{\@ifnextchar[{;~\citetwo}{]}}}
\newcommand{\citefirst}{\@ifnextchar\bgroup{\citeone}{\@ifnextchar[{\citetwo}{]}}}

\makeatother

\usepackage{tikz,tikz-cd,tikz-3dplot}
\usepackage{pgfplots}
\pgfplotsset{compat=1.18}
\usetikzlibrary{calc}
\usetikzlibrary{fadings}
\usetikzlibrary{decorations.pathmorphing}
\usetikzlibrary{decorations.pathreplacing}
\usetikzlibrary{patterns}
\usetikzlibrary{arrows,shadows,positioning, calc, decorations.markings, 
	hobby,quotes,angles,decorations.pathreplacing,intersections,shapes}
\usepgflibrary{shapes.geometric}
\usetikzlibrary{fillbetween,backgrounds}

\usepackage{xcolor}

\definecolor{highlight}{HTML}{3465a4}

\usepackage{anyfontsize}
\usepackage[lining]{libertine}%
\usepackage{courier}
\usepackage{csquotes}
\makeatletter
\renewcommand*\libertine@figurestyle{LF}
\makeatother
\usepackage[libertine,libaltvw,liby]{newtxmath}
\usepackage{microtype}

\usepackage{array}
\newcolumntype{H}{>{\setbox0=\hbox\bgroup}c<{\egroup}@{}}

\BibSpec{article}{%
+{}{\sc \PrintAuthors} {author}
+{,}{ } {title}
+{,}{ \textit} {journal}
+{,}{ } {volume}
+{}{ \PrintDate } {date}
+{,}{ } {pages}
+{,}{ } {status}
+{.}{} {transition}
+{}{ \, Preprint available at \tt} {eprint}
+{.}{} {transition}
}

\BibSpec{book}{%
+{}{\sc \PrintAuthors} {author}
+{,}{ } {title}
+{.}{ \textit} {series}
+{,}{ } {volume}
+{.}{ } {publisher}
+{,}{ } {place}
+{,}{ } {date}
+{.}{} {transition}
}

\BibSpec{thesis}{%
+{}{\sc \PrintAuthors} {author}
+{,}{ } {title}
+{.}{ } {type}
+{,}{ } {organization}
+{,}{ \PrintDateB } {date}
+{.}{ \url } {url}
+{.}{} {transition}
}

\BibSpec{misc}{%
+{}{\sc \PrintAuthors} {author}
+{,}{ } {title}
+{.}{ } {note}
+{,}{ \PrintDateB } {date}
+{.}{ \url } {url}
+{.}{} {transition}
}

\BibSpec{incollection}{
+{} {\sc \PrintAuthors} {author}
+{,} { } {title}
+{,} { in \textit} {booktitle}
+{,}{ } {series}
+{,}{ } {volume}
+{}{ \PrintDate } {date}
+{,} { pp.~} {pages}
+{,}{ } {publisher}
+{ }{ Preprint available at \tt} {eprint}
+{.}{} {transition}
}

\BibSpec{inproceedings}{
+{} {\sc \PrintAuthors} {author}
+{,} { } {title}
+{,} { in \textit} {booktitle}
+{,} { pp.~} {pages}
+{,} { \PrintDateB} {date}
+{,} { available at \eprint} {eprint}
+{.} {} {transition}
}

\BibSpec{collection.article}{
+{} {\sc \PrintAuthors} {author}
+{,} { } {title}
+{,} { in \it \PrintConference} {conference}
+{,} { pp.~} {pages}
+{.} {\PrintBook} {book}
+{,} { \PrintDateB} {date}
+{.} {} {transition}
}

\BibSpec{innerbook}{
+{.} { \emph} {title}
+{.} { } {part}
+{:} { \emph} {subtitle}
+{.} { } {series}
+{,} { } {volume}
+{.} { Edited by \PrintNameList} {editor}
+{.} { Translated by \PrintNameList}{translator}
+{.} { \PrintContributions} {contribution}
+{.} { } {publisher}
+{.} { } {organization}
+{,} { } {address}
+{,} { \PrintEdition} {edition}
+{,} { \PrintDateB} {date}
+{.} { } {note}
}

\tikzset{
	commutative diagrams/.cd, 
	arrow style=tikz, 
	diagrams={>=stealth}
}
\tikzset{
	arrow/.pic={\path[tips,every arrow/.try,->,>=#1] (0,0) -- +(0,4pt);},
	pics/arrow/.default={triangle 90}
}
\tikzset{->-/.style={decoration={
			markings,
			mark=at position .6 with {\arrow{latex}}},postaction={decorate}}
}
\tikzset{
	c/.style={every coordinate/.try}
}
\usepackage{enumerate}

\DeclareMathOperator{\Ad}{Ad}

\DeclareMathOperator{\GL}{GL}
\DeclareMathOperator{\age}{age}

\DeclareMathOperator{\vir}{vir}

\usepackage{xifthen}
\newcommand{\Gm}[1][]{%
	\ifthenelse{\isempty{#1}}%
	{\mathbb{C}^\times}
	{(\mathbb{C}^\times)^{#1}}
}

\newcommand{\rd}{\mathrm{d}}
\newcommand{\ri}{\mathtt{i}}
\newcommand{\Beta}{\mathrm{B}}
\newcommand{\re}{\mathrm{e}}

\newcommand{\de}{{\partial}}

\newcommand{\bbO}{\mathbb{O}}
\newcommand{\bbM}{\mathbb{M}}

\newcommand{\bbZ}{\mathbb{Z}}
\newcommand{\bbR}{\mathbb{R}}
\newcommand{\bbC}{\mathbb{C}}
\newcommand{\bbP}{\mathbb{P}}

\newcommand{\bbQ}{\mathbb{Q}}

\newcommand{\cU}{\mathfrak{U}}
\newcommand{\cY}{\mathcal{Y}}
\newcommand{\cZ}{\mathcal{Z}}
\newcommand{\cO}{\mathcal{O}}
\newcommand{\cT}{\mathcal{T}}

\newcommand{\cC}{\mathcal{C}}
\newcommand{\DD}{\mathcal{D}}
\newcommand{\LL}{\mathcal{L}}

\newcommand{\cK}{\mathcal{K}}
\newcommand{\cN}{\mathcal{N}}
\newcommand{\cW}{\mathcal{W}}

\newcommand{\cA}{\mathcal{A}}
\newcommand{\HH}{\mathcal{H}}
\newcommand{\cB}{\mathcal{B}}

\newcommand{\cX}{\mathcal{X}}

\newcommand{\cM}{\mathcal M}
\newcommand{\cQ}{\mathcal Q}

\def\beq{\begin{equation}}                     %
	\def\eeq{\end{equation}}                       %
\def\bea{\begin{eqnarray}}                     
	\def\eea{\end{eqnarray}}
\def\bary{\begin{array}} 
	\def\eary{\end{array}} 
\def\ben{\begin{enumerate}} 
	\def\een{\end{enumerate}}
\def\bit{\begin{itemize}} 
	\def\eit{\end{itemize}}
\def\nn{\nonumber}

\theoremstyle{plain}
\newtheorem{thm}{Theorem}[section]
\newtheorem*{thm*}{Theorem}
\newtheorem{lem}[thm]{Lemma}

\newtheorem{prop}[thm]{Proposition}
\newtheorem*{prop*}{Proposition}
\newtheorem{conj}[thm]{Conjecture}

\newtheorem*{conj*}{Conjecture}

\newtheorem{cor}[thm]{Corollary}
\newtheorem*{cor*}{Corollary}

\theoremstyle{definition}
\newtheorem{defn}[thm]{Definition}
\newtheorem{rmk}[thm]{Remark}

\theoremstyle{plain}
\newtheorem{innercustomthm}{Theorem}
\newenvironment{customthm}[1]
{\renewcommand\theinnercustomthm{#1}\innercustomthm}
{\endinnercustomthm}

\theoremstyle{plain}

\theoremstyle{plain}

\theoremstyle{definition}

\theoremstyle{plain}

\crefname{equation}{Eq.}{Eqs.}
\crefname{eqnarray}{Eq.}{Eqs.}
\crefname{algo}{algorithm}{algorithms}
\crefname{conj}{conjecture}{conjectures}
\crefname{lem}{lemma}{lemmas}
\crefname{thm}{theorem}{theorems}
\crefname{claim}{claim}{claims}
\crefname{rmk}{remark}{remarks}
\crefname{prop}{proposition}{propositions}
\crefname{section}{section}{sections}
\crefname{appendix}{appendix}{appendices}
\crefname{cor}{corollary}{corollaries}
\crefname{figure}{figure}{figures}
\crefname{table}{table}{tables}
\crefname{example}{example}{examples}
\crefname{prob}{problem}{problems}
\crefname{assm}{assumption}{assumptions}
\crefname{defn}{definition}{definitions}
\crefname{notation}{notation}{notations}
\crefname{speculation}{speculation}{speculations}
\crefname{construction}{construction}{constructions}
\crefname{observation}{observation}{observations}
\crefname{innercustomthm}{Theorem}{Theorems}
\crefname{innercustomconj}{Conjecture}{Conjectures}
\crefname{innercustomassumption}{assumption}{Assumption}
\crefname{innerpracticalresult}{practical result}{practical results}

\crefname{equation}{Eq.}{Eqs.}
\crefname{eqnarray}{Eq.}{Eqs.}
\crefname{algo}{Algorithm}{Algorithms}
\crefname{conj}{Conjecture}{Conjectures}
\crefname{lem}{Lemma}{Lemmas}
\crefname{thm}{Theorem}{Theorems}
\crefname{customthm}{Theorem}{Theorems}
\crefname{claim}{Claim}{Claims}
\crefname{rmk}{Remark}{Remarks}
\crefname{prop}{Proposition}{Propositions}
\crefname{section}{Section}{Sections}
\crefname{appendix}{Appendix}{Appendices}
\crefname{appendix}{Appendix}{Appendices}
\crefname{cor}{Corollary}{Corollaries}
\crefname{figure}{Figure}{Figures}
\crefname{table}{Table}{Tables}
\crefname{example}{Example}{Examples}
\crefname{prob}{Problem}{Problems}
\crefname{assm}{Assumption}{Assumptions}
\crefname{defn}{Definition}{Definitions}
\crefname{customconj}{Conjecture}{Conjectures}

\numberwithin{equation}{section}

\title{Integral and higher-genus quantum McKay correspondence in type D}
\author{Andrea Brini$^{1,2}$, Jingxiang Ma$^{1,3}$}
\address{$^1$ University of Sheffield, School of Mathematical and Physical Sciences, Hounsfield Road, Sheffield S3 7RH, United Kingdom}
\address{$^2$ On leave from CNRS, DR~13, Montpellier, France}
\address{$^3$ Department of Mathematics, Ohio State University, 231 West 18th Ave., Columbus, OH
43210, USA}

\email{a.brini@sheffield.ac.uk,ma.3149@osu.edu}

\allowdisplaybreaks
\begin{document}

\setcounter{tocdepth}{2}

\begin{abstract}
We establish Iritani's integral form of the $\bbC^\times$-equivariant Crepant Transformation Conjecture (CTC) for the minimal resolution of the surface singularity $\bbC^2/\mathbb{BD}_{4n}$, where $\mathbb{BD}_{4n}$ is the binary dihedral group of order $4n$. We then use Givental's quantisation formalism to prove the all-genus full-descendent CTC in that case. As an application, we deduce the Bryan--Cadman--Young degree-zero form of the orbifold Gromov--Witten/Donaldson--Thomas correspondence for the Calabi--Yau 3-orbifold $[\bbC^2/\mathbb{BD}_{4n}] \times \bbC$. 
\end{abstract}
\maketitle





\section{Introduction}

\subsection{Crepant transformations and Gromov--Witten theory}
Let $\cX$ and $\cY$ be Deligne--Mumford stacks with Gorenstein coarse moduli space $X$ and $Y$. Recall that $\cX$ and $\cY$ are said to be $K$-equivalent if there is a stack $\cZ$ and projective birational morphisms 
\[
\xymatrix{
& \cZ \ar[ld]_{\pi_\cX} \ar[rd]^{\pi_\cY}  \\
  \cX   \ar@{<-->}[rr] & & \cY\\
  } 
\]
such that $\pi_\cX^* K_\cX \simeq \pi_\cY^* K_\cY$.
A series of influential conjectures of Ruan  \cite{Ruan:2006hvs},  Bryan--Graber \cite{Bryan:2009wkv}, Coates--Corti--Iritani--Tseng \cite{Coates:2009afh} and Iritani \cite{Iritani:2009vat,Iritani:2010jow}, the \emph{Crepant Transformation Conjectures (CTC)}, assert that the Gromov--Witten theories of $\cX$ and $\cY$ should be non-trivially equivalent, and determine each other, as follows. The small orbifold quantum cohomology of $\cY$, viewed as an analytic family of rings over its degree-two Chen--Ruan cohomology $H_{\rm CR}^2(\cY,\bbC)$, should be isomorphic to the small orbifold quantum cohomology of $\cX$ upon analytic continuation near suitable roots of unity in the quantum parameters. When $\cX$ is furthermore an orbifold satisfying the Hard Lefschetz condition \cite{Bryan:2009wkv,Coates:2009afh}, this statement can be enhanced to a Frobenius manifold isomorphism of the big quantum cohomologies of $\cX$ and $\cY$. Even dispensing with the Hard Lefschetz condition, the equivalence of the curve counting theories of $\cX$ and $Y$ has a salient description as an isomorphism of polarised loop spaces of the respective cohomologies (in genus zero) and their canonical quantisation (in higher genus). This isomorphism is further conjecturally induced by a Fourier--Mukai transformation between the $K$-groups of $\cX$ and $\cY$. The genus zero conjectures have been established for toric targets \cite{Coates:2018ifu}, and for a handful of non-toric families of examples (such as symmetric products \cite{Bryan:2009wkv} or some non-abelian quotients of $\bbC^3$ \cite{Bryan:2009xda}).\\

The prototypical example of these Crepant Transformation Conjectures is the ``quantum McKay correspondence''. Let $\cX$ be the quotient stack $[\bbC^2/G]$, where $G<\mathrm{SL}(2;\bbC)$ is finite, and write $X$ for its coarse moduli space, and $Y$ for its $G\text{-}\mathrm{Hilb}(\bbC^2)$ resolution. The McKay correspondence classifies $G$ in terms of root systems of $\mathrm{ADE}_l$ types. In type $\mathrm{A}_l$, $G$ is a cyclic group, hence $X$ and $Y$ are toric varieties, and the Crepant Resolution Conjectures were proved in \cite{Coates:2009afh,Brini:2019wts}. The Ruan/Bryan--Graber conjecture was further proved in type $\mathrm{D}_l$ by Hu \cite{Hu:2013rfh}.

\subsection{Main results}

In this paper we upgrade Hu's result to establish the Coates--Corti--Iritani--Tseng version of the Crepant Transformation Conjecture \cite{Coates:2009afh} in type $\mathrm{D}_l$. We further leverage this to establish an \emph{integral} version of the quantum McKay correspondence through Iritani's $K$-group version of the CTC \cite{Iritani:2009vat,Iritani:2010jow}, as well as a \emph{higher genus} quantum McKay correspondence  using Givental's quantisation formalism for Gromov--Witten theory \cite{Givental:2001mre}. Below we state our results somewhat informally, but provide pointers in each case to the precise statements contained in the main body of the paper.\\

Let $G=\mathbb{BD}_{4l-8}<\mathrm{SL}(2;\bbC)$ be the binary dihedral group of order $4(l-2)$, $l\geq3$. We let $\cX=[\bbC^2/G]$, with $X=\bbC^2/G$ the $\mathrm{D}_l$ Kleinian singularity and $Y=G\text{-}\mathrm{Hilb}(\bbC^2)\to X$ its minimal resolution. We will also endow $X$ with the $\bbC^\times$-action induced by scalar multiplication on $\bbC^2$, together with its lifts to $\cX$ and $Y$. In contrast with type $\mathrm{A}_l$, the group $G$ is non-abelian, $\cX$ is not toric, and the exceptional curve $C$ corresponding to the trivalent node of the Dynkin diagram is fixed pointwise by the torus. For $\cZ$ either $\cX$ or $Y$, we write $\HH(\cZ) \coloneqq H_{\rm CR}^{\bullet}(\cZ, \bbC) \otimes \bbC((z^{-1}))$ for the formal loop space of its Chen--Ruan cohomology.
\begin{customthm}{A}[Integral quantum McKay correspondence in type $\mathrm{D}_l$, =\cref{thm:identity_general,thm:symplectic_general}]
\label{thm:A}
There is an explicit $\bbC\llbracket z^{-1}\rrbracket$-linear symplectic isomorphism $\mathbb{U} : \HH_\cX \to \HH_Y$ which identifies, after analytic continuation, the Lagrangian cones cut by the graph of the differential of the genus-0 descendent GW potentials of $\cX$ and $Y$,
\[
\mathbb{U}(\LL_\cX) = \LL_Y\,.
\] 
Moreover, $\mathbb{U}$ is induced through Iritani's $K$-group $\widehat{\Gamma}$-class framings by the $\bbC^\times$-equivariant derived McKay correspondence $K_{\bbC^\times}(\cX) \xrightarrow{\sim} K_{\bbC^\times}(Y)$.
\end{customthm}

Our second result extends the genus zero quantum McKay correspondence in type $\mathrm{D}_l$ to descendent higher genus Gromov--Witten theory.

\begin{customthm}{B}[Higher genus quantum McKay correspondence in type $\mathrm{D}_l$, =\cref{cor:cqcrc_typeD}]
\label{thm:B}
Under the same analytic continuation of \cref{thm:A}, there exists a triangular change-of-variables, determined by $\mathbb{U}$, identifying the genus-$g$ full-descendent Gromov--Witten potential of $\cX$ with that of $Y$.
\end{customthm}

Our final result combines \cref{thm:B} with the Gromov--Witten/stable-pairs correspondence of Pardon \cite{Pardon:2023cxy} and with the degree-zero Donaldson--Thomas crepant resolution correspondence \cite{Young:2010gfc,Bryan:2012rcb,Calabrese:2016crc} to establish the orbifold GW/DT correspondence for $\cX \times \bbC$, in the form predicted by Bryan--Young and Bryan--Cadman--Young.
\begin{customthm}{C}[Orbifold GW/DT correspondence, =\cref{cor:crc_threefold,thm:gwdt_orbifold,rmk:square}]
\label{thm:C}
For $\cX_3=\cX\times\bbC$ with the Calabi--Yau torus action, the Donaldson--Thomas partition function and the reduced disconnected Gromov--Witten partition function are related, up to classical terms, by
\[
{\rm DT}_{\cX_3}(\mathbf{q}) \;=\; M(-q)^{l+1}\,{\rm GW}'_{\cX_3}(x,\varepsilon)^2 
\,,
\]
under the exponential identification \eqref{eq:gwdt_variables} of their arguments, with $M(q)$ the MacMahon function. 
\end{customthm}

Our strategy follows closely in spirit the one employed in \cite{Brini:2019wts} for type $\mathrm{A}_l$, by leveraging the global mirror we introduced with Strachan in \cite{Brini:2025fbw}. By \cite{Hu:2013rfh,Brini:2025fbw}, the quantum cohomologies of $\cX$ and $Y$ are two charts of a single Frobenius manifold $\cM$, presented as a one-dimensional Landau--Ginzburg model on $\mathfrak{h}\times\bbC$, $\mathfrak{h}$ the Cartan subalgebra of $\mathfrak{so}_{2l}$. The Landau--Ginzburg superpotential of \cite{Brini:2025fbw} coincides with the (logarithm of) the spectral parameter of the Lax operator for the $\widehat{\mathfrak{so}_{2l}}$-relativistic Toda chain. The technical core of the paper rests on the realisation that the associated oscillating integrals (i.e. the twisted periods of $\cM$) are Euler-type integrals, as in \cite{Brini:2025fbw}. Their analytic continuation to the large radius limit point and orbifold point, and the relation to the respective $J$-functions, can be deduced from the monodromy properties of the Lauricella $F_D$ system and its exact reduction to Euler Beta integrals (\cref{sec: mirror_symmetry_typeD}). This yields closed-form expressions for the CCIT symplectomorphism (\cref{sec:Umat_typeD}), from which the statement of Iritani's integral CTC can be verified directly (\cref{sec:integral_McKay}). For \cref{thm:B}, \cite[Lemma~6.3]{Brini:2019wts} reduces \eqref{eq:sqcrc} to the equality of the canonical $R$-matrices of the $\bbC^\times$-equivariant Gromov--Witten theories of $\cX$ and $Y$. Both are computed by the same twisted periods, so they differ at most by a scalar on each idempotent. As in \cite{Brini:2019wts}, at the orbifold point the calibration is pinned by Tseng's orbifold quantum Riemann--Roch theorem \cite{Tseng:2010lhb}, while near the large radius limit point it is fixed by Givental's formula at the fixed points, and in both cases the scalar ambiguities in the $R$-calibration are read off the Stirling expansions of the twisted periods (\cref{sec:higher_genus}). Our analysis requires some additional care compared to the one of \cite{Brini:2019wts}, as the large radius limit point is not semi-simple, forcing the comparison to be performed on a punctured disk in the K\"ahler parameter associated to the point-wise fixed by the torus action. Finally, \Cref{thm:C} follows from \cref{thm:B} by composing, in the square of Bryan--Cadman--Young \cite{Bryan:2012rcb}, with the Gromov--Witten/stable-pairs correspondence of Pardon \cite{Pardon:2023cxy} on $Y\times\bbC$ and with the degree-zero Donaldson--Thomas crepant resolution correspondence of Calabrese \cite{Calabrese:2016crc}, in the form conjectured by Bryan--Young \cite{Young:2010gfc} (\cref{sec:gwdt}).

\subsection{Organisation}
The rest of the paper is organised as follows. \Cref{sec:preliminaries} collects the relevant background material on the classical McKay correspondence for $\mathrm{D}_l$, the equivariant Gromov--Witten theories of $\cX$ and $Y$, Givental's formalism, Iritani's integral structure and Givental's quantisation, and concludes with a discussion of the Crepant Transformation Conjectures in our specific case; the experienced reader will find no new material here. \Cref{sec: mirror_symmetry_typeD} recalls the global mirror of \cite{Brini:2025fbw} and expresses the flat coordinates of its Dubrovin connection as twisted periods. \Cref{sec:Umat_typeD} computes the CCIT symplectomorphism $\mathbb{U}$, which is subsequently identified in \cref{sec:integral_McKay} with the $K$-group isomorphism induced by the derived McKay correspondence, proving \cref{thm:A}. In \cref{sec:higher_genus} we perform the comparison of $R$-calibrations relevant to  \cref{thm:B}, and conclude with the proof of \cref{thm:C} in \cref{sec:gwdt}. \\
In \Cref{sec: twisted_periods} we collect foundational material on the twisted periods of Hurwitz Frobenius manifolds for use in \cref{sec: mirror_symmetry_typeD}. We also derive connection formulas for Lauricella functions in the resonant case in \cref{appendix: Lauricella_func}, which are needed for \cref{sec:Umat_typeD}, and have not to our knowledge previously appeared in the literature on generalised hypergeometric functions beyond the single-variable case.

\subsection*{Statement about Generative AI}
Per the current {\it zeitgeist} in the aftermath of the AIxMaths upheaval of the summer of 2026, we detail the use (or lack thereof) of LLMs for this project. We shall broadly follow what we believe to be the gold standard in the disclosure of uses of generative artificial intelligence, due to Fern\'andez-Real--Serra\footnote{\url{https://arxiv.org/pdf/2609.05669}}. \\
This paper started in Autumn 2024 as a collaborative project forming part of the second author's PhD thesis. An earlier incarnation of the main results of the paper was included, in part, as Chapter~3 of \cite{Ma:2026ths}, which eventually appeared online in July 2026. No use of AI was made in the design, calculations, and write-up of the material up to this point. Around the same time, we became aware of an obstruction to the symplecticity of the CCIT morphism in \cite{Ma:2026ths}, which we ascribed to an inconsistent choice of branches in the analytic continuation of twisted periods. To address this, we subjected our existing human-written {\it Wolfram Language} codebase to an AI-assisted review through interactive sessions with Claude Code (Fable 5.1). This quickly confirmed our suspicion about inconsistent branching conventions and identified for us the source of the discrepancy, which we were then able to rectify. We shall be explicit at three relevant points in the text about the exact nature of the original discrepancy and how Claude Code assisted us  (\cref{rmk:thesis_correction,rmk:thesis_correction_A} and the beginning of \cref{sec:Umat_typeD}). In the write-up stage, we used Claude Code and Codex (GPT-6) for minor typesetting assistance, numerical verifications in low rank of \cref{tab:identity_values,thm:identity_general}, and proof-reading checks. We also refer to \cref{rmk: primitive_eq2} for an AI-suggested correction of a missing term in \cite[Prop.~3.4.4]{Ma:2026ths}.

\subsection*{Acknowledgements} We thank R.~Cavalieri, T.~Coates, G.~Gugiatti, P.~Johnson, N.~Nabijou, D.~Ross, A.~Thompson and K.~van Gemst for discussions. In the final stages of the preparation of this paper, we received a manuscript by B.~Fang and Z.~Lan on the remodelling formalism for threefold $\mathrm{D}_l$-type resolutions, which has some limited overlap with the content of \cref{sec:Umat_typeD,sec:integral_McKay}. We wish to convey our appreciation to the authors for coordinating their submission with ours and kindly acknowledging the priority of the material in this paper and in \cite{Ma:2026ths}. A.~B. was supported by EPSRC through grants ref.~EP/S003657/2 and UKRI2394. J.~M. was supported by a PhD studentship of the EPSRC Doctoral Training Partnership EP/W524360/1.

\section{Preliminaries}
\label{sec:preliminaries}

\subsection{Classical McKay correspondence}
\label{sec:classical_McKay}
The Kleinian subgroup $G< \mathrm{SL}(2;\bbC)$ of type $\mathrm{D}_l$ is the binary dihedral group
\begin{equation*}
G  \coloneqq \mathbb{BD}_{4l-8},
\end{equation*}
generated by
\begin{equation*}
    a = \begin{pmatrix} \zeta & 0 \\ 0 & \zeta^{-1} \end{pmatrix}, \qquad b = \begin{pmatrix} 0 & 1 \\ -1 & 0 \end{pmatrix}, \qquad \zeta = \exp\left(\frac{\ri\pi}{l-2}\right),
\end{equation*}
subject to the relations $a^{l-2} = b^2$, $b^4 = I$, $b a b^{-1} = a^{-1}$. The $l+1$ conjugacy classes of $G$ are $[1]$, $[a^{r}]$ for $1\leq r\leq l-2$ (with $[a^{l-2}]=[-1]$), $[b]$ and $[ab]$, the last two being exchanged by the outer automorphism of $G$ fixing $a$ and sending $b\mapsto ab$. We enumerate them once and for all by the nodes of the affine Dynkin diagram, choosing representatives
\beq
\label{eq:class_enumeration}
g_0\coloneqq 1\,,\qquad g_m\coloneqq a^{m}\quad(1\leq m\leq l-2)\,,\qquad g_{l-1}\coloneqq b\,,\qquad g_l\coloneqq ab\,,
\eeq
so that the trivalent node carries $[-1]=[g_{l-2}]$ and the two legs carry $[b]$ and $[ab]$, in accordance with the $\bbZ_2$-symmetry of the diagram. Henceforth, all objects attached to the conjugacy classes are labelled by $m\in\{0,\dots,l\}$. The centralisers $C_m\coloneqq C(g_m)$ are $C_m = \langle a\rangle$ for $1\leq m\leq l-3$, $C_0 = C_{l-2} = G$, $C_{l-1} = \langle b\rangle$ and $C_l = \langle ab\rangle$, and the inverse of $g_m$ is conjugate to $g_{\bar m}$, where $\bar m = m$ for $m\leq l-2$, while $\overline{l-1} = l-1$, $\bar l = l$ for $l$ even and $\overline{l-1} = l$, $\bar l = l-1$ for $l$ odd (as $b^{-1} = a^{l-2}b$). 

Let $X\coloneqq\bbC^2/G$ be the corresponding Kleinian singularity. We will denote by $\cX\coloneqq [\bbC^2/G]$ the associated quotient stack, and $Y=G\text{-}\mathrm{Hilb}(\bbC^2)$ the minimal crepant desingularisation of $X$:
\beq
\xymatrix{
  \cX \ar[rd]^{\varpi} & & Y \ar[ld]_{\pi} \\
  &X&} 
\label{eq:crdiag}
\eeq
with $\varpi$ and $\pi$ denoting, respectively, the structure map and the resolution morphism. 

By the McKay correspondence \cite{McKay:1980ehc}, the irreducible representations of $G$ are in bijection with the vertices of the affine Dynkin diagram $\widehat{\mathrm{D}}_l$. There are four one-dimensional representations $\rho_0$, $\rho_1$, $\rho_{l-1}$, $\rho_l$, with $\rho_0$ trivial and
\beq
\label{eq:one_dimensional_representations}
\rho_1(a) = 1,\quad \rho_1(b) = -1,\qquad \rho_{l-1}(a) = \rho_l(a) = -1,\quad \rho_{l-1}(b) = -\rho_l(b) = \varepsilon,\qquad \varepsilon\coloneqq\begin{cases} 1 & l\ \text{even},\\ \ri & l\ \text{odd},\end{cases}
\eeq
and $l-3$ two-dimensional representations $\rho_2,\dots,\rho_{l-2}$,
\begin{equation*}
\rho_k(a) =
\begin{pmatrix}
\zeta^{k-1}&\\
&\zeta^{1-k}
\end{pmatrix},\qquad
\rho_k(b) =
\begin{pmatrix}
&1\\
(-1)^{k-1}&
\end{pmatrix},\qquad 2\leq k\leq l-2,
\end{equation*}
so that $\rho_2 = Q$ is the defining representation of $G$ on $\bbC^2$. The McKay graph, whose vertices are the irreducible representations and in which $\rho_v$ and $\rho_{v'}$ are joined by $\dim\mathrm{Hom}_G(\rho_{v'},Q\otimes\rho_v)$ edges, is the affine $\mathrm{D}_l$ Dynkin diagram of \cref{fig: dynkin_diagram_convention}.

\begin{figure}[!h]
\begin{tikzpicture}[scale=1.5]
  \node[circle, draw=black, fill=white, inner sep=1.5pt, label=below:{$\rho_0$}] (rho0) at (-0.8, 0.5) {};
  \node[circle, fill=black, inner sep=1.5pt, label=below:{$\rho_1$}] (rho1) at (-0.8, -0.5) {};
  
  \node[circle, fill=black, inner sep=1.5pt, label=below:{$\rho_2$}] (rho2) at (0, 0) {};
  \node[circle, fill=black, inner sep=1.5pt, label=below:{$\rho_3$}] (rho3) at (1, 0) {};
  \node[circle, fill=black, inner sep=1.5pt, label=below:{$\rho_{l-3}$}] (rhom3) at (3, 0) {};
  \node[circle, fill=black, inner sep=1.5pt, label=below:{$\rho_{l-2}$}] (rhom2) at (4, 0) {};
  
  \node[circle, fill=black, inner sep=1.5pt, label=below:{$\rho_{l-1}$}] (rhom1) at (4.8, 0.5) {};
  \node[circle, fill=black, inner sep=1.5pt, label=below:{$\rho_l$}] (rhom) at (4.8, -0.5) {};

  \draw[dotted] (rho0) -- (rho2);
  \draw (rho1) -- (rho2);
  \draw (rho2) -- (rho3);
  
  \draw (rho3) -- (1.5, 0);
  \draw[dotted] (1.5, 0) -- (2.5, 0);
  \draw (2.5, 0) -- (rhom3);
  
  \draw (rhom3) -- (rhom2);
  \draw (rhom2) -- (rhom1);
  \draw (rhom2) -- (rhom);
\end{tikzpicture}
  \caption{The affine Dynkin diagram of type $\widehat{\mathrm{D}}_l$.}
  \label{fig: dynkin_diagram_convention}
\end{figure}
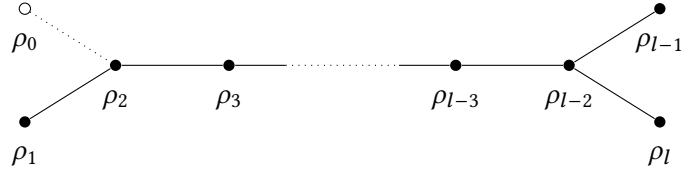

At the level of the resolution $Y$, the exceptional locus of $\pi:Y \to X$ is a configuration of $l$ rational curves $E_1,\dots,E_l$ of self-intersection $-2$, whose dual graph is the Dynkin sub-diagram of type $\mathrm{D}_l$ labelled as in \cref{fig: dynkin_diagram_convention}: under this convention, $C\coloneqq E_{l-2}$ is the curve corresponding to the trivalent node, and the non-trivial representation $\rho_v$, $1\leq v\leq l$, corresponds to the exceptional curve $E_v$. 
By \cite{GonzalezSprinberg:1983gah,Ito:1999kyz}, $Y$ is the moduli space $G\text{-}\mathrm{Hilb}(\bbC^2)$ of $G$-clusters in $\bbC^2$, and the tautological bundles $\mathcal{R}_\rho$ on it, of rank $\dim\rho$ and with $\mathcal{R}_{\rho_0} = \cO_Y$, satisfy $\int_{E_{v'}}c_1(\mathcal{R}_{\rho_v}) = \delta_{vv'}$ for $1\leq v,v'\leq l$. 

\subsection{Equivariant cohomology}
\label{sec:equivariant_cohomology}
The scalar action of $\bbC^\times$ on $\bbC^2$ commutes with $G$ and descends to $X$, and lifts to $\cX$  and $Y$ via the $\bbC^\times$-equivariant diagram \eqref{eq:crdiag}. By construction, the fixed locus on $\cX$ is the origin, and on $Y$ it is contained in the exceptional locus, and it consists of $l-1$ isolated points\footnote{All intersection points of the rational configuration $E_1, \dots, E_l$ are $\bbC^\times$-fixed, hence $C=E_{l-2}$ is point-wise fixed since it is a $\bbP^1$ with three $\bbC^\times$-fixed points. As we shall show in \eqref{eq:localisation_normal_weights}, the pull-back of the $\bbC^\times$-action to $E_i$, $i \neq l-2$ is otherwise effective, hence the rest of the fixed locus consists of the remaining $l-1$ fixed points.} $(p_1, \dots, p_{l-3}, p_{l-1}, p_{l})$, together with the point-wise fixed curve $C$: here $p_i\in E_i$ denotes the fixed point of $E_i$ away from $E_{i+1}$ for $i\leq l-3$, and away from $C$ for $i\in\{l-1,l\}$.  
We shall write $\nu=c_1(\bbO_{B\bbC^\times }(1))$ for the generator of $H^2_{\bbC^\times}(\mathrm{pt})$, and extend scalars from $H^*_{\bbC^\times}(\mathrm{pt})\cong\bbQ[\nu]$ to $\bbC(\nu)$ throughout.

Since $Y$ retracts onto its exceptional locus, its cohomology is concentrated in degrees zero and two, and $H^*_{\bbC^\times}(Y)$ is the free $\bbQ[\nu]$-module of rank $l+1$ with basis $\mathbf{1}_Y$ and the equivariant Poincar\'e duals $[E_1],\dots,[E_l]$ of the exceptional curves, whose intersection matrix
\[\int_{E_i}[E_j] = -\cC_{ij}\]  
is minus the Cartan matrix of $\mathrm{so}_{2l}(\bbC)$, the simple Lie algebra of type $\mathrm{D}_l$. It will be convenient to instead use the basis dual to the curve classes,
\beq
\label{eq:phi_basis}
\phi_j \coloneqq -\sum_{k=1}^l (\cC^{-1})_{jk}\,[E_k]\,,\qquad \int_{E_i}\phi_j = \delta_{ij}\,,\qquad 1\leq i,j\leq l\,,
\eeq
so that, non-equivariantly, $\phi_v = c_1(\mathcal{R}_{\rho_v})$. A general point $\tau\in H^*_{\bbC^\times}(Y)$ is
\beq
\label{eq:tau_coords}
\tau = \tau_0\,\mathbf{1}_Y + \sum_{j=1}^l \tau_j\,\phi_j\,,\qquad \tau_j = \int_{E_j}\tau\,.
\eeq
In this basis the equivariant Poincar\'e pairing $(\alpha,\beta)_{Y}\coloneqq\int^{\bbC^\times}_Y\alpha\cup\beta$ reads
\beq
\label{eq:pairing_Y}
(\mathbf{1}_Y,\mathbf{1}_Y)_Y = \frac{1}{|G|\,\nu^2}\,,\qquad (\mathbf{1}_Y,\phi_j)_Y = 0\,,\qquad (\phi_i,\phi_j)_Y = -(\cC^{-1})_{ij}\,.
\eeq
\smallskip
Since $G$ acts freely on $\bbC^2\setminus\{0\}$, the inertia stack of $\cX$ is
\begin{equation*}
  I\cX = \bigsqcup_{m=0}^{l} \cX_m\,,\qquad \cX_0 = \cX\,,\qquad \cX_m = [\{0\}/C_m]\quad (1\leq m\leq l)\,,
\end{equation*}
so that the twisted sectors $\cX_m$, $m\geq1$, are points with stabiliser $C_m$. The involution $\iota\colon I\cX\to I\cX$, $(g)\mapsto(g^{-1})$, acts on the labels as $m\mapsto\bar m$. Each $g_m$, $m\geq1$, acts on $T_0\bbC^2$ with eigenvalues $\re^{\pm2\pi\ri\theta_m}$, $0<\theta_m<1$, so that its age is \[\age(g_m) = \theta_m + (1-\theta_m) = 1\,.\] The $\bbC^\times$-equivariant Chen--Ruan cohomology \cite{Chen:2004bok}
\begin{equation*}
  H^\ast_{\mathrm{CR},\bbC^\times}(\cX) \coloneqq H^\ast_{\bbC^\times}(I\cX)
\end{equation*}
is therefore the free $\bbQ[\nu]$-module of rank $l+1$ with basis the fundamental classes $\mathbf{1}_m$ of the sectors $\cX_m$, $0\leq m\leq l$, graded by the age-shifted degree $\deg_{\mathrm{CR}}\mathbf{1}_m = 2\age(g_m) = 2$ for $m\geq1$.
We write $\deg_0$ for the unshifted degree, and denote a general point of $H^\ast_{\mathrm{CR},\bbC^\times}(\cX)$ as
\beq
x = \sum_{m=0}^{l} x_m\,\mathbf{1}_{m}\,,\qquad x = (x_0,\dots,x_l)\,.
\label{eq: coord_orbifold_typeD}
\eeq
In this basis, the $\bbC^\times$-equivariant Chen--Ruan Poincar\'e pairing \[(\alpha,\beta)_{\cX}\coloneqq\int^{\bbC^\times}_{I\cX}\alpha\cup\iota^\ast\beta\]
reads
\beq
\label{eq:pairing_X}
(\mathbf{1}_0,\mathbf{1}_0)_{\cX} = \frac{1}{|G|\,\nu^2}\,,\qquad (\mathbf{1}_0,\mathbf{1}_m)_{\cX} = 0\,,\qquad (\mathbf{1}_m,\mathbf{1}_{m'})_{\cX} = \frac{\delta_{m',\bar m}}{|C_m|}\,,\qquad m,m'\geq1\,.
\eeq
As $\age(g) = \age(g^{-1})=1$ for all $1 \neq g\in G$, $\cX$ satisfies the Hard Lefschetz condition of \cite{Bryan:2009wkv,Coates:2009hmp}.

The character in the fundamental representation of $\mathrm{SU}(2)$ satisfies $|\chi_Q(g)| \leq 2$ for all $g \in G$, and we fix the branch of the square root as $\sqrt{\chi_Q(g)-2} \in \ri \bbR_+$. The $\bbC(\nu)$-linear map
\beq
\label{eq:McKay_cohomology}
H^\ast_{\mathrm{CR},\bbC^\times}(\cX)\longrightarrow H^\ast_{\bbC^\times}(Y)\,,\qquad \mathbf{1}_0\longmapsto \mathbf{1}_Y\,,\qquad \mathbf{1}_m\longmapsto \sum_{v=1}^l D_{v,m}\,\phi_v\,,\qquad D_{v,m}\coloneqq\frac{|(g_m)|}{|G|}\sqrt{\chi_Q(g_m)-2}\,\chi_{\rho_v}(g_m)\,,
\eeq
is an isometry of graded inner product spaces \cite{Bryan:2008cym} w.r.t. the Poincar\'e pairings \eqref{eq:pairing_Y} and \eqref{eq:pairing_X}. In the coordinates \eqref{eq:tau_coords} and \eqref{eq: coord_orbifold_typeD} it reads $\tau_0 = x_0$ and
\begin{equation}
\tau_v = \sum_{m=1}^{l} D_{v,m}\,x_m\,,\qquad 1\leq v\leq l\,.
\label{eq: coord_resolution_typeD}
\end{equation}


\subsection{Gromov--Witten invariants}
\label{sec:GW-QC}
We will use $\cZ$ to denote either $\cX$ or $Y$. Let $\cK_{g,n}(\cZ,\beta)$ be the moduli stack of genus-$g$, $n$-pointed twisted stable maps to $\cZ$ of class $\beta\in H_2(\cZ;\bbZ)$ \cite{Chen:2002jow,Abramovich:2008pcn}
It carries a perfect obstruction theory, evaluation morphisms ${\rm ev}_i\colon \cK_{g,n}(\cZ,\beta)\to I\cZ$ \cite[Sec.~4.4]{Abramovich:2008pcn} and cotangent line classes $\psi_i$ at the markings. The $\bbC^\times$-action on $\cZ$ induces one on $\cK_{g,n}(\cZ,\beta)$ with compact fixed locus, and the {$\bbC^\times$-equivariant Gromov--Witten invariants} of $\cZ$ are defined by virtual localisation \cite{Graber:1999eki}:
\begin{equation*}
  \big\langle \alpha_1\psi_1^{k_1},\dots,\alpha_n\psi_n^{k_n} \big\rangle^{\cZ,\bbC^\times}_{g,n,\beta}
  \;\coloneqq\;
  \sum_{F \subset \cK_{g,n}(\cZ,\beta)^{\bbC^\times}}
  \int_{[F]^{\vir}}
  \frac{i_F^\ast\!\Bigl(\prod_{i=1}^n {\rm ev}_i^\ast(\alpha_i)\psi_i^{k_i}\Bigr)}{e_{\bbC^\times}(N^{\vir}_F)}
  \;\in\; \bbC(\nu),\qquad \alpha_i\in H(\cZ),
\end{equation*}
where the sum runs over the connected components $F$ of the fixed locus, $i_F$ is the inclusion, and the Euler class of the virtual normal bundle is inverted in $\bbC(\nu)$. 

For $t\in H(\cZ)$, the {quantum product} $\star_t$ on $H(\cZ)$ is defined by
\begin{equation*}
  (\alpha_1 \star_t \alpha_2,\alpha_3)_{\cZ}
  \;\coloneqq\;
  \sum_{\beta} \sum_{n\ge 0} \frac{1}{n!}
  \big\langle \alpha_1,\alpha_2,\alpha_3,t,\dots,t \big\rangle^{\cZ,\bbC^\times}_{0,n+3,\beta},
\end{equation*}
and the {quantum cohomology} of $\cZ$ is the formal Frobenius manifold over $\bbC(\nu)$
\begin{equation*}
  \mathrm{QH}^\ast_{\bbC^\times}(\cZ)
  \;\coloneqq\;
  \big(H(\cZ), \star_t, (\cdot,\cdot)_{\cZ}, \mathbf{1}\big).
\end{equation*}
Let $\{\phi_\alpha\}$ be a basis of $H(\cZ)$, $\{\phi^\alpha\}$ the dual basis with respect to $(\cdot,\cdot)_{\cZ}$, $t^\alpha$ the corresponding linear coordinates, $\partial_\alpha\coloneqq\partial/\partial t^\alpha$, and $C^\gamma_{\alpha\beta}(t)$ the structure constants of the quantum product, 
\[\phi_\alpha \star_t\phi_\beta = \sum_\gamma C^\gamma_{\alpha\beta}(t)\phi_\gamma\,.\] 
Recall that the Dubrovin connection is the pencil of flat connections on the trivial bundle $TH(\cZ)$
\begin{equation*}
  \nabla_\alpha \coloneqq \partial_\alpha + \frac{1}{z}\, \phi_\alpha \star_t\,,
\end{equation*}
with $z\in \bbC^\times$. A canonical system of flat coordinates for $\nabla$, i.e. of functions whose gradients form a flat frame, is given by the components $J^{\cZ,\gamma}\coloneqq (J^{\cZ},\phi^\gamma)_{\cZ}$ of the $J$-function of $\cZ$
\beq
\label{eq: Jfunc_big}
J^{\cZ}(t, z) \coloneqq z \mathbf{1} + t + \sum_{\beta} \sum_{n \ge 0} \frac{1}{n!} \Big\langle t, \dots, t, \frac{\phi_\gamma}{z-\psi} \Big\rangle^{\cZ,\bbC^\times}_{0,n+1,\beta} \phi^\gamma = z\, S_{\cZ}(t, -z)^{\dagger} \mathbf{1}\,,
\eeq
which satisfy the quantum differential equation
\beq
\label{eq:QDE}
z\,\partial_\alpha\partial_\beta J^{\cZ,\gamma}(t,z) = \sum_\delta C^\delta_{\alpha\beta}(t)\,\partial_\delta J^{\cZ,\gamma}(t,z)\,.
\eeq


\subsection{Givental's symplectic formalism}
\label{sec:givental}
Following Givental \cite{Givental:2001mre,Coates:2007tho}, let
\begin{equation*}
  \HH_{\cZ} \coloneqq H(\cZ)(\!(z^{-1})\!)
\end{equation*}
be the loop space of $H(\cZ)$, equipped with the symplectic form
\begin{equation*}
  \Omega_{\cZ}(f, g) \coloneqq \operatorname{Res}_{z=0} \big( f(-z), g(z) \big)_{\cZ}\, \rd z, \qquad f,g \in \HH_{\cZ},
\end{equation*}
and with the polarisation $\HH_{\cZ} = \HH_{\cZ,+} \oplus \HH_{\cZ,-}$ by the Lagrangian subspaces
\begin{equation*}
  \HH_{\cZ,+} \coloneqq H(\cZ)[z], \qquad \HH_{\cZ,-} \coloneqq z^{-1}H(\cZ)[\![z^{-1}]\!],
\end{equation*}
which identifies $\HH_{\cZ}$ with $T^\ast \HH_{\cZ,+}$. Darboux coordinates are obtained by writing $f\in\HH_{\cZ}$ as
\begin{equation*}
f = \sum_{k\geq 0}\sum_\alpha q^{\alpha,k}\,\phi_\alpha\, z^{k} + \sum_{k\geq0}\sum_\alpha p_{\alpha,k}\,\phi^\alpha\,(-z)^{-k-1}\,.
\end{equation*}
The genus-zero descendent potential
\beq
\label{eq:F0}
\mathcal{F}_0^{\cZ}(\mathbf{t}) \coloneqq \sum_{\beta} \sum_{n \ge 0} \frac{1}{n!} \big\langle \mathbf{t}(\psi), \dots, \mathbf{t}(\psi) \big\rangle^{\cZ,\bbC^\times}_{0,n,\beta},
\qquad \mathbf{t}(z) = \sum_{k\geq 0}\sum_\alpha t^{\alpha,k}\phi_\alpha z^k \in \HH_{\cZ,+},
\eeq
is a formal function on $\HH_{\cZ,+}$, and {Givental's Lagrangian cone} $\LL_{\cZ}\subset\HH_{\cZ}$ is the graph of its differential,
\begin{equation*}
\LL_{\cZ}\coloneqq\Big\{(\mathbf{q},\mathbf{p})\in\HH_{\cZ}\;\Big|\; p_{\alpha,k} = \frac{\partial\mathcal{F}_0^{\cZ}}{\partial q^{\alpha,k}}\Big\}\,,
\end{equation*}
in the dilaton-shifted coordinates $\mathbf{q}(z) = \mathbf{t}(z)-z\mathbf{1}$. By the dilaton equation, $\LL_\cZ$ is a cone with vertex at the origin. In terms of the $\mathrm{End}(H(\cZ))$-valued series
\begin{equation*}
  S_{\cZ}(t, z)\alpha \coloneqq \alpha - \sum_{\beta} \sum_{n \ge 0} \frac{1}{n!}
  \Big\langle \alpha, t, \dots, t, \frac{\phi_\gamma}{z+\psi} \Big\rangle^{\cZ,\bbC^\times}_{0,n+2,\beta} \phi^\gamma\,,
\end{equation*}
$\LL_\cZ$ is furthermore ruled by the subspaces $zT_t$, $T_t\coloneqq S_{\cZ}(t,-z)^{-1}\HH_{\cZ,+}$, $t\in H(\cZ)$, each of which is tangent to $\LL_{\cZ}$ along $zT_t$ \cite{Coates:2007tho}. In particular
\begin{equation*}
J^{\cZ}(t,-z) = -z\,S_{\cZ}(t,-z)^{-1}\mathbf{1}\in\LL_{\cZ}\,,
\end{equation*}
and $\LL_{\cZ}$ is reconstructed from $J^{\cZ}$ by the genus-zero topological recursion relations. Viewed as an element of $\mathrm{End} \HH_\cZ$ by multiplication, the fundamental solution satisfies
\[
S_{\cZ}(t,z)\,S_{\cZ}(t,-z)^{T} = \mathrm{Id}
\]
and is therefore a symplectomorphism, the $S$-calibration of $\cZ$. 

\subsection{Integral structures and Iritani's crepant resolution conjecture}
\label{subsec:integral-structure}
Following Iritani \cite{Iritani:2009vat,Iritani:2010jow}, the solutions of the quantum differential equation \eqref{eq:QDE} carry a natural integral structure, governed by the topological $K$-group $K_{\bbC^\times}(\cZ)$ and by the $\widehat{\Gamma}$-class of $\cZ$, defined as follows. On a component $\cZ_v$ of $I\cZ$, the pull-back of the tangent bundle splits into eigenbundles $T\cZ|_{\cZ_v} = \bigoplus_{0 \le f < 1} (T\cZ)_{v,f}$ on which the stabiliser acts by $\re^{2\pi \ri f}$. If $\delta_{v,f,i}$ are their $\bbC^\times$-equivariant Chern roots, the {equivariant $\widehat{\Gamma}$-class} of $\cZ$ is
\begin{equation*}
  \widehat{\Gamma}_\cZ \coloneqq \bigoplus_v \prod_{0 \le f < 1} \prod_i \Gamma(1 - f + \delta_{v,f,i})\ \in H(\cZ)\,,
\end{equation*}
where $\Gamma$ is Euler's Gamma function, expanded in the Chern roots.
The \emph{integral structure framing} is the map
\beq
   \widehat{\Psi}_{\cZ}\colon K_{\bbC^\times}(\cZ)\to\HH_{\cZ}\,,\qquad \widehat{\Psi}_{\cZ}(V)\coloneqq (2\pi)^{-\frac{\dim \cZ}{2}}  z^{-\mu} \left( \widehat{\Gamma}_\cZ \cup (2\pi \ri)^{\frac{\deg_0}{2}} \iota^\ast\mathrm{ch}(V) \right),
\label{eq:integral_structure_framing}
\eeq
where $\iota\colon I\cZ\to I\cZ$ is the involution $(g)\mapsto(g^{-1})$ of the inertia stack, 
$\mathrm{ch}$ is the $\bbC^\times$-equivariant orbifold Chern character, $\deg_0$ is the unshifted degree operator on homogeneous classes, with $\deg_0\nu = 2$, and $\mu = \frac{1}{2}(\deg-\dim\cZ)$ is the grading operator, $\deg$ being the age-shifted degree. The flat sections $S_{\cZ}(t,z)\widehat{\Psi}_{\cZ}(V)$, $V\in K_{\bbC^\times}(\cZ)$, span a full lattice in the space of multivalued flat sections of $\nabla$, called the \emph{$K$-theoretic integral structure} of $\cZ$.

\subsection{Quantisation}
\label{sec:quantisation}
The {Fock space} $\mathbb{F}_{\cZ}$ is a space of formal functions of the form $\exp\big(\sum_{g\geq0}\varepsilon^{2g-2}f_g(\mathbf{q})\big)$ on $\HH_{\cZ,+}$. 
An infinitesimal symplectic transformation $A$ of $(\HH_{\cZ},\Omega_{\cZ})$ corresponds to the quadratic Hamiltonian $h_A(f)\coloneqq\frac{1}{2}\Omega_{\cZ}(Af,f)$, which is quantised to a differential operator $\widehat{A}\coloneqq\widehat{h_A}$ on $\mathbb{F}_{\cZ}$ by the rules
\begin{equation*}
\widehat{p_{\alpha,k}\,p_{\beta,l}} = \varepsilon^{2}\frac{\partial^2}{\partial q^{\alpha,k}\partial q^{\beta,l}}\,,\qquad
\widehat{p_{\alpha,k}\,q^{\beta,l}} = q^{\beta,l}\frac{\partial}{\partial q^{\alpha,k}}\,,\qquad
\widehat{q^{\alpha,k}\,q^{\beta,l}} = \varepsilon^{-2}\,q^{\alpha,k}q^{\beta,l}\,,
\end{equation*}
and a symplectomorphism $\cQ = \exp A$ of $\HH_{\cZ}$ to $\widehat{\cQ}\coloneqq\exp\widehat{A}$ \cite{Givental:2001mre}. In the neighbourhood of a point $t$ where the quantum product $\star_t$ is semisimple, the idempotent vector fields pairwise commute and induce a canonical coordinate chart $(u_1,\dots,u_{l+1})$, with
\[
\partial_{u_i}\star_t\partial_{u_j} = \delta_{ij}\partial_{u_i}\,.
\]
Writing $\Psi_{\cZ}(t)$ for the component matrix of the normalised idempotents $ (\partial_{u_i},\partial_{u_i})^{-1/2}_{\cZ} \partial_{u_i}$ in the basis $\{\phi_\alpha\}$,  an {$R$-calibration} of $\cZ$ at $t$ is a series
\begin{equation*}
R_{\cZ}(t,z) = \mathrm{Id}+\sum_{k\geq1}R_k(t)\,z^k\in\mathrm{End}(\bbC^{l+1})[\![z]\!]\,,\qquad R_{\cZ}(t,z)\,R_{\cZ}(t,-z)^{T} = \mathrm{Id}\,,
\end{equation*}
such that $\Psi_{\cZ}R_{\cZ}\re^{u/z}$ is a formal flat frame near $z=0$ of the Dubrovin connection. $R_{\cZ}$ is unique up to right multiplication by $\exp\big(\sum_{k\geq1}a_k z^{2k-1}\big)$ with $a_k$ constant diagonal matrices \cite{Givental:2001hvs,Teleman:2012urq}. In the context of torus equivariant Gromov--Witten theory 
of targets with isolated fixed points, or more generally with fixed components as in \cref{subsec:DZ_from_A}, a canonical choice is fixed by the quantum Riemann--Roch theorem of \cite{Givental:2001mre,Tseng:2010lhb}. 
Let $\mathcal{F}_g^{\cZ}(\mathbf{t})$ be the genus-$g$ analogue of \eqref{eq:F0}, and
\beq
\label{eq:descendent_potential}
\mathsf{Z}_{\cZ}(\mathbf{t})\coloneqq\exp\Big(\sum_{g\geq0}\varepsilon^{2g-2}\,\mathcal{F}_g^{\cZ}(\mathbf{t})\Big)
\eeq
the {total descendent potential} of $\cZ$. The {ancestor potential} of $\cZ$ at $t\in H(\cZ)$ is
\begin{equation*}
\mathcal{A}_{\cZ,t}(\bar{\mathbf{t}})\coloneqq\exp\Big(\sum_{g\geq0}\varepsilon^{2g-2}\sum_{\beta}\sum_{n,m\geq0}\frac{1}{n!\,m!}\big\langle\bar{\mathbf{t}}(\bar\psi),\dots,\bar{\mathbf{t}}(\bar\psi),t,\dots,t\big\rangle^{\cZ,\bbC^\times}_{g,n+m,\beta}\Big)\,,
\end{equation*}
where the classes $\bar\psi_i$ at the first $n$ markings are pulled back from $\overline{M}_{g,n}$ along the morphism forgetting the map and the last $m$ markings. The {Givental--Teleman theorem} \cite{Givental:2001hvs,Givental:2001mre,Teleman:2012urq}  expresses the total ancestor and descendent potentials at a semisimple point through the quantised $S$- and $R$-calibrations as
\beq
\label{eq:qgiv}
\mathcal{A}_{\cZ,t} \;=\; \re^{C_{\cZ}(t)-F_1^{\cZ}(t)}\widehat{\Psi_{\cZ}}\widehat{R_{\cZ}}\re^{\widehat{u/z}}\prod_{i=1}^{l+1}\mathsf{Z}_{\rm{pt}}(q^{i})\,, \qquad
\mathsf{Z}_{\cZ} \;=\; \re^{C_{\cZ}(t)}\;\widehat{S_{\cZ}^{-1}}\;\widehat{\Psi_{\cZ}}\;\widehat{R_{\cZ}}\;\re^{\widehat{u/z}}\prod_{i=1}^{l+1}\mathsf{Z}_{\rm{pt}}(q^{i})\,,
\eeq
where $\mathsf{Z}_{\rm{pt}}$ is the descendent potential of a point (the Witten--Kontsevich tau-function), $q^{i}$ are its variables in the $i$-th normalised idempotent direction, and $C_{\cZ}(t)$ is a scalar function. In particular, 
\beq
\label{eq:desc_anc}
\mathsf{Z}_{\cZ} = \re^{F_1^{\cZ}(t)}\;\widehat{S_{\cZ}^{-1}} \mathcal{A}_{\cZ,t}\,,
\eeq
where $F_1^{\cZ}(t)$ is the genus-one primary Gromov--Witten potential of $\cZ$.
\subsection{Crepant resolution conjectures}

\subsubsection{Ruan/Bryan--Graber}
Ruan's conjecture \cite{Ruan:2006hvs}, in the sharpened form of Bryan--Graber \cite{Bryan:2009wkv}, asserts that the quantum cohomologies of $\cX$ and $Y$ are related by analytic continuation and by the affine extension of \eqref{eq: coord_resolution_typeD}. Let
\begin{equation*}
F_0^{\cZ}(t)\coloneqq\sum_{\beta}\sum_{n\geq 3}\frac{1}{n!}\big\langle t,\dots,t\big\rangle^{\cZ,\bbC^\times}_{0,n,\beta}
\end{equation*}
be the genus-zero primary potential, so that $C_{\alpha\beta\gamma} = \partial_\alpha\partial_\beta\partial_\gamma F_0^{\cZ}$.
\begin{conj}[\cite{Bryan:2009wkv}]
\label{conj: BG}
The potential $F_0^{Y}(\tau)$ defines an analytic function of $\re^{\tau_1},\dots,\re^{\tau_l}$ near the large radius limit $\mathrm{Re}\,\tau_j\to-\infty$ and admits an analytic continuation to a neighbourhood of the point $\tau_v = 2\pi\ri\dim\rho_v/|G|$ which agrees (up to quadratic terms) with $F_0^{\cX}(x)$  under the change of variables $\tau_0 = x_0$ and
\begin{equation}
\label{eq: BG_affine}
    \tau_v = 2\pi\ri\frac{\dim\,\rho_v}{|G|}+\sum_{m=1}^{l} D_{v,m}\,x_{m}\,,\qquad 1\leq v\leq l\,.
\end{equation}
In particular, $\mathrm{QH}^\ast_{\bbC^\times}(\cX)\simeq\mathrm{QH}^\ast_{\bbC^\times}(Y)$ as Frobenius manifolds, after analytic continuation.
\end{conj}
\begin{thm}[Hu \cite{Hu:2013rfh}]
\label{thm:Hu}
\Cref{conj: BG} holds.
\end{thm}

\subsubsection{Coates--Corti--Iritani--Tseng}

The Bryan--Graber form of Ruan's Crepant Resolution Conjecture can be phrased in terms of the symplectic geometry of the Givental spaces for $\cX$ and $Y$, which in particular extends beyond the case of Hard Lefschetz orbifolds. In this language, \cref{conj: BG} is the $z \to \infty$ limit of the following statement.
\begin{conj}[\cite{Coates:2009hmp}]
\label{conj:ccit}
There is a degree-preserving $\bbC\llbracket z^{-1}\rrbracket$-linear symplectic isomorphism $\mathbb{U}\colon \HH_{\cX}\to \HH_Y$ such that, under the analytic continuation of \cref{conj: BG}, \[\mathbb{U}(\LL_{\cX}) = \LL_Y\,.\]
\end{conj}
As explained in \cite{Coates:2009hmp}, the fact that the entries of the matrix of $\mathbb{U}$ are formal Taylor series in $z^{-1}$, rather than Laurent series, is a manifestation of the Hard Lefschetz condition, which is satisfied in our case. 

\subsubsection{Iritani}

Iritani's integral structure further bridges the classical McKay correspondence on $K$-groups and \cref{conj:ccit}, by proposing that the symplectomorphism $\mathbb{U}$ is induced by an isomorphism of $K$-groups through the framings \eqref{eq:integral_structure_framing}. In the $\bbC^\times$-equivariant form of \cite{Brini:2019wts} (though already completely implicit in \cite{Iritani:2009vat,Iritani:2010jow}), it reads as follows.
\begin{conj}[\cite{Iritani:2009vat,Iritani:2010jow}]
\label{conj:iritani-integral-structure}
The symplectic isomorphism $\mathbb{U}$ of \cref{conj:ccit} lifts to a $K_{\bbC^\times}(\mathrm{pt})$-linear isomorphism $\mathbb{U}_K\colon K_{\bbC^\times}(\cX)\to K_{\bbC^\times}(Y)$,
\beq
\label{diagram:induced_U_equivariant}
\xymatrix@C=3.5pc@R=3pc{
K_{\bbC^\times}(\cX)\ar[r]^{\mathbb{U}_K}\ar[d]^{\widehat{\Psi}_{\cX}}& K_{\bbC^\times}(Y)\ar[d]^{\widehat{\Psi}_{Y}}\\
\HH_{\cX}\ar[r]^{\mathbb{U}}& \HH_{Y}
}
\eeq
which preserves the Euler pairings.
\end{conj}

\subsubsection{Higher genus}

Let $\mathbb{U} = \mathbb{U}_-\mathbb{U}_0$  with $\mathbb{U}_-\in\mathrm{Id}+z^{-1}\mathrm{End}(H(Y))[\![z^{-1}]\!]$ and $\mathbb{U}_0\colon H(\cX)\to H(Y)$ the isomorphism \eqref{eq:McKay_cohomology}, and let $\widehat{\mathbb{U}}$ be the corresponding exponentiated quantised Hamiltonian.  Givental's higher genus reconstruction suggests that the Fock spaces associated to the polarisations $\LL_\cX$ and $\LL_Y$ should be related through a quantised version of \cref{conj:ccit}.
\begin{conj}[The Hard Lefschetz quantised CRC]
\label{conj:quantised_crc}
Under the analytic continuation of \cref{conj: BG},
\beq
\label{eq:sqcrc}
\mathsf{Z}_{Y} \;=\; \widehat{\mathbb{U}}\,\mathsf{Z}_{\cX}\,.
\eeq
\end{conj}
By \cite[Prop.~5.3]{Givental:2001mre}, \eqref{eq:sqcrc} states that the full descendent potentials of $Y$ and $\cX$ agree to all genera, up to genus-zero terms of degree at most two and a genus-one constant, under the triangular change of variables $\mathbf{t}_{\cX} = [\mathbb{U}^{-1}\mathbf{t}_Y]_+$, where $[\,\cdot\,]_+$ is the projection to $\HH_{\cX,+}$. 

A perhaps more salient phrasing of \eqref{eq:sqcrc} can be made at the level of ancestor potentials. As $\mathbb{U}$ maps the ruling $zT_{x}$ of $\LL_{\cX}$ to $zT_{\tau}$, 
the composition $S_{Y}\,\mathbb{U}\,S_{\cX}^{-1}$ preserves $\HH_+$ and equals $\mathbb{U}_0$. Therefore \eqref{eq:desc_anc} turns \eqref{eq:sqcrc} into the statement that the ancestor potentials of $Y$ and $\cX$ agree, up to a constant, under the same analytic continuation and linear change of variables as the genus-zero potentials of \cref{conj: BG}:
\begin{equation*}
\mathcal{A}_{Y,\tau}(\bar{\mathbf{t}}) = \mathcal{A}_{\cX,x}\big(\mathbb{U}_0^{-1}\bar{\mathbf{t}}\big)\,.
\end{equation*}

\section{Global mirror symmetry}
\label{sec: mirror_symmetry_typeD}
In this section we recall the Landau--Ginzburg mirror construction of \cite{Brini:2025fbw}, and leverage it to provide an Euler-type integral representation of flat coordinates for the Dubrovin connection of $\cX$ and $Y$, as done in \cite{Brini:2019wts,Brini:2018xtg} for type A resolutions.


We start by fixing some notation. Let $\mathfrak{g} = \mathfrak{so}_{2l}(\bbC)$ be the simple complex Lie algebra of type $\mathrm{D}_l$. We write $\mathfrak{h}$ for its Cartan subalgebra, $\cW$ for the Weyl group, $\alpha_1,\dots,\alpha_l$ for the simple roots and $\omega_1,\dots,\omega_l$ for the dual fundamental weights, with labelling consistent with \cref{fig: dynkin_diagram_convention}, and $\cC_{ab} = \langle\alpha_a,\alpha_b^\vee\rangle$ for the Cartan matrix. We write $(s_1,\dots,s_l)$ for the coroot coordinates on $\mathfrak{h}$, $s = \sum_{j=1}^l s_j\alpha_j^\vee$, and $s_{l+1}$ for a linear coordinate on an additional factor $\bbC$. The weights of the defining $2l$-dimensional vector representation are $\pm\epsilon_1,\dots,\pm\epsilon_l$, and we write $w_i\coloneqq\epsilon_i$ for their expansions in simple roots,
\beq
\label{eq:vector_weights}
w_i = \sum_{j=i}^{l-2}\alpha_j + \frac{\alpha_{l-1}+\alpha_l}{2}\,,\quad 1\leq i\leq l-2\,,\qquad
w_{l-1} = \frac{\alpha_{l-1}+\alpha_l}{2}\,,\qquad
w_{l} = \frac{\alpha_{l}-\alpha_{l-1}}{2}\,.
\eeq
Armed with these definitions, we introduce the Saito pair $(\log\lambda,\phi)$, $\phi\coloneqq\rd\log\mu$, given by the Landau--Ginzburg superpotential $\log\lambda(\mu)$ with
\beq
\label{eq:lambdaD}
\lambda(\mu) = \frac{\kappa_{l+1}\prod_{i=1}^l(\mu-\kappa_i)(\mu-\kappa_i^{-1})}{\mu^{l-2}(\mu^2-1)^2}\,,\qquad \kappa_i\coloneqq\re^{\langle\epsilon_i,s\rangle}\,,\qquad \log\kappa_{l+1}\coloneqq\frac{s_{l+1}}{2\nu}\,,
\eeq
and the primitive form $\phi$. As explained in \cite{Borot:2018hwm,Brini:2020fkq,Brini:2022vkq,Brini:2025fbw}, $\lambda$ is the spectral parameter of the Lax matrix of the $\widehat{\mathfrak{so}_{2l}}$ relativistic Toda chain for a degenerate leaf corresponding to vanishing Casimir invariant, and $\lambda=\lambda(\mu)$ is obtained from the corresponding characteristic polynomial in the defining $2l$-dimensional representation of $\mathfrak{so}_{2l}$.
In root coordinates, we have explicitly
\beq
\log\kappa_i = \sum_{j=1}^lG_{ij}s_j\,,\qquad G_{ij}\coloneqq\langle\epsilon_i,\alpha_j^\vee\rangle = \delta_{ij}-\delta_{i,j+1}+\delta_{i,l-1}\delta_{j,l}\,,\qquad (G^TG)_{ij} = \cC_{ij}\,.
\label{eq:kappa_to_yi}
\eeq
The Saito pair $(\log\lambda,\phi)$ defines a non-conformal Frobenius manifold structure $\mathrm{LG}(\log\lambda,\phi)$ on $\mathfrak{h} \times \bbC$ with flat identity, as follows. We define a pairing $\eta$ and product $\star$ on the space of holomorphic vector fields on $\mathfrak{h} \times \bbC$ by the Landau--Ginzburg formulas
\begin{align}
    \eta(\de_{s_i}, \de_{s_j}) =-\sum_{p | \lambda'(p)=0} {\rm {Res}}_{\mu = p} \frac{\de_{s_i}\lambda  \de_{s_j} \lambda}{\lambda\, \mu\, \lambda'} \frac{\rd\mu}{2\nu\mu}\,, \nonumber\\
   \eta(\de_{s_i} \star \de_{s_j}, \de_{s_k}) =-\sum_{p | \lambda'(p)=0} {\rm {Res}}_{\mu = p} \frac{\de_{s_i}\lambda  \de_{s_j} \lambda \de_{s_k} \lambda}{\lambda^2\, \mu\, \lambda'} \frac{\rd\mu}{2\nu\mu}\,.   
\label{eq:LGres}
\end{align}
In terms of the root coordinates we define the linear coordinates
\begin{equation*}
\tau_j = -\langle\alpha_j,s\rangle = -\sum_{k=1}^l \cC_{jk}s_k\,, \quad 1\leq j\leq l\,, \qquad x_0 \coloneqq s_{l+1}\,.
\end{equation*}
Here $x_0$ is the coordinate along the identity class, so that $\kappa_{l+1} = \re^{x_0/(2\nu)}$ and \eqref{eq:kappa_to_yi} reads
\beq
\label{eq:logkappa_weights}
\log\kappa_i = -\sum_{j=1}^l (w_i)_j\,\tau_j\,,  \qquad 1\leq i \leq l\,.
\eeq
With these identifications, we introduce the \emph{large radius limit point} and \emph{orbifold point} $\rm{LR}$, $\rm{OP} \in \mathfrak{h}\times\bbC$ via
\begin{align*}
{\rm LR}  = \{\mathrm{Re}(\tau_i)\to -\infty \},\nn \\
{\rm OP} = \bigg\{\tau_0 = 0\,,\ \tau_v  =2\pi\ri \frac{\dim\rho_v}{|G|}\bigg\}\,.
\end{align*}

The main result of \cite{Brini:2025fbw}, restricted to type $\mathrm{D}_l$, asserts that $\mathrm{LG}(\log\lambda,\phi)$ is a global mirror of the quantum cohomologies of $\cX$ and $Y$.

\begin{thm}[\cite{Brini:2025fbw}]
\label{thm:msQH}
The Landau--Ginzburg formulas \eqref{eq:LGres} define a global holomorphic non-degenerate flat pairing $\eta$ and global holomorphic compatible commutative associative algebra $\star$ on $\cT_{\mathfrak{h}\times \bbC}$, with linear coordinates on $\mathfrak{h}\times \bbC$ being flat for $\eta$. Moreover, there exist analytic neighbourhoods $\cU_Y$ and $\cU_{\cX}$ of the large radius limit and orbifold points such that
\[
\mathrm{QH}_{\bbC^\times}^*(Y) \cong \mathrm{LG}(\log\,\lambda,\phi)\Big|_{\cU_Y}\,, \qquad \mathrm{QH}_{\bbC^\times}^*(\cX) \cong \mathrm{LG}(\log\,\lambda,\phi)\Big|_{\cU_{\cX}}\,.
\]
\end{thm}

We write $\cM\coloneqq\mathrm{LG}(\log\lambda,\phi)$ for this global mirror, so that $\mathrm{QH}^*_{\bbC^\times}(Y)$ and $\mathrm{QH}^*_{\bbC^\times}(\cX)$ are the charts $\cU_Y$ and $\cU_{\cX}$ of $\cM$. We shall set throughout $\hbar\coloneqq2\nu/z$, with $z$ the parameter of the Dubrovin connection in \eqref{eq:QDE}.



\subsection{$J$-functions as hypergeometric Euler integrals}
\label{sec: periods_qde}

We will now construct deformed flat coordinates of $\cM$, and hence of $\mathrm{QH}^*_{\bbC^\times}(Y)$ and $\mathrm{QH}^*_{\bbC^\times}(\cX)$, as twisted periods of the superpotential, as in \cite{Brini:2019wts,Brini:2018xtg} for type $\mathrm{A}_l$. The mechanism, recalled in \cref{sec: twisted_periods} after \cite{Dubrovin:2004dho,Milanov:2017gzt}, is that $\cM$ is the Dubrovin dual of a Frobenius submanifold of a Hurwitz space, and that twisted periods of a Hurwitz Frobenius manifold are deformed flat coordinates of its dual.

Let $\mathscr{H}\coloneqq H_{0;\,l-3,l-3,1,1}$ be the Hurwitz space \cite{Dubrovin:1996onf} of degree $2l$ rational functions on $\bbP^1$ with two poles of order $l-2$ and two double poles. Up to automorphisms of $\bbP^1$ its general element is
\begin{equation*}
\lambda(\mu) = a\,\frac{\prod_{k=1}^{2l}(\mu-\mu_k)}{\mu^{l-2}(\mu+1)^2(\mu-\mu_0)^2}\,,
\end{equation*}
so that $\dim\mathscr{H} = 2l+2$. With the primitive form $\phi = \rd\log\mu$, Dubrovin's construction \cite[Lect.~5]{Dubrovin:1996onf} endows $\mathscr{H}$ with the structure of a semisimple conformal Frobenius manifold $\mathrm{LG}(\lambda,\phi)$ of charge $d=1$ with Euler vector field $E = a\,\partial_a$; we denote by $\circ_t$ the induced product on $\cT_t\mathscr{H}$.
Its Dubrovin-dual Frobenius manifold \cite{Dubrovin:2004dho} is $\mathrm{LG}(\log\lambda,\phi)$,  which is non-conformal, semisimple, and with flat unit given by the Euler vector field $E$ of $\mathrm{LG}(\lambda,\phi)$ \cite{Morrison:2011mre}; we will write $\star_t$ for the associated algebra structure on $\cT_t\mathscr{H}$.

The holomorphic involution $\varsigma\colon\mu\mapsto1/\mu$ of $\bbP^1$ acts on $\mathscr{H}$ by $\lambda\mapsto\lambda\circ\varsigma$, and its fixed locus $\mathscr{H}^{\varsigma}$ is precisely the family \eqref{eq:lambdaD}. This sets $\mu_0=1$, and pairs up the zeroes as $\kappa_i^{\pm1}$, with furthermore $a=\kappa_{l+1}$. Through \eqref{eq:kappa_to_yi}, this identifies $\mathfrak{h}\times\bbC$ with $\mathscr{H}^{\varsigma}$, and, up to the overall factor in \eqref{eq:LGres}, $\cM$ with the Dubrovin-dual of $\mathrm{LG}(\lambda,\phi)$ restricted to $\mathscr{H}^{\varsigma}$: in the notation of \cref{sec: twisted_periods}, the pairing $\eta$ of \eqref{eq:LGres} is $1/(2\nu)$ times the dual metric $g$ of $\mathrm{LG}(\lambda,\phi)$ restricted to $\mathscr{H}^{\varsigma}$, and $\star$ is the restriction of the dual product.

Let $\hbar \in \bbC^\times$, $(\lambda)$ be the divisor of zeros and poles of $\lambda$, and $H^{\rm lf}_1(\bbP^1\setminus(\lambda);\lambda^{\hbar})$ the locally finite homology with coefficients in the rank-one local system of branches of $\lambda^{\hbar}$. For generic $\hbar$ its rank is $2l+2$, one less than the number of finite zeros and poles (\cref{prop:rank_twisted_homology}). For $\gamma$ in this group, the twisted period
\begin{equation*}
\Pi_\gamma \coloneqq \int_\gamma \lambda^{\hbar}\,\phi
\end{equation*}
is a multivalued function on $\mathscr{H}$, holomorphic away from the discriminant. Since $E(\lambda)=\lambda$ and $\nabla E$ is diagonal and constant in flat coordinates, \cref{prop: integral_dual} entails that its gradient \[\xi = (\partial_{t_1}\Pi_\gamma,\dots,\partial_{t_{2l+2}}\Pi_\gamma)\] in flat coordinates $t_\alpha$ of $\mathrm{LG}(\lambda,\phi)$ satisfies the differential system \eqref{eq: qDE_dual},
\begin{equation*}
\partial_\alpha\xi\cdot\mathcal{U} = \xi\cdot\big(\mathcal{V}+\hbar-\tfrac12\big)\cdot C_\alpha\,,
\end{equation*}
where $\mathcal{U} = E\,\circ$, $\mathcal{V} = \frac12\operatorname{Id} - \nabla E$ and $C_\alpha$ is the matrix of $\partial_{t_\alpha}\circ$; this is the system for deformed flat coordinates of the dual $\mathrm{LG}(\log\lambda,\phi)$ \cite[Prop.~5.1]{Dubrovin:2004dho}.

\begin{lem}
\label{lem:periods_restrict}
The involution $\varsigma$ is an automorphism of $\mathrm{LG}(\lambda,\phi)$, so that $\mathscr{H}^{\varsigma}$ is a natural Frobenius submanifold \cite{Strachan:2004chj}, and for every $\gamma\in H^{\rm lf}_1(\bbP^1\setminus(\lambda);\lambda^{\hbar})$ the restriction of $\Pi_\gamma$ to $\mathscr{H}^{\varsigma}\simeq\mathfrak{h}\times\bbC$ is a deformed flat coordinate of $\cM$: in the flat coordinates $s_1,\dots,s_{l+1}$,
\beq
\label{eq:qde_M}
\partial_{s_i}\partial_{s_j}\Pi_\gamma = \hbar\sum_{k=1}^{l+1} C^{k}_{ij}\,\partial_{s_k}\Pi_\gamma\,,
\eeq
with $C^k_{ij}$ the structure constants of the dual product restricted to $\mathscr{H}^{\varsigma}$.
\end{lem}
\begin{proof}
Since $\varsigma^*\phi = -\phi$ and $\partial_{\log\mu}(\lambda\circ\varsigma) = -(\partial_{\log\mu}\lambda)\circ\varsigma$, the residue formulas for $\eta$ and $\circ$ are invariant under $\lambda\mapsto\lambda\circ\varsigma$, which also preserves $E$. Hence $T\mathscr{H}|_{\mathscr{H}^\varsigma}$ splits into the $\pm1$-eigenbundles of $\rd\varsigma$: the $(+1)$-eigenbundle is $T\mathscr{H}^\varsigma$, the $(-1)$-eigenbundle is its $\eta$-orthogonal complement, and every structure constant with two tangent and one normal index vanishes by invariance of $\eta$ and $\circ$. The same holds for the dual metric $g(X,Y)=\eta(E^{-1}\circ X,Y)$ and the dual product $X\star Y = E^{-1}\circ X\circ Y$, since $E$ is tangent. Restricting the dual system \eqref{eq: qDE_dual}, written in the flat coordinates of $\mathrm{LG}(\log\lambda,\phi)$, to $\mathscr{H}^\varsigma$ under these conditions gives \eqref{eq:qde_M}, by \cref{prop:qDE_dual_sub_Hurwitz} applied to $(g,\star)$.
\end{proof}

\begin{prop}
\label{prop:period_coordinates}
Let $\gamma_i\coloneqq[0,\kappa_i^{-1}]$ for $1\leq i\leq l$ and $\gamma_{l+1}\coloneqq[0,1]$. For $\lambda\in\mathscr{H}^{\varsigma}$ and generic $\hbar$, the segments \[\gamma_1,\dots,\gamma_{l+1},\varsigma(\gamma_1),\dots,\varsigma(\gamma_{l+1})\] form a basis of $H^{\rm lf}_1(\bbP^1\setminus(\lambda);\lambda^{\hbar})$. Moreover, the periods $\Pi_i\coloneqq\Pi_{\gamma_i}$, $1\leq i\leq l+1$, are local coordinates on $\cM$ near any semisimple point.
\end{prop}
\begin{proof}
The $2l+2$ segments form a tree connecting all zeros and poles of $\lambda$ other than $-1$, hence a basis, the rank being $2l+2$. By \cref{prop:period_map_full_Hurwitz}, the $2l+2$ periods $\Pi_{\gamma_i}$, $\Pi_{\varsigma(\gamma_i)}$ are local coordinates on $\mathscr{H}$ near a semisimple point. As $\varsigma^*\phi=-\phi$, the involution acts on them by $\Pi_{\gamma_i}\mapsto-\Pi_{\varsigma(\gamma_i)}$, so its fixed locus $\mathscr{H}^\varsigma$ is cut out by $\Pi_{\gamma_i}+\Pi_{\varsigma(\gamma_i)}=0$, and $\Pi_1,\dots,\Pi_{l+1}$ restrict to coordinates on it.
\end{proof}

\begin{prop}
\label{prop:periods_flat_coordinates}
With $\hbar=2\nu/z$, the periods
\beq
\label{eq:Piper}
\Pi_i = \int_{\gamma_i}\lambda^{\hbar}\,\phi\,,\qquad 1\leq i\leq l+1\,,
\eeq
are a flat coordinate chart for the Dubrovin connection of $\cM$. In particular, they are a basis of solutions of the quantum differential equation \eqref{eq:QDE} of $\mathrm{QH}^*_{\bbC^\times}(Y)$ on $\cU_Y$ and of $\mathrm{QH}^*_{\bbC^\times}(\cX)$ on $\cU_{\cX}$.
\end{prop}
\begin{proof}
This is immediate from \Cref{lem:periods_restrict,prop:period_coordinates} and \cref{thm:msQH}. The value $\hbar=2\nu/z$ is fixed by the normalisation of the identity direction in \eqref{eq:lambdaD}, $\partial_{\log\kappa_{l+1}} = 2\nu\,\partial_{s_{l+1}}$, which matches the parameter $z$ of \eqref{eq:QDE}.
\end{proof}

We use the short-hand notation $\Pi \coloneqq (\Pi_1, \dots, \Pi_{l+1})$ for the column vector of periods.
 
Pochhammer integrals \eqref{eq:Piper} of products of complex powers of rational functions are Lauricella functions (see \cref{sec: lauricella_intro} for details). In the specific case of the periods \eqref{eq:Piper}, we get the following

\begin{prop}
\label{prop:hypergeometric_expression}
The integrals $\Pi_i$ can be expressed as:
\begin{align*}
     \Pi_i  = & \kappa_{l+1}^{\hbar} \Beta\big(-(l-2)\hbar,1+\hbar\big)\,\kappa_i^{(l-2)\hbar} \nonumber\\
     &F_D^{(2l+1)}\bigg(-(l-2)\hbar; -\hbar,\cdots,-\hbar,2\hbar,2\hbar,-\hbar,\cdots,-\hbar;1-(l-3)\hbar\nonumber\\
    &\bigg| \frac{1}{\kappa_1 \kappa_i}, \cdots, \frac{1}{\kappa_l\kappa_i},\frac{1}{\kappa_i}, -\frac{1}{\kappa_i},\frac{\kappa_l}{\kappa_i},\cdots,\widehat{\frac{\kappa_i}{\kappa_i}},\cdots, \frac{\kappa_1}{\kappa_i} \bigg),
\end{align*}
and
\begin{align*}
    \Pi_{l+1} = & \kappa_{l+1}^{\hbar}\Beta\big(-(l-2)\hbar,1-2\hbar\big)\nonumber\\
    &F_D^{(2l+1)}\bigg(-(l-2)\hbar; -\hbar,\cdots,-\hbar,2\hbar,-\hbar,\cdots,-\hbar;1-l\hbar 
    \bigg| \frac{1}{\kappa_1}, \cdots, \frac{1}{\kappa_l}, -1,\kappa_l,\cdots, \kappa_1 \bigg).
\end{align*}
Here $\Beta(x,y)\coloneqq\Gamma(x)\Gamma(y)/\Gamma(x+y)$ denotes Euler's Beta function, and $F_D^{(N)}(a; b_1,\dots, b_N; c | x_1,\dots,x_N)$ is the generalised hypergeometric Lauricella series $F_D^{(N)}$  in $N$-variables \eqref{eq:lau_series}.

\end{prop}

\section{The CCIT symplectomorphism from mirror symmetry}
\label{sec:Umat_typeD}
\subsection{Analytic continuation}
\label{sec:branches}
In the previous section, we have constructed three flat coordinate charts for the Dubrovin connection over $\cM$:
\ben
\item locally over $\cU_Y$: the components $J^{Y}(\tau, z)$ of the $J$-function of $Y$ with respect to a linear basis of $H(Y)$;
\item locally over $\cU_\cX$: the components $J^{\cX}(x, z)$ of the $J$-function of $\cX$ with respect to a linear basis of $H(\cX)$;
\item globally over $\cM$: the (multi-valued) twisted periods $\Pi = (\Pi_i)_{i=1}^{l+1}$ of the superpotential \eqref{eq:lambdaD}.
\een
%
We will use (3) to intertwine between (1) and (2). In the coordinates $(\kappa_i)_{i=1}^{l+1}$, the large radius limit point of $Y$ corresponds to taking
\begin{align*}
    \frac{1}{\kappa_k}\ll 1\,, \qquad \frac{\kappa_{i+1}}{\kappa_i}\ll 1\,,\qquad 1\leq k\leq l\,,\quad 1\leq i\leq l-1\,,
\end{align*}
while the orbifold point sits at
\begin{align}
\label{eq: OP_DN}
\log\kappa_k|_{\rm OP}=-\pi\ri 
\begin{cases}
1 & k=1\,, \\
\frac{2l-2k-1}{2(l-2)}\,, & 2\leq k\leq l-1\,, \\
0 & k=l\,.
\end{cases}
\end{align}
We shall first fix an analytic continuation path $\rho$ from ${\rm LR}$ to ${\rm OP}$ as follows. Picking $R\gg0$ and $0<\delta<\pi/l$, define $\rho\colon[0,1]\to\cM$, keeping $x_0$ constant, by
\beq
\label{eq:rho}
\log\kappa_k(\rho(s)) = (1-s)\,(l+1-k)\,(R-\ri\delta)+s\,\log\kappa_k\big|_{\rm OP}\,,\qquad 1\leq k\leq l\,.
\eeq
For $0\leq s<1$ one has $|\kappa_1|>\dots>|\kappa_l|>1$. In the representation of $\Pi_i$ in \cref{prop:hypergeometric_expression}, the arguments $1/(\kappa_k\kappa_i)$, $\pm1/\kappa_i$ and $\kappa_k/\kappa_i$ with $k>i$ therefore have modulus less than one along $\rho$, whereas $\kappa_k/\kappa_i$ with $k<i$ have modulus greater than one: the continuation of $\Pi_i$ along $\rho$ is obtained by expanding in the former and applying the connection formula of \cref{prop: acformula_generic} in the latter, as done in \cref{subsec:LR_general}. Moreover, since $-\pi<\arg\kappa_1<\dots<\arg\kappa_l<0$, the zeros $\kappa_k^{-1}$ of $\lambda$ stay in the upper half plane in the order of the contours $\gamma_k$, and never meet $[0,1]$. 

As the integrand $\lambda^{\hbar}$ is multivalued, the sheet on which each $\Pi_j$ is analytically continued along $\rho$ has to be carefully fixed in a consistent manner \footnote{Failure to do so initially led us to slight errors in the expansion of $\Pi_j$ and its derivatives near $\rm{LR}$ and ${\rm OP}$. We acknowledge a debugging analysis of our {\it Wolfram Language} code through an interactive session with Anthropic Claude Code, which was essential in helping us single out the error of our ways here.} between $\rm{LR}$ and $\rm{OP}$. Near $\mu = 0$ one has $\lambda\sim \kappa_{l+1}\,\mu^{-(l-2)}$, so that along $\gamma_j = [0,1/\kappa_j]$, parametrised as $\mu = s/\kappa_j$ with $s\in(0,1)$, we have
\begin{equation*}
\lambda^{\hbar} \sim \big(\kappa_{l+1}\kappa_j^{\,l-2}\big)^{\hbar} s^{-(l-2)\hbar}\,,
\end{equation*}
hence $\arg\lambda\to (l-2)\,\mathrm{Im}\log\kappa_j + \mathrm{Im}\log\kappa_{l+1}$ as $\mu\to 0$ along $\gamma_j$. Along the path \eqref{eq:rho}, on which $-\pi<\arg\kappa_1<\dots<\arg\kappa_l< 0$, no further crossing of branch cuts occurs, and this prescription is transported to $\rm{OP}$. By \eqref{eq: OP_DN}, the resulting sheets for $\Pi_k$ differ from the principal one by integer powers of $\re^{-2\pi\ri\hbar}$.

With these conventions for branches and analytic continuation path, flatness of $\Pi$ implies the following
\begin{prop}
\label{prop:period_comparison_U}
There exist unique invertible $(l+1)\times(l+1)$ matrices $\cA,\cB$ with entries meromorphic in $\hbar$ such that, in the bases of $H(Y)$ and $H(\cX)$ fixed below,
    \begin{equation*}
    \cA\,\Pi\big|_{\cU_Y} = J^{Y}\,, \qquad
    \Pi\big|_{\cU_\cX} = \cB\,J^{\cX}\,.
\end{equation*}
\end{prop}
\Cref{prop:period_comparison_U} implies that, under analytic continuation of $J^{\cX}$ along $\rho$ traversed from $\mathrm{OP}$ to $\mathrm{LR}$,
\begin{equation*}
J^{Y} = \mathbb{U} \, J^{\cX}\,,
\end{equation*}
with $\mathbb{U} \coloneqq \cA \cB$.

\begin{rmk}
   Varying the choice of path, and hence the determination of the twisted periods $\Pi$, results in a transformation of $\cA$ by an associated monodromy matrix. Explicitly, let $\rho$ and $\rho'$ be two paths from $\rm{OP}$ to $\rm{LR}$, then there exists $\bbM$ such that
    \begin{equation*}
    \Pi^{\rho} = \bbM\,\Pi^{\rho'}\,,
    \end{equation*}
    resulting in
    \begin{equation*}
    \cA_{\rho'} = \cA_{\rho}\,\bbM.
    \end{equation*}
\end{rmk}

Following \cite{Brini:2019wts}, we will achieve the comparison between twisted periods and $J^{Y}$ near $\mathrm{LR}$ by comparing monodromies. For this purpose we fix once and for all bases of $H_{\bbC^\times}(Y)$ and $H_{\bbC^\times}(\cX)$, and from now on identify linear maps between $\HH_{\cX}$ and $\HH_{Y}$ with their matrices in these bases, acting on the vectors of components. In particular, $\mathbb{U}=\cA\cB$ sends the components of $J^{\cX}$ to those of $J^{Y}$. For monodromy calculations, the most convenient choice is the localised basis \eqref{eq:loc_basis_general} of $H_{\bbC^\times}(Y)$ introduced below, together with the ordered basis $(\mathbf{1}_1,\dots,\mathbf{1}_l,\mathbf{1}_0)$ of \eqref{eq: coord_orbifold_typeD} for $H_{\bbC^\times}(\cX)$, untwisted class last. 

We start with the resolution $Y$. Throughout, we set $n\coloneqq l-2$, so that $|G|=4n$, with the fixed points $(p_1, \dots, p_{n-1}, p_{l-1}, p_{l})$ labelled as in \cref{sec:equivariant_cohomology}. For example, in the $\mathrm{D}_4$ case we have
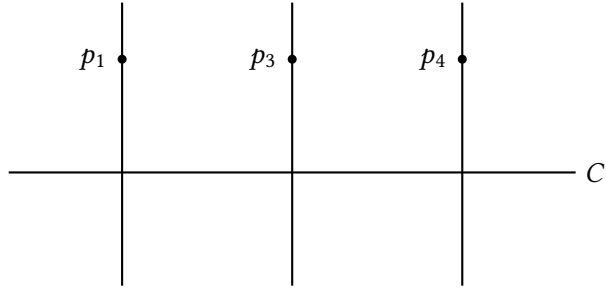
\begin{figure}[H]
    \centering
    \begin{tikzpicture}[scale=1.5]

    \tikzset{fixedpoint/.style={fill=black, circle, inner sep=1.2pt}}

    \draw[thick] (-2.5, 0) -- (2.5, 0) node[right] {$C$};


    \draw[thick] (-1.5, -1) -- (-1.5, 1.5);
    \node[fixedpoint] at (-1.5, 1) {};
    \node[left=2pt] at (-1.5, 1) {$p_1$};

    \draw[thick] (0, -1) -- (0, 1.5);
    \node[fixedpoint] at (0, 1) {};
    \node[left=2pt] at (0, 1) {$p_3$};

    \draw[thick] (1.5, -1) -- (1.5, 1.5);
    \node[fixedpoint] at (1.5, 1) {};
    \node[left=2pt] at (1.5, 1) {$p_4$};

\end{tikzpicture}
\end{figure}

Denoting by $i\colon Y^{\bbC^\times}\hookrightarrow Y$ the inclusion of the fixed locus, the localisation map
\begin{equation*}
i^*\colon H_{\bbC^\times}^\ast(Y)\to H_{\bbC^\times}^\ast(Y^{\bbC^\times}) \cong H^\ast_{\bbC^\times}(C)\oplus\bigoplus_{i\neq l-2}H^\ast_{\bbC^\times}(p_i)
\end{equation*}
is an isomorphism of vector spaces over $\bbC(\nu)$. 

Take $\{\varphi_i \coloneqq [E_i]\}_{i=1}^l\cup \{\mathbf{1}_Y\}$ as the basis of $H_{\bbC^\times}^\ast(Y)$, as in \cref{sec:equivariant_cohomology}. The restriction of $\varphi_i$ to $E_i$ is $c_1^{\bbC^\times}(\cN_{E_i/Y})$. Let $P_i,i\neq l-2$ be the fundamental class of $p_i$, $\mathbf{1}_C$ be the fundamental class of $C$ and $\omega_C = c_1(\cO_C(1))$. We will refer to 
\beq
\label{eq:loc_basis_general}
\big(P_1,\dots,P_{n-1},\mathbf{1}_C,\omega_C,P_{l-1},P_l\big)
\eeq
as the localised basis of $H_{\bbC^\times}(Y)$. The image of $\mathbf{1}_Y$ is just the identity 
\begin{equation*}
i^*(\mathbf{1}_Y) = \mathop{\sum_{1\leq j\leq l}}_{j\neq l-2}\; P_j+ \mathbf{1}_C,
\end{equation*}
and the image of $\varphi_i$ can be computed from the $\bbC^\times$-weights of $\cN_{E_i/Y}$ at the fixed points. Notice that
\begin{enumerate}
    \item[(W1)] at every $\bbC^\times$-fixed point of $Y$, the two weights of the tangent representation add to $2\nu$;
    \item[(W2)] if an irreducible curve has two fixed points, then tangent weights along the curve add to $0$.
\end{enumerate}
The images of $\varphi_i$ are therefore computed as
\beq
\label{eq:localisation_normal_weights}
i^*(\varphi_i) = 
\begin{cases}
    w_i^L\;P_i+ w_i^R\; P_{i+1}&1\leq i\leq l-4\\
    4\nu\; P_i + 1\;\omega_C& i= l-3,\,l-1,\, l\\
    -2\;\omega_C+ 2\nu\;\mathbf{1}_C& i=l-2
\end{cases},
\eeq
where $\{w_i^L,w_i^R\}_{i=1}^{l-4}$ are the unique solutions to
\begin{equation*}
2\nu- w_i^L = -2\nu + w_i^R\;,\; w_i^R + w_{i+1}^L = 2\nu\;,\;w_{l-4}^R = -2\nu.
\end{equation*}
We now compute the components of $J^{Y}$ in the localised basis \eqref{eq:loc_basis_general}. The string and divisor equations applied to \eqref{eq: Jfunc_big} give
\begin{equation*}
J^{Y}(\tau, z) = z\, \re^{\tau/z}\bigg(\mathbf{1}_{Y}+ \sum_{\beta\neq 0}\re^{\langle\beta,\tau\rangle}\sum_{\gamma}\Big\langle\frac{\phi_\gamma}{z(z-\psi)}\Big\rangle^{Y,\bbC^\times}_{0,1,\beta}\phi^\gamma\bigg)\,,\qquad \langle\beta,\tau\rangle\coloneqq\int_\beta \tau = \sum_j d_j\,\tau_j\,,
\end{equation*}
from which we see that 
\begin{equation*}
J^{Y}(\tau,z) \sim z\,\re^{\tau/z}\mathbf{1}_Y\;,\; \tau \to \rm{LR}. 
\end{equation*}
Writing $\tau$ in the localised basis yields
\beq
\label{eq:localisation_morphism}
i^*(\tau) =:\mathop{\sum_{1\leq j\leq l}}_{j\neq l-2}L^j(\tau)\; P_j + L^{l-2}(\tau)\;\mathbf{1}_C
+ L^C(\tau)\;\omega_C\,.
\eeq
Then, expanding
\beq
J^{Y}(\tau,z) =\mathop{\sum_{1\leq j\leq l}}_{j\neq l-2}J^{Y}_j\;P_j + J^{Y}_{l-2}\;\mathbf{1}_C + J^{Y}_C\;\omega_C
\label{eq:JYcomps}
\eeq
we have
\beq
\label{eq:asymp_points}
J^{Y}_j \sim
\begin{cases}
    z\, \re^{L^j(\tau)/z}&j\neq l-2\\
    z\, \re^{L^{l-2}(\tau)/z}&j= l-2
\end{cases},
\eeq
and
\beq
\label{eq:asymp_curve}
J^{Y}_C \sim L^C(\tau)\, \re^{L^{l-2}(\tau)/z},
\eeq
since $\omega_C^2 = 0$. The components of $J^{Y}$ therefore triangularise the monodromy of the Dubrovin connection around ${\rm LR}$, which is unipotent with a single Jordan block of size two, and the matrix $\cA$ can then be computed by decomposing the asymptotics of the twisted periods with respect to the generalised eigenvectors of the large-radius monodromy.

We now turn to $\cX$. Note that the orbifold point ${\rm OP}$ is a regular point for the quantum differential equation of $\cX$, meaning that any flat frame for the Dubrovin connection of $\cM$ is analytic at $x=0$. Since the components of $J^{\cX}$ at $\rm{OP}$ satisfy
\begin{align*}
J^{\cX}(0,z) &= z\;\mathbf{1}_0,\\
\frac{\partial J^{\cX}}{\partial x_m}(0,z)&= \mathbf{1}_m,\qquad 1\leq m\leq l,
\end{align*}
to compute $\cB$ it suffices to study the Jacobian matrix of the twisted periods at $\rm{OP}$. Explicitly, writing $x_0$ for the coordinate along $\mathbf{1}_0$ and normalising the overall scaling of the superpotential as $\kappa_{l+1} = \re^{x_0/(2\nu)}$, so that $\Pi_j\propto \re^{x_0/z}$,
\beq
\label{eq:B_structure}
\cB_{j,0} = \frac{1}{z}\,\Pi_j\big|_{\rm OP}\,,\qquad
\cB_{j,m} = \frac{\partial \Pi_j}{\partial x_m}\Big|_{\rm OP} = -\sum_{i=1}^l\sum_{v=1}^l (w_i)_v\, D_{v,m}\;\frac{\partial \Pi_j}{\partial \log\kappa_i}\Big|_{\rm OP}\,,\qquad 1\leq m\leq l,
\eeq
where $D_{v,m}$ is given in \eqref{eq:McKay_cohomology} and $w_i$ are the weights \eqref{eq:vector_weights}, by \eqref{eq:logkappa_weights} and \eqref{eq: coord_resolution_typeD}.

\begin{rmk}
\label{rmk:lambda_OP}
The superpotential is subject to a confluence of its zeroes and poles at $\rm{OP}$. From \eqref{eq: OP_DN}, we get
\begin{equation*}
\lambda_0\coloneqq\lambda\big|_{\rm{OP}} = \mu^{n}+\mu^{-n}\,.
\end{equation*}
In particular $\kappa_1$ and $\kappa_l$ collide with the double poles $\mu=-1$ and $\mu = 1$ of $\lambda$, and the remaining $\kappa_k$, $2\leq k\leq l-1$, are zeros of $\lambda|_{\rm OP}$, i.e. $\kappa_k^{2(l-2)} = -1$.
\end{rmk}

\subsection{Computing $\cB$}
\label{subsec:B_general}
We now compute $\cB$. Throughout, all periods are taken on the sheets fixed at the beginning of this section. By \eqref{eq:B_structure}, the column $m=0$ of $\cB$ is the vector of values $\Pi_j|_{\rm OP}/z$, and the columns $1\leq m\leq l$ are the product of the Jacobian matrix $\big(\partial\Pi_j/\partial\log\kappa_i\big)\big|_{\rm OP}$ of the twisted periods in the flat logarithms with the Jacobian matrix $\big(\partial\log\kappa_i/\partial x_m\big) = -\big(\sum_v(w_i)_v D_{v,m}\big)$ of the flat logarithms in the orbifold coordinates. Differentiating \eqref{eq:lambdaD} under the integral sign,
\beq
\label{eq:period_derivative}
\frac{\partial \Pi_j}{\partial\log\kappa_i} = \hbar\int_{\gamma_j}\lambda^{\hbar}\Big(-\frac{\kappa_i}{\mu-\kappa_i}+\frac{\kappa_i^{-1}}{\mu-\kappa_i^{-1}}\Big)\frac{\rd\mu}{\mu}\,,
\eeq
the moving endpoint of $\gamma_i$ not contributing since $\lambda$ vanishes there. All integrals below are Euler integrals in which the exponents $-n\hbar$ at $\mu=0$ and $\hbar$ at the zeros of $\lambda$ may be treated as independent: they are evaluated where they converge, and the resulting identities between meromorphic functions of $\hbar$ hold by analytic continuation.

At $\rm OP$ the directions $\log\kappa_i$ are of two kinds. For $i\in\{1,l\}$ the pair of zeros $\kappa_i^{\pm1}$ of $\lambda$ coalesces with a double pole (\cref{rmk:lambda_OP}): the bracket in \eqref{eq:period_derivative} vanishes identically at $\rm OP$, and the derivative is carried by a boundary layer around the collision point, which we compute as a limit along $\rho$ in \cref{subsubsec:coalescing}. For $2\leq i\leq l-1$ the integrand is regular at $\rm OP$, and we evaluate it in \cref{subsubsec:noncoalescing} by expanding the bracket in the Fourier modes on the roots of $\mu^{2n}=-1$. The Jacobian matrix of the flat logarithms, computed in the same modes, selects a single pair of them in each twisted sector (\cref{prop:telescoping}), and the closed form of $\cB$ is assembled in \cref{subsubsec:B_assembled}.

\subsubsection{Coalescing directions}
\label{subsubsec:coalescing}
\begin{lem}
\label{lem:collision}
Let $\beta(\hbar) \coloneqq -\hbar\,2^\hbar\,\Beta\big(\hbar,\tfrac12-\hbar\big)$. For $i\in\{1,l\}$,
\beq
\label{eq:collision}
\frac{\partial \Pi_j}{\partial\log\kappa_i}\Big|_{\rm OP} = \beta(\hbar)
\begin{cases}
 \cos(\pi\hbar) & (j,i)=(l,l)\,, \\
 -\ri\sin(\pi\hbar) & (j,i)=(l+1,l)\,, \\
 \re^{-\pi\ri n\hbar}\cos(\pi\hbar) & (j,i)=(1,1)\,, \\
 0 & \mathrm{otherwise}\,,
 \end{cases}
\eeq
the limits $\kappa_l\to 1$ and $\kappa_1\to-1$ being taken along $\rho$.
\end{lem}
\begin{proof}
Write $\kappa_l = \re^{h}$, with the other parameters at their values at $\rm{OP}$. Then $\lambda = \lambda_0(\mu)\,R_h(\mu)$ with $R_h(\mu) = (\mu-\re^{h})(\mu-\re^{-h})/(\mu-1)^2$, and by \eqref{eq:period_derivative}
\begin{equation*}
\frac{\partial\Pi_j}{\partial h} = \hbar\int_{\gamma_j}\lambda^{\hbar}\,\Big(-\frac{\re^{h}}{\mu-\re^{h}}+\frac{\re^{-h}}{\mu-\re^{-h}}\Big)\frac{\rd\mu}{\mu}\,.
\end{equation*}
As $h\to0$ the bracket tends to zero uniformly away from $\mu = 1$, so that only a neighbourhood of the collision contributes, and there $\lambda$ can be replaced by $\lambda_0(1)R_h(\mu)$. In the scaling variable $\mu = \re^{hu}$ one has $R_h\to 1-u^{-2}$ and the bracket tends to $-\frac{2}{h}\frac{1}{u^2-1}$, whence
\begin{equation*}
\frac{\partial\Pi_j}{\partial h}\Big|_{\rm OP} = -2\hbar\,\lambda_0(1)^{\hbar}\int_{u(\gamma_j)}(1-u^{-2})^{\hbar-1}\,\frac{\rd u}{u^2}\,.
\end{equation*}
For $\Pi_l$ the contour $\gamma_l = [0,\re^{-h}]$ corresponds to $u\in(-\infty,-1)$, and the substitution $v = u^{-2}$ gives $\frac12 \Beta(\frac12,\hbar) = \frac12\cos(\pi\hbar)\,\Beta(\hbar,\frac12-\hbar)$ by the reflection formula $\Gamma(\frac12-\hbar)\Gamma(\frac12+\hbar) = \pi/\cos(\pi\hbar)$. For $\Pi_{l+1}$ the contour $[0,1]$ corresponds to $u\in(-\infty,0)$, passing below $u=-1$ since $\arg\kappa_l<0$ along $\rho$. On $(-1,0)$ one has $(1-u^{-2})^{\hbar-1} = -\re^{\pi\ri\hbar}(u^{-2}-1)^{\hbar-1}$, which contributes in addition $-\re^{\pi\ri\hbar}\cdot\frac12\Beta(\hbar,\frac12-\hbar)$, and $\cos(\pi\hbar)-\re^{\pi\ri\hbar} = -\ri\sin(\pi\hbar)$. Since $\lambda_0(1)^{\hbar} = 2^{\hbar}$, this gives the first two lines of \eqref{eq:collision}. The statement for $\kappa_1\to-1$ is identical, with $\gamma_1 = [0,-1]$ ending at the collision point and $\lambda_0(-1)^{\hbar} = 2^{\hbar}\re^{-\pi\ri n\hbar}$ on the sheet of $\Pi_1$. Finally, for contours not passing through the collision point the bracket tends to zero uniformly along the contour, and the corresponding derivatives vanish.
\end{proof}

\subsubsection{The non-coalescing directions}
\label{subsubsec:noncoalescing}
For $2\leq i\leq l-1$, $\kappa_i|_{\rm OP}$ is a simple zero of $\lambda_0$ and \eqref{eq:period_derivative} at $\rm OP$ reads
\beq
\label{eq:noncoalescing_derivative}
\frac{\partial \Pi_j}{\partial\log\kappa_i}\Big|_{\rm OP} = \hbar\int_{\gamma_j}\lambda_0^{\hbar}\Big(-\frac{\kappa_i}{\mu-\kappa_i}+\frac{\kappa_i^{-1}}{\mu-\kappa_i^{-1}}\Big)\frac{\rd\mu}{\mu}\,.
\eeq
The bracket has simple poles at two of the roots of $\mu^{2n}=-1$, and we expand it in the following basis.
\begin{lem}
\label{lem:modes}
For $0\leq s\leq 2n-1$ set
\begin{equation*}
G_s(\mu) \coloneqq \frac{2n\,\mu^{\,2n-1-s}}{\mu^{2n}+1}\,,\qquad
c_s(\kappa) \coloneqq \frac{\kappa^{-(s+1)}-\kappa^{\,s+1}}{2n}\,.
\end{equation*}
Then for every $\kappa$ with $\kappa^{2n}=-1$,
\beq
\label{eq:mode_expansion}
-\frac{\kappa}{\mu-\kappa}+\frac{\kappa^{-1}}{\mu-\kappa^{-1}} \;=\; \sum_{s=0}^{2n-1} c_s(\kappa)\,G_s(\mu)\,,
\eeq
with $c_{2n-1}\equiv 0$, $c_s(\pm1)=0$, $c_{2n-2-s}=c_s$ and $G_s(1/\mu)=G_{2n-2-s}(\mu)$.
\end{lem}
\begin{proof}
Partial fractions with respect to the roots of $\mu^{2n}+1$ give \[\frac{1}{\mu-\varrho}=\frac{1}{2n}\sum_{s}\varrho^{\,s}G_s(\mu)\] for every such root $\varrho$. The expansion \eqref{eq:mode_expansion} follows upon application to $\varrho=\kappa^{\pm1}$, and the remaining assertions are immediate from the definitions. 
\end{proof}

The mode integrals along the contours $\gamma_j$ are Beta functions.
\begin{lem}
\label{lem:beta_modes}
For $0\leq s\leq 2n-2$ put
\begin{equation*}
f_s\coloneqq \frac{2n-1-s}{2n}\in(0,1)\,, \qquad
B_s \coloneqq \Beta\big(f_s-\tfrac{\hbar}{2},\,1-f_s-\tfrac{\hbar}{2}\big)\,,
\end{equation*}
so that $B_{2n-2-s}=B_s$. Then, the following identities hold:
\begin{align}
\label{eq:beta_modes}
\int_0^{\infty}\lambda_0^{\hbar}\,G_s\,\frac{\rd\mu}{\mu} &= B_s =
\int_0^{1}\lambda_0^{\hbar}\,\big(G_s+G_{2n-2-s}\big)\,\frac{\rd\mu}{\mu}\,,\\
\label{eq:beta_modes_regular}
\int_{\gamma_j}\lambda_0^{\hbar}\,G_s\,\frac{\rd\mu}{\mu} &= -\big(\kappa_j^{\,n}\big)^{\hbar}\,\kappa_j^{\,s+1}\,\Beta\Big(f_s-\frac{\hbar}{2},\hbar\Big)\,,\qquad 2\leq j\leq l-1\,,\\
\label{eq:beta_modes_first}
\int_{\gamma_1}\lambda_0^{\hbar}\,G_s\,\frac{\rd\mu}{\mu} &= (-1)^{s+1}\,\re^{-\pi\ri n\hbar}\int_0^{1}\lambda_0^{\hbar}\,G_s\,\frac{\rd\mu}{\mu}\,.
\end{align}
\end{lem}
\begin{proof}
The proof is elementary by appropriate monomial substitutions
in each case; we omit the details.
\end{proof}


\begin{prop}
\label{prop:B_general_closed}
Let $2\leq i\leq l-1$. Then
\begin{align}
\label{eq:single_beta}
\frac{\partial\Pi_j}{\partial\log\kappa_i}\Big|_{\rm OP} &= -\hbar\,\big(\kappa_j^{\,n}\big)^{\hbar}\sum_{s=0}^{2n-2}c_s(\kappa_i)\,\kappa_j^{\,s+1}\,\Beta\Big(f_s-\frac{\hbar}{2},\hbar\Big)\,,\qquad 2\leq j\leq l-1\,,\\
\label{eq:B_colcontour}
\frac{\partial\Pi_l}{\partial\log\kappa_i}\Big|_{\rm OP} = \frac{\partial\Pi_{l+1}}{\partial\log\kappa_i}\Big|_{\rm OP} &= \frac{\hbar}{2}\sum_{s=0}^{2n-2}c_s(\kappa_i)\,B_s\,,\\
\label{eq:reflection}
\frac{\partial\Pi_1}{\partial\log\kappa_i}\Big|_{\rm OP} &= \frac{\hbar}{2}\,\re^{-\pi\ri n\hbar}\sum_{s=0}^{2n-2}c_s(-\kappa_i)\,B_s = -\,\re^{-\pi\ri n\hbar}\,\frac{\partial\Pi_l}{\partial\log\kappa_{l+1-i}}\Big|_{\rm OP}\,.
\end{align}
\end{prop}
\begin{proof}
Insert \eqref{eq:mode_expansion} into \eqref{eq:noncoalescing_derivative} and integrate term by term by \cref{lem:beta_modes}. This gives \eqref{eq:single_beta} directly, including the case $j=i$ as the endpoint pole is integrable for $\mathrm{Re}\,\hbar>0$. Pairing $s$ with $2n-2-s$, using $c_{2n-2-s}=c_s$, yields \eqref{eq:B_colcontour}.
The first expression in \eqref{eq:reflection} follows then from \eqref{eq:beta_modes_first}, as $(-1)^{s+1}c_s(\kappa)=c_s(-\kappa)$. Finally, the last equality in \eqref{eq:reflection} is $-\kappa_i=\kappa_{l+1-i}^{-1}$, by \eqref{eq: OP_DN}, together with $c_s(\kappa^{-1})=-c_s(\kappa)$.
\end{proof}

By \eqref{eq:B_structure} and \cref{lem:modes}, the directions $2\leq i\leq l-1$ enter $\cB$ only through the mode coefficients of the Jacobian of the change-of-variables $\log \kappa \to \log \kappa(x)$:
\beq
\label{eq:B_via_M}
\cB^{\rm reg}_{j,m}=\hbar\sum_{s=0}^{2n-2}M_{m,s}\,I_j(s)\,,\qquad
M_{m,s}:=-\sum_{i=2}^{l-1}\sum_{v=1}^{l}(w_i)_v\,D_{v,m}\,c_s(\kappa_i)\,,\qquad
I_j(s):=\int_{\gamma_j}\lambda_0^{\hbar}G_s\frac{\rd\mu}{\mu}\,,
\eeq
where $\cB^{\rm reg}$ omits the collision directions $i=1,l$ of \cref{lem:collision}. These do not enter the columns $1\leq m\leq n$, since $\sum_v(w_i)_v\chi_{\rho_v}(a^{m})=0$ for $i\in\{1,l\}$: a full period of the powers of $\re^{2\pi\ri m/(2n)}\neq1$ for $i=1$, and half the difference of two characters agreeing on $a$ for $i=l$.
\begin{prop}
\label{prop:telescoping}
With the sectors labelled as in \eqref{eq:class_enumeration}, for all $0\leq s\leq 2n-2$
\begin{equation*}
M_{l-1,s}=M_{l,s}=0\,,\qquad
M_{m,s}=\frac{(-1)^{m}}{2n}\Big(\delta_{s,\,m-1}+\delta_{s,\,2n-1-m}\Big)\quad(1\leq m\leq n)\,.
\end{equation*}
\end{prop}
\begin{proof}[Proof (sketch)]
Write \[S_i(g)\coloneqq\sum_v(w_i)_v\chi_{\rho_v}(g)\,,\] so that, by \eqref{eq:McKay_cohomology},
\[
M_{m,s} = -\frac{|(g_m)|}{|G|}\sqrt{\chi_Q(g_m)-2}\;\sum_{i=2}^{l-1}S_i(g_m)\,c_s(\kappa_i)\,.
\]
For $m\in\{l-1,l\}$ every $S_i(g_m)$ vanishes, the two-dimensional characters being zero at $b$ and $ab$, with moreover $\rho_{l-1}(b)=-\rho_l(b)$. For $1\leq m\leq n$, the sum $S_{l-j}(a^{m})$, $1\leq j\leq n$, telescopes to
\[
S_{l-j}(a^{m}) = (-1)^{m}\frac{\sin\Big(\frac{(2j-1)\pi m}{2n}\Big)}{\sin\Big(\frac{\pi m}{2n}\Big)}\,,
\]
whose denominator cancels against $\sqrt{\chi_Q(a^{m})-2} = 2\ri\sin\big(\tfrac{\pi m}{2n}\big)$. On the other hand, by \eqref{eq: OP_DN},
\[
c_s(\kappa_{l-j}) = \frac{\ri}{n}\,\sin\bigg(\frac{(2j-1)\pi(s+1)}{2n}\bigg)\,.
\]
Summing the product of the two sines over $j$ is a discrete orthogonality relation, which is non-zero only for $s+1\in\{m,2n-m\}$. The class sizes $|(a^{m})|$ produce the normalisation $(-1)^{m}/(2n)$.
\end{proof}

\subsubsection{The closed form of $\cB$}
\label{subsubsec:B_assembled}
The only missing bit is the column $m=0$ of $\cB$, given by the vector of values of the periods at $\rm OP$. Using \cref{lem:beta_modes}, we get the following
\begin{lem}
\label{lem:OP_values}
At $\rm OP$,
\begin{equation*}
\begin{gathered}
\Pi_j\big|_{\rm OP} = \frac{(\kappa_j^{\,n})^{\hbar}}{2n}\,\Beta\big(-\tfrac{\hbar}{2},1+\hbar\big)\,,\qquad 2\leq j\leq l-1\,,\\
\Pi_l\big|_{\rm OP} = \Pi_{l+1}\big|_{\rm OP} = \frac{1}{4n}\,\Beta\big(-\tfrac{\hbar}{2},-\tfrac{\hbar}{2}\big)\,,\qquad
\Pi_1\big|_{\rm OP} = \re^{-\pi\ri n\hbar}\,\Pi_{l+1}\big|_{\rm OP}\,.
\end{gathered}
\end{equation*}
\end{lem}

The remaining columns are obtained through a straightforward, if tedious, calculation of the product \eqref{eq:B_structure} of the Jacobian matrices of \cref{lem:collision,prop:B_general_closed} and of \cref{prop:telescoping}.
\begin{prop}
\label{prop:B_closed_form}
Let $\varepsilon=\rho_{l-1}(b)=-\rho_l(b)$ be as in \cref{sec:classical_McKay}, and $\beta(\hbar)$ as in \cref{lem:collision}. Then, for $1\leq m\leq n$,
\beq
\label{eq:B_single_beta}
\cB_{j,m} = \frac{\hbar\,\Beta\big(1-\tfrac{m}{2n}-\tfrac{\hbar}{2},\tfrac{m}{2n}-\tfrac{\hbar}{2}\big)}{2n\,(1+\delta_{m,n})}
\begin{cases}
-\dfrac{(-1)^{m}\,\big(\kappa_j^{\,n}\big)^{\hbar}}{\sin(\pi\hbar)}\Big[\kappa_j^{\,m}\,\sin\tfrac{\pi(m-n\hbar)}{2n}-\kappa_j^{-m}\,\sin\tfrac{\pi(m+n\hbar)}{2n}\Big] & 2\leq j\leq l-1\,,\\[10pt]
(-1)^{m} & j\in\{l,l+1\}\,,\\[6pt]
\re^{-\pi\ri n\hbar} & j=1\,,
\end{cases}
\eeq
and, for $m\in\{l-1,l\}$, with the upper sign for $m=l-1$,
\begin{equation*}
\cB_{j,m} = \frac{\beta(\hbar)}{2\sqrt2}
\begin{cases}
\pm\,\ri\,\varepsilon\cos(\pi\hbar) & j=l\,,\\[3pt]
\pm\,\varepsilon\sin(\pi\hbar) & j=l+1\,,\\[3pt]
\ri\,\re^{-\pi\ri n\hbar}\cos(\pi\hbar) & j=1\,,\\[3pt]
0 & 2\leq j\leq l-1\,.
\end{cases}
\end{equation*}
\end{prop}

\begin{rmk}
\label{rmk:thesis_correction}
For $l=3$, \cref{prop:B_closed_form} corrects the expression for the matrix $\cB$ given in \cite[Prop.~3.5.8, Eq.~(3.150)]{Ma:2026ths}. This correction was obtained after an AI-assisted review of our {\it Wolfram Language} code using interactive sessions in Claude Code (Fable 5.1).
\end{rmk}

\subsection{Computing $\cA$}
\label{subsec:LR_general}
We keep the notation of \cref{subsec:B_general} and the localised basis \eqref{eq:loc_basis_general}, and set $\tau_0=0$, i.e.\ $\kappa_{l+1}=1$: the common factor $\re^{\tau_0/z}=\kappa_{l+1}^{\hbar}$ of $J^{Y}$ and $\Pi$ is immaterial for $\cA$.

\subsubsection{Localisation and the asymptotics of $J^{Y}$}
The constraints on the weights in \eqref{eq:localisation_normal_weights} are solved by
\[
w_i^L = 2\nu\,(n+1-i)\,,\qquad w_i^R = -2\nu\,(n-1-i)\,,\qquad 1\leq i\leq n-2\,.
\]
Contracting \eqref{eq:localisation_normal_weights} with the inverse Cartan matrix of type $\mathrm{D}_l$,
the localisation of the basis \eqref{eq:phi_basis}
reads
\beq
\label{eq:loc_phi_general}
i^*(\phi_k) =
\begin{cases}
-2\nu\displaystyle\sum_{i=1}^{n-1}\mathfrak{n}_{ik}\,P_i-2\nu k\,\mathbf{1}_C+\delta_{k,n}\,\omega_C-2\nu k\,(P_{l-1}+P_l) & 1\leq k\leq n\,,\\[10pt]
-\nu n\Big(\displaystyle\sum_{i=1}^{n-1}P_i+\mathbf{1}_C+P_l\Big)-\nu l\,P_{l-1} & k=l-1\,,\\[10pt]
-\nu n\Big(\displaystyle\sum_{i=1}^{n-1}P_i+\mathbf{1}_C+P_{l-1}\Big)-\nu l\,P_{l} & k=l\,,
\end{cases}
\eeq
where
\[
\mathfrak{n}_{ik}\coloneqq
\begin{cases}
n & i\leq k\,,\\
k & i>k\,.
\end{cases}
\]

\begin{prop}
\label{prop:JZ_asymptotics_general}
The leading asymptotics at $\rm LR$ of the components \eqref{eq:JYcomps} of $J^{Y}$ are
\beq
\label{eq:J_asymp_general}
J^{Y}_j \sim z\,
\begin{cases}
(\kappa_1\cdots\kappa_{j-1})^{\hbar}\,\kappa_j^{(n+1-j)\hbar} & 1\leq j\leq n\,,\\[6pt]
\Big(\dfrac{\kappa_1\cdots\kappa_{l-1}}{\kappa_l}\Big)^{\hbar} & j=l-1\,,\\[10pt]
(\kappa_1\cdots\kappa_{l})^{\hbar} & j=l\,,
\end{cases}
\eeq
and
\[
J^{Y}_C \sim \tau_n\,(\kappa_1\cdots\kappa_{n})^{\hbar} = \log\Big(\frac{\kappa_{l-1}}{\kappa_n}\Big)\,(\kappa_1\cdots\kappa_{n})^{\hbar}\,.
\]
\end{prop}
\begin{proof}
By \eqref{eq:asymp_points}--\eqref{eq:asymp_curve},
\[
J^{Y}_j\sim z\,\exp\Big(\frac{L^j(\tau)}{z}\Big)\,,\qquad J^{Y}_C\sim L^C(\tau)\,\exp\Big(\frac{L^{l-2}(\tau)}{z}\Big)\,,
\]
and the $L$'s are read off from \eqref{eq:loc_phi_general}:
\[
L^j(\tau) = -2\nu\Big[\sum_{k<j}k\,\tau_k+n\sum_{k=j}^{n}\tau_k+\frac{n}{2}(\tau_{l-1}+\tau_l)\Big]\quad(j\leq n)\,,\qquad L^C(\tau) = \tau_n\,,
\]
and similarly for $j=l-1,l$. Converting to the $\kappa$'s by \eqref{eq:logkappa_weights} and \eqref{eq:vector_weights}, i.e.
\[
\log\kappa_r = -\sum_{k=r}^{n}\tau_k-\frac{1}{2}(\tau_{l-1}+\tau_l)\quad(r\leq n)\,,\qquad \log\kappa_{l-1} = -\frac{1}{2}(\tau_{l-1}+\tau_l)\,,\qquad \log\kappa_l = \frac{1}{2}(\tau_{l-1}-\tau_l)\,,
\]
gives \eqref{eq:J_asymp_general}.
\end{proof}

\subsubsection{The comparison matrix $\cA$ at large radius}
The leading large-radius asymptotics of the periods are obtained from the reduced Lauricella forms
\begin{align}
\label{eq:Lauricella_reduced_general}
\Pi_i &= \Beta(-n\hbar,1+\hbar)\,\,\kappa_i^{n\hbar}\,F_D^{(i-1)}\Big(-n\hbar;-\hbar,\dots,-\hbar;1-(n-1)\hbar\,\Big|\,\frac{\kappa_{i-1}}{\kappa_i},\dots,\frac{\kappa_1}{\kappa_i}\Big)\,,\qquad 1\leq i\leq l\,,\nonumber\\
\Pi_{l+1} &= \Beta(-n\hbar,1-2\hbar)\,\,F_D^{(l+1)}\Big(-n\hbar;2\hbar,-\hbar,\dots,-\hbar;1-l\hbar\,\Big|\,-1,\kappa_l,\dots,\kappa_1\Big)
\end{align}
of \cref{prop:hypergeometric_expression}, by applying the connection formulas of \cref{appendix: Lauricella_func} recursively to the largest variable: \cref{prop: acformula_generic} in the generic case, and \eqref{eq:acformula_log} in the resonant case $b_N=a$, which occurs for the last three rows.
The resulting monomials are precisely the exponentials of \cref{prop:JZ_asymptotics_general}, and $\cA^{-1}$ is read off from their coefficients (cf \cite[Lem.~5.8]{Brini:2019wts}). 

The outcome is the following closed form. Set
\begin{align}
\label{eq:LR_entries_general}
a_k &:= \frac{\mathsf{q}^{n}}{z}\,\Beta\big(-(n+1-k)\hbar,\,(n-k)\hbar\big)\,,\qquad 1\leq k\leq n-1\,,\nonumber\\
b_i &:= \frac{\mathsf{q}^{i-1}}{z}\,\Beta\big(1+\hbar,\,-(n+1-i)\hbar\big)\,,\qquad 1\leq i\leq n\,,\qquad\text{so that}\qquad b_n = -\frac{\pi\,\mathsf{q}^{n-1}}{z\sin(\pi\hbar)}\,,\nonumber\\
c &:= \frac{\mathsf{q}^n\big(\pi\ri - H_{\hbar}-H_{-1-\hbar}\big)}{z}\,,\qquad d := -\frac{2\,\mathsf{q}^n H_{-1-\hbar}}{z}\,,\qquad e := \frac{\mathsf{q}^{n+1}}{z}\,\Beta(1+\hbar,\hbar)\,,\nonumber\\
f &:= \frac{\mathsf{q}^{n}}{z}\,\Beta(\hbar,-2\hbar)\,,\qquad g := \frac{(1-\hbar)\,\mathsf{q}^{l}}{2\hbar\, z}\,\Beta(1-2\hbar,1+\hbar)
\,,
\end{align}
where
\beq
\label{eq:def_q}
\mathsf{q} \coloneqq \re^{2\pi\ri\nu/z} = \re^{\pi\ri\hbar}\,,\qquad H_{x} \coloneqq \psi(1+x) + \gamma = \frac{\Gamma'(1+x)}{\Gamma(1+x)} + \gamma\,,
\eeq
$H_x$ being the harmonic number function.

\begin{prop}
\label{lem:LR_general}
In the basis \eqref{eq:loc_basis_general} of $H(Y)$, ordering the rows of $\cA^{-1}$ as $(\Pi_1,\dots,\Pi_{l+1})$, the matrix $\cA^{-1}$ is lower triangular, with entries
\beq
\label{eq:Ainv_general}
(\cA^{-1})_{ij} =
\begin{cases}
a_j & 1\leq j<n,\ i>j\,, \\
b_i & i=j \leq n\,, \\
c & i=n+1, j=n\,, \\
d & i\in\{l,l+1\}, j=n\,, \\
-\mathsf{q}^n & n+1\leq i\leq l+1, j=n+1\,, \\
e & i=j=l\,, \\
f & i=l+1, j=l\,, \\
g & i=j=l+1\,, \\
0 & \text{otherwise}\,.
\end{cases}  
\eeq
\end{prop}

\begin{proof}
We write 
\[\Pi_j \simeq \sum_{\alpha\neq\omega_C} z\,(\cA^{-1})_{j,\alpha}\,\exp\Big(\frac{L^{\alpha}(\tau)}{z}\Big)+(\cA^{-1})_{j,\omega_C}\log\Big(\frac{\kappa_{l-1}}{\kappa_n}\Big)\exp\Big(\frac{L^{l-2}(\tau)}{z}\Big)\]
for the leading asymptotics, where $\alpha$ runs over the localised basis \eqref{eq:loc_basis_general} and $L^{\alpha}(\tau)$ is the coefficient of $\alpha$ in \eqref{eq:localisation_morphism}, so that $L^{P_k}=L^{k}$ and $L^{\mathbf 1_C}=L^{l-2}$. We identify the monomials with the fixed components using \cref{prop:JZ_asymptotics_general}:
\begin{align}
\label{eq:monomials_general}
P_k &\leftrightarrow (\kappa_1\cdots\kappa_{k-1})^{\hbar}\kappa_k^{(n+1-k)\hbar}\,,\qquad
\mathbf{1}_C \leftrightarrow (\kappa_1\cdots\kappa_n)^{\hbar}\,,\qquad
\omega_C \leftrightarrow \log\Big(\frac{\kappa_{l-1}}{\kappa_n}\Big)(\kappa_1\cdots\kappa_n)^{\hbar}\,,\nonumber\\
P_{l-1}&\leftrightarrow \Big(\frac{\kappa_1\cdots\kappa_{l-1}}{\kappa_l}\Big)^{\hbar}\,,\qquad
P_{l}\leftrightarrow (\kappa_1\cdots\kappa_{l})^{\hbar}\,.
\end{align}

Fix $1\leq i\leq l$ and put, for $0\leq K\leq \min(i-1,n-1)$,
\begin{equation*}
\mathsf{a}_K := -(n-K)\hbar\,,\qquad \mathsf{c}_K := 1-(n-1-K)\hbar\,,\qquad
F_i^{(K)} := F_D^{(i-1-K)}\Big(\mathsf{a}_K;-\hbar,\dots,-\hbar;\mathsf{c}_K\,\Big|\,\frac{\kappa_{i-1}}{\kappa_i},\dots,\frac{\kappa_{K+1}}{\kappa_i}\Big)\,,
\end{equation*}
so that \[\Pi_i = \Beta(\mathsf{a}_0,1+\hbar)\,\kappa_i^{n\hbar}F_i^{(0)}\] by \eqref{eq:Lauricella_reduced_general}. We claim that for $0\leq K\leq\min(i-1,n-1)$
\beq
\label{eq:induction_general}
\Pi_i \;\simeq\; \sum_{k=1}^{K} z\,a_k\,(\kappa_1\cdots\kappa_{k-1})^{\hbar}\kappa_k^{(n+1-k)\hbar}
\;+\;\mathsf{q}^{K}\,\Beta(1+\hbar,\mathsf{a}_K)\,\,(\kappa_1\cdots\kappa_K)^{\hbar}\,\kappa_i^{(n-K)\hbar}\,F_i^{(K)}\,.
\eeq
For $K=0$ this is the definition. The inductive step is \cref{prop: acformula_generic}, applied to the largest argument $\kappa_{K+1}/\kappa_i$ of $F_i^{(K)}$: its first term reproduces the summand $k=K+1$ of \eqref{eq:induction_general}, carrying the monomial attached to $P_{K+1}$ in \eqref{eq:monomials_general}, and its second term reproduces \eqref{eq:induction_general} for $K+1$. The Beta-function bookkeeping is elementary and we leave it to the reader; the connection formula itself is derived in \cref{appendix: Lauricella_func}.

Consider now $i\leq n$. Here \eqref{eq:induction_general} with $K = i-1$ has
\[F_i^{(i-1)} = F_D^{(0)} = 1\,,\]
and its last term is
\[\frac{\mathsf{q}^{i-1}\,\Gamma(1+\hbar)\Gamma(-(n+1-i)\hbar)}{\Gamma(1-(n-i)\hbar)}\,(\kappa_1\cdots\kappa_{i-1})^{\hbar}\kappa_i^{(n+1-i)\hbar} = z\,b_i\]
times the monomial of $P_i$ (or of $\mathbf{1}_C$ when $i=n$). By induction, this gives the first $n$ rows of \eqref{eq:Ainv_general}. 
A similar argument, using the resonant formula \eqref{eq:acformula_log} for the branch carrying $\log(-x_N)$, establishes the rows $\Pi_{l-1}$, $\Pi_l$ and $\Pi_{l+1}$; the logarithm is the monomial attached to $\omega_C$ in \eqref{eq:monomials_general}. Finally, all monomials produced by the recursion are among \eqref{eq:monomials_general}, so no other component of $J^{Y}$ is picked up.
\end{proof}

\begin{rmk}
\label{rmk:thesis_correction_A}
For $l=3$, \cref{lem:LR_general} corrects the entry $(\cA^{-1})_{\Pi_4,P_3}=g$ of \cite[Prop.~3.5.8, Eq.~(3.148)]{Ma:2026ths}, which was evaluated with the phase $(-x)^{-a}=\re^{\pi\ri a}x^{-a}$ in the connection formula in place of $(-x)^{-a}=\re^{-\pi\ri a}x^{-a}$. The convention $\arg x\in(-\pi,0)$ here is dictated by the choice of sector $-\pi<\arg\kappa_1<\dots<\arg\kappa_l<0$ of \cref{sec:branches}. This correction was obtained after an AI-assisted interactive review of our human-written {\it Wolfram Language} code using Claude Code (Fable 5.1). All other entries agree.
\end{rmk}

\section{The integral quantum McKay correspondence}
\label{sec:integral_McKay}
Recall the commutative diagram \eqref{diagram:induced_U_equivariant} proposed by Iritani. In this section, we prove that the analytic-continuation map $\mathbb{U}$ of \cref{sec:Umat_typeD} is induced through this diagram via a natural $K$-theory equivalence.

\subsection{The $K$-theory correspondence $\mathbb{U}_K$}
\label{subsec:UK}

Recall that the resolution $Y=G\text{-}\mathrm{Hilb}(\bbC^2)$ parametrises \emph{$G$-clusters}, i.e.\ $G$-invariant zero-dimensional subschemes $\xi\subset\bbC^2$ such that $H^0(\cO_\xi)$ is the regular representation of $G$. Let $\mathscr Z\subset Y\times\bbC^2$ be the universal subscheme, and let $\mathrm{pr}_1\colon Y\times\bbC^2\to Y$ and $\mathrm{pr}_2\colon Y\times\bbC^2\to\bbC^2$ be the projections, with $G$ acting trivially on $Y$.

\begin{thm}[{\cite{GonzalezSprinberg:1983gah, Kapranov:2000ksd}}]
\label{thm:derived_McKay}
The Fourier--Mukai functor
\[
\Phi:D^b(\cX)\simeq D^b_G(\bbC^2)\longrightarrow D^b(Y),
\qquad
\Phi(F)=\bigl(R\,\mathrm{pr}_{1*}(\cO_{\mathscr Z}\otimes^{\mathbf L}\mathrm{pr}_2^*F)\bigr)^G
\]
is an equivalence of derived categories. In particular, it induces an isomorphism of Grothendieck groups
\[
\Phi_K:K(\cX)\xrightarrow{\sim}K(Y).
\]
\end{thm}
Since the $\bbC^\times$-action is compatible with this construction, $\Phi_K$ admits an equivariant $K_{\bbC^\times}(\mathrm{pt})$-linear lift
\[
\Phi_{\bbC^\times ,K}:K_{\bbC^\times}(\cX)\xrightarrow{\sim}K_{\bbC^\times}(Y)\,.
\]
For an irreducible $G$-representation $\rho$, endowed with the trivial $\bbC^\times$-action, set
\[
\cO_\rho:=\cO_{\bbC^2}\otimes\rho,
\qquad
\mathcal R_\rho:=\bigl(\mathrm{pr}_{1*}\cO_{\mathscr Z}\otimes\rho^\vee\bigr)^G
\cong\Phi(\cO_{\rho^\vee}).
\]
Here $\rho^\vee$ denotes the dual $G$-representation.

\begin{defn}
\label{def:McKay_convention}
Using $K_{\bbC^\times}(\cX)\cong K_{\bbC^\times}(\mathrm{pt})\otimes_{\bbZ}R(G)$, with $R(G)$ the representation ring of $G$, let $d_G$ be the $K_{\bbC^\times}(\mathrm{pt})$-linear involution which dualises only the $G$-representation:
\[
d_G\bigl([\chi\otimes\cO_\rho]\bigr)
=[\chi\otimes\cO_{\rho^\vee}]
\]
for every $\bbC^\times$-character $\chi$. Define
\[
\mathbb U_K:=\Phi_{\bbC^\times,K}\circ d_G:
K_{\bbC^\times}(\cX)\xrightarrow{\sim}K_{\bbC^\times}(Y),
\qquad
\mathbb U_K([\cO_\rho])=[\mathcal R_\rho].
\]
\end{defn}

\begin{rmk}
\label{rmk:McKay_FM}
The map $\mathbb U_K$ is induced by an equivariant Fourier--Mukai equivalence. Indeed $d_G$ can be identified with $\Ad(s)$, where $s=\mathrm{diag}(1,-1)$ normalises $G$. Since conjugation by $s$ sends each conjugacy class to its inverse and $s$ commutes with the scalar action, pullback along the induced automorphism of $\cX$ lifts $d_G$ to a $\bbC^\times$-equivariant autoequivalence. Composing this lift with the $\bbC^\times$-equivariant lift of $\Phi$ gives the required equivalence.
\end{rmk}

As in \eqref{eq:def_q}, we write $\mathsf{q} = \re^{2\pi\ri\nu/z}$ for the image of the class of the weight-$\nu$ character of $\bbC^\times$, i.e.\ of the generator of $K_{\bbC^\times}(\mathrm{pt})=\bbZ[\mathsf{q},\mathsf{q}^{-1}]$.

\subsection{Induced cohomological correspondence from $\mathbb{U}_K$}
\label{subsec:K_general}
We now compute the cohomological correspondence $\mathbb{U}_K^{\sharp}$ induced by the geometric correspondence $\mathbb{U}_K$ of \cref{def:McKay_convention}, that is, the linear map $\HH_{\cX}\to\HH_{Y}$ determined by $\mathbb{U}_K^{\sharp}\circ\widehat{\Psi}_{\cX}=\widehat{\Psi}_{Y}\circ\mathbb{U}_K$. We first compute the change from the tautological basis to the line-bundle basis.

On $H(Y)$ and $H(\cX)$ we use the bases \eqref{eq:loc_basis_general} and $(\mathbf{1}_1,\dots,\mathbf{1}_l,\mathbf{1}_0)$, respectively. The $K$-theory bases are
\[
\begin{aligned}
K_{\bbC^\times}(Y)&:\quad (\cO(\phi_1),\dots,\cO(\phi_l),\cO_Y),\\
K_{\bbC^\times}(\cX)&:\quad (\cO_{\rho_1},\dots,\cO_{\rho_l},\cO_{\rho_0}),
\end{aligned}
\]
where $c_1^{\bbC^\times}(\cO(\phi_k)) = \phi_k$ and $\cO_\rho := \cO_{\bbC^2}\otimes\rho$.

The representations are labelled\footnote{For $\mathrm{D}_3$, the cyclic labelling $G = \bbZ_4 = \langle b\rangle$, $\rho_k(b) = \ri^{\,k}$ uses the opposite choice $\varepsilon = -\ri$. Its representations $(\rho_1,\rho_2,\rho_3)$ and sectors $(\mathbf{1}_1,\mathbf{1}_2,\mathbf{1}_3)$ correspond to $(\rho_l,\rho_1,\rho_{l-1})$ and $(\mathbf{1}_3,\mathbf{1}_1,\mathbf{1}_2)$ in our notation. This convention matches the one used, as $\mathrm D_3=\mathrm A_3$, in \cite[Thm.~4.1]{Brini:2019wts}.} as in \cref{sec:classical_McKay}. Recall from \eqref{eq:one_dimensional_representations} that $\varepsilon=\rho_{l-1}(b)=-\rho_l(b)$ equals $1$ for $n$ even and $\ri$ for $n$ odd. Changing its sign exchanges the two legs and the sectors $l-1,l$.

\subsubsection{Localisation and change of basis of $K_{\mathbb{C}^\times}(Y)$}
\label{subsubsec:taut_general}
The tautological bundles $\mathcal R_\rho$ give another natural basis of $K_{\bbC^\times}(Y)$. We compare the ordered bases
\beq
\label{eq:K_bases_general}
\begin{gathered}
\big([\cO(\phi_1)],\dots,[\cO(\phi_l)],[\cO_Y]\big),\\
\big([\mathcal{R}_{\rho_1}],\dots,[\mathcal{R}_{\rho_l}],[\mathcal{R}_{\rho_0}]\big),
\end{gathered}
\eeq
where $\mathcal{R}_{\rho_0} = \cO_Y$.

\begin{prop}
\label{prop:taut_general}
In $K_{\bbC^\times}(Y)$,
\begin{align}
\label{eq:taut_general}
[\mathcal{R}_{\rho_v}] &= [\cO(\phi_v)]\quad (v\in\{1,l-1,l\})\,,\qquad [\mathcal{R}_{\rho_0}] = [\cO_Y]\,,\nonumber\\
[\mathcal{R}_{\rho_v}] &= \mathsf{q}^{v-1}\,[\cO(\phi_v)]+\mathsf{q}^{-(v-1)}\,[\cO_Y]\quad (2\leq v\leq n)\,.
\end{align}
Equivalently, in the ordered bases \eqref{eq:K_bases_general},
\begin{equation*}
\big([\mathcal{R}_{\rho_1}],\dots,[\mathcal{R}_{\rho_l}],[\mathcal{R}_{\rho_0}]\big)
= \big([\cO(\phi_1)],\dots,[\cO(\phi_l)],[\cO_Y]\big)\,\mathsf{M}_l\,,
\end{equation*}
where, writing $\mathsf{E}_{ij}$ for the elementary matrix,
\begin{equation*}
\mathsf{M}_l := \mathrm{diag}\big(1,\mathsf{q},\mathsf{q}^{2},\dots,\mathsf{q}^{n-1},1,1,1\big) + \sum_{v=2}^{n}\mathsf{q}^{-(v-1)}\,\mathsf{E}_{l+1,v}\,.
\end{equation*}
\end{prop}

To prove \cref{prop:taut_general}, we first compute the images of the tautological classes under the $K$-theoretic analogue
\[
i^*\colon K_{\bbC^\times}(Y)\longrightarrow K_{\bbC^\times}(Y^{\bbC^\times})
\]
of the localisation map of \cref{sec:Umat_typeD}, and compare them with the images of the line-bundle combinations in \eqref{eq:taut_general}. We use the following cluster data, with the one-dimensional labels of \eqref{eq:one_dimensional_representations}. We assume $n\geq2$, and let $R=\bbC[x,y]$, $\mathfrak m=(x,y)$ and $f_\pm=x^n\pm\varepsilon y^n$.

\begin{lem}
\label{lem:fixed_clusters}
The $\bbC^\times$-fixed clusters have ideals
\[
\begin{aligned}
I_{p_1}&=(xy,\;x^{2n}+y^{2n}),\\
I_{p_i}&=(x^iy,\;xy^i,\;x^2y^2,\;x^{2n+2-i},\;y^{2n+2-i})\qquad(2\leq i\leq n),\\
I_{p_{l-1}}&=(x^2y^2,\;f_+),\qquad I_{p_l}=(x^2y^2,\;f_-),\\
I_t&=(x^2y^2,\;x^{n+1}-txy^n,\;y^{n+1}-(-1)^ntx^ny)+\mathfrak m^{n+2}\qquad(t\in\bbC).
\end{aligned}
\]
Here $p_i=E_{i-1}\cap E_i$ for $2\leq i\leq n-1$, while $p_n\coloneqq E_{n-1}\cap C$ is a point of the fixed curve, and $t$ is an affine coordinate on $C$ with $I_\infty=I_{p_n}$. 
\end{lem}
\begin{proof}
This is {\cite[Thm.~13.6 and Sec.~13.7]{Ito:1999kyz}}. From \cite[Table~8]{Ito:1999kyz}, 
polynomial representatives for bases of the irreducible copies in $R/I_{p_i}$, $i\in\{1,\dots,n,l-1,l\}$, are given in \cref{tab:fixed_clusters}.
\end{proof}

\begin{table}[h]
\[
\begin{array}{c|lcc}
\text{Representation}&\text{Representatives}&\text{Degree}&\text{Condition}\\ \hline
\rho_0&1&0&\text{all }i\\ \hline
\rho_1&x^{2n}-y^{2n}&2n&i=1\\
&xy&2&i\neq1\\ \hline
\rho_{l-1}&f_+&n&i\neq l-1\\
&xyf_-&n+2&i=l-1\\ \hline
\rho_l&f_-&n&i\neq l\\
&xyf_+&n+2&i=l\\ \hline
\rho_v&x^{v-1},\;y^{v-1}&v-1&\text{all }i\\
&x^{2n+1-v},\;y^{2n+1-v}&2n+1-v&i\leq v\\
&x^vy,\;xy^v&v+1&i>v
\end{array}
\]
\caption{Representatives and their degrees for the bases of $R/I_{p_i}$.}
\label{tab:fixed_clusters}
\end{table}

We are now ready to compute the fibre characters.

\begin{lem}
\label{lem:taut_localisation}
At the isolated fixed points $p_i$, where $i\in\{1,\dots,n-1,l-1,l\}$, the fibre characters are
\begin{align*}
[\mathcal R_{\rho_0}|_{p_i}] &= 1,\\
[\mathcal R_{\rho_1}|_{p_i}] &=
\begin{cases}
\mathsf q^{-2n} & i=1,\\
\mathsf q^{-2} & i\neq1,
\end{cases}\\
[\mathcal R_{\rho_v}|_{p_i}] &= \mathsf q^{-n-2\delta_{iv}}
\qquad (v\in\{l-1,l\}),\\
[\mathcal R_{\rho_v}|_{p_i}] &= \mathsf q^{1-v}+
\begin{cases}
\mathsf q^{-(2n+1-v)} & 1\leq i\leq v,\\
\mathsf q^{-(v+1)} & \text{otherwise},
\end{cases}
\qquad (2\leq v\leq n).
\end{align*}
On the fixed curve $C=E_n$, writing $\chi_w$ for the character of $\bbC^\times$ of weight $w$, the rank-one bundles restrict as
\begin{align*}
\mathcal R_{\rho_0}|_C &\cong\cO_C,\\
\mathcal R_{\rho_1}|_C &\cong\cO_C\otimes\chi_{-2\nu},\\
\mathcal R_{\rho_v}|_C &\cong\cO_C\otimes\chi_{-n\nu}
\qquad (v\in\{l-1,l\}),
\end{align*}
and the rank-two bundles restrict as
\beq
\label{eq:taut_on_C_general}
\mathcal R_{\rho_v}|_C\cong
\cO_C\otimes\chi_{-(v-1)\nu}\oplus
\cO_C(\delta_{vn})\otimes\chi_{-(v+1)\nu},
\qquad 2\leq v\leq n.
\eeq
In particular, the only summand of non-zero degree is the summand of weight $-(n+1)\nu$ of $\mathcal R_{\rho_n}|_C$, which is $\cO_C(1)$.
\end{lem}
\begin{proof}
Since $\mathscr Z\to Y$ is finite flat, base change gives, for any cluster $\xi=\mathrm{Spec}(R/I_\xi)$,
\[
\mathcal R_\rho|_\xi\cong\bigl((R/I_\xi)\otimes\rho^\vee\bigr)^G
\cong\mathrm{Hom}_G(\rho,R/I_\xi).
\]

If $\xi$ is $\bbC^\times$-fixed, scalar invariance makes $I_\xi$ homogeneous, so $R/I_\xi$ inherits the polynomial grading. Since $x,y$ have torus weight $-\nu$ and $\rho$ has trivial $\bbC^\times$-action, each copy of $\rho$ in degree $d$ contributes $\mathsf q^{-d}$ to the fibre character. The characters at the isolated fixed points therefore follow from the degrees in \cref{lem:fixed_clusters,tab:fixed_clusters}.

For the restrictions to $C$, consider the family $I_t$ in \cref{lem:fixed_clusters}. For every $t\in\bbP^1$, we have $(R/I_t)_d=R_d/(x^2y^2R_{d-4})$ for $0\leq d\leq n$, with $R_k=0$ for $k<0$, while the quotient vanishes in degrees above $n+1$.
The graded pieces in degrees at most $n$ are thus constant. They contain the four one-dimensional representations, both copies of $\rho_v$ for $2\leq v<n$, and the degree-$(n-1)$ copy of $\rho_n$. Each of these copies gives a trivial line bundle in the corresponding tautological restriction.

It remains to determine the degree-$(n+1)$ summand of $\mathcal R_{\rho_n}|_C$. Modulo $x^2y^2$, the degree-$(n+1)$ part of $R$ is $V_1\oplus V_2$, where
\[
V_1=\langle x^{n+1},y^{n+1}\rangle\cong\rho_n,
\qquad
V_2=\langle xy^n,x^ny\rangle\cong\rho_n.
\]
The quotient $(R/I_t)_{n+1}$ is obtained by imposing the two degree-$(n+1)$ relations defining $I_t$. The $G$-equivariant isomorphism $V_1\to V_2$ given by $x^{n+1}\mapsto xy^n$ and $y^{n+1}\mapsto(-1)^nx^ny$ gives $V_1\oplus V_2\cong\rho_n\otimes\bbC^2$, where $G$ acts trivially on $\bbC^2$. Let $e_1,e_2$ be the basis of $\rho_n$ corresponding to $x^{n+1},y^{n+1}$ in $V_1$. In homogeneous coordinates $[s:u]$ on $C\cong\bbP^1$, with $t=u/s$, the relations correspond to
\[
\begin{aligned}
sx^{n+1}-uxy^n&\longleftrightarrow e_1\otimes(s,-u),\\
sy^{n+1}-(-1)^nux^ny&\longleftrightarrow e_2\otimes(s,-u).
\end{aligned}
\]
They therefore span
\[
\rho_n\otimes\operatorname{span}_{\bbC}\{(s,-u)\}\subset\rho_n\otimes\bbC^2.
\]
The lines $\operatorname{span}_{\bbC}\{(s,-u)\}$ form the subbundle $\cO_C(-1)\subset\cO_C^{\oplus2}$. Applying $\mathrm{Hom}_G(\rho_n,-)$ to the degree-$(n+1)$ quotient gives the required line bundle $L$, which fits into
\[
0\longrightarrow\cO_C(-1)\longrightarrow\cO_C^{\oplus2}\longrightarrow L\longrightarrow0,
\qquad L\cong\cO_C(1).
\]
Restoring the torus character $\chi_{-d\nu}$ in degree $d$ gives \eqref{eq:taut_on_C_general} and the rank-one restrictions.
\end{proof}

\begin{proof}[Proof of \cref{prop:taut_general}]
By \cref{thm:derived_McKay}, $K_{\bbC^\times}(Y)$ is free over $K_{\bbC^\times}(\mathrm{pt})$, so the restriction map $i^*$ is injective by equivariant localisation. The restrictions of the line-bundle combinations in \eqref{eq:taut_general}, obtained from \eqref{eq:loc_phi_general}, agree with those in \cref{lem:taut_localisation}. This proves the equalities in $K_{\bbC^\times}(Y)$.
\end{proof}

\subsubsection{The induced cohomological correspondence}
\label{subsubsec:induced_cohomological_matrix}
In the following, we write the ordered cohomology bases as
\begin{align*}
\mathbf e_Y &:= (P_1,\dots,P_{n-1},\mathbf1_C,\omega_C,P_{l-1},P_l),\\
\mathbf e_{\cX} &:= (\mathbf1_1,\dots,\mathbf1_l,\mathbf1_0),
\end{align*}
and the ordered $K$-theory bases
\begin{align*}
\mathcal E_Y &:= ([\cO(\phi_1)],\dots,[\cO(\phi_l)],[\cO_Y]),\\
\mathcal E_{\cX} &:= ([\cO_{\rho_1}],\dots,[\cO_{\rho_l}],[\cO_{\rho_0}]).
\end{align*}
Writing the framings \eqref{eq:integral_structure_framing} as $\widehat\Psi(V)=\big(z^{-\frac{\deg}{2}}\widehat\Gamma\big)\cup\big(\frac{2\pi\ri}{z}\big)^{\frac{\deg_0}{2}}\iota^*\mathrm{ch}_{\bbC^\times}(V)$ up to the factor $(z/2\pi)^{\dim\cZ/2}=z/(2\pi)$, common to $Y$ and $\cX$, we define the matrices $\Gamma_{\cX}$, $\Gamma_Y$, $\mathrm{ch}_{\cX}$ and $\mathrm{ch}_Y$ by
\begin{align*}
\big(z^{-\frac{\deg}{2}}\widehat\Gamma(T_{\cX})\big)\cup\mathbf e_{\cX}
&=\mathbf e_{\cX}\Gamma_{\cX},\\
\big(z^{-\frac{\deg}{2}}\widehat\Gamma(T_Y)\big)\cup\mathbf e_Y
&=\mathbf e_Y\Gamma_Y,\\
\big(\tfrac{2\pi\ri}{z}\big)^{\frac{\deg_0}{2}}\iota^*\mathrm{ch}(\mathcal E_{\cX})
&=\mathbf e_{\cX}\mathrm{ch}_{\cX},\\
\big(\tfrac{2\pi\ri}{z}\big)^{\frac{\deg_0}{2}}\mathrm{ch}(\mathcal E_Y)
&=\mathbf e_Y\mathrm{ch}_Y.
\end{align*}
Here the Chern character and multiplication by the Gamma class are applied componentwise to the ordered bases, $\cup$ is the componentwise product on the inertia stack, $\deg$ is the age-shifted degree and $\deg_0$ the ordinary one (both with $\deg\nu=2$), and $\iota$ is the involution of the inertia stack. In these bases the matrix of $\mathbb{U}_K^{\sharp}$, denoted by the same symbol as agreed in \cref{sec:Umat_typeD}, is determined by
\[
\mathbb{U}_K^{\sharp}(\mathbf e_{\cX})=\mathbf e_Y\,\mathbb{U}_K^{\sharp}.
\]

\begin{prop}
\label{prop:induced_cohomological_matrix}
With $\mathsf M_l$ as in \cref{prop:taut_general},
\beq
\label{eq:MtoU}
\mathbb{U}_K^{\sharp}
=\Gamma_Y\,\mathrm{ch}_Y\,\mathsf M_l
\,\mathrm{ch}_{\cX}^{-1}\,\Gamma_{\cX}^{-1}.
\eeq
\end{prop}
\begin{proof}
This is shown by composing the framings in the bases $\mathcal E_{\cX}$, $\mathcal E_Y$ and $\mathbf e_{\cX}$, $\mathbf e_Y$, using \cref{prop:taut_general} for $\mathbb{U}_K$.
\end{proof}
\begin{prop}
\label{prop:chern_character_matrices}
In the bases and normalisation of \cref{prop:induced_cohomological_matrix},
\beq
\label{eq:chX_general}
(\mathrm{ch}_{\cX})_{m,\rho}=\chi_\rho(g_m^{-1}),
\eeq
where the rows are ordered by $m=1,\dots,l,0$ and the columns by $\rho_1,\dots,\rho_l,\rho_0$, with $g_0=1$, $g_r=a^r$ for $1\leq r\leq n$, $g_{l-1}=b$ and $g_l=ab$. In particular, the last row is $(\dim\rho)_\rho$.

For $\mathrm{ch}_Y$, the following table lists the components of each bundle column in the basis $\mathbf e_Y$:
\beq
\label{eq:chY_general}
\begin{array}{c|ccccc}
 & P_i\ (i\leq n-1) & \mathbf{1}_C & \omega_C & P_{l-1} & P_l\\ \hline
\cO(\phi_k),\ k\leq n & \mathsf{q}^{-2\mathfrak{n}_{ik}} & \mathsf{q}^{-2k} & \frac{2\pi\ri}{z}\,\delta_{k,n}\,\mathsf{q}^{-2n} & \mathsf{q}^{-2k} & \mathsf{q}^{-2k}\\[2pt]
\cO(\phi_{l-1}) & \mathsf{q}^{-n} & \mathsf{q}^{-n} & 0 & \mathsf{q}^{-l} & \mathsf{q}^{-n}\\[2pt]
\cO(\phi_{l}) & \mathsf{q}^{-n} & \mathsf{q}^{-n} & 0 & \mathsf{q}^{-n} & \mathsf{q}^{-l}\\[2pt]
\cO_Y & 1 & 1 & 0 & 1 & 1
\end{array}\,,
\eeq
with $\mathfrak n_{ik}$ as in \eqref{eq:loc_phi_general}.
\end{prop}
\begin{proof}
On the sector indexed by $g_m$, the orbifold Chern character of $\cO_\rho$ is $\chi_\rho(g_m)$. Applying $\iota^*$ replaces $g_m$ by $g_m^{-1}$, giving \eqref{eq:chX_general}.

On $Y$, the scaled Chern character of $\cO(\phi_k)$ restricts to $\exp(2\pi\ri\,i^*\phi_k/z)$. Substituting \eqref{eq:loc_phi_general}, using $\exp(-4\pi\ri\nu j/z)=\mathsf q^{-2j}$ and $\omega_C^2=0$, gives \eqref{eq:chY_general}.
\end{proof}

\begin{prop}
\label{prop:gamma_class_matrices}
In the bases and normalisation of \cref{prop:induced_cohomological_matrix}, the Gamma-class multiplication matrices are
\beq
\label{eq:GammaX_general}
\Gamma_{\cX} = \mathrm{diag}\Big(\Big\{\frac{\Gamma(1-\frac{r}{2n}+\frac{\nu}{z})\Gamma(\frac{r}{2n}+\frac{\nu}{z})}{z}\Big\}_{r=1}^{n},\;\frac{\Gamma(\frac14+\frac{\nu}{z})\Gamma(\frac34+\frac{\nu}{z})}{z},\;\frac{\Gamma(\frac14+\frac{\nu}{z})\Gamma(\frac34+\frac{\nu}{z})}{z},\;\Gamma\Big(1+\frac{\nu}{z}\Big)^2\Big)\,,
\eeq
and
\beq
\label{eq:GammaY_general}
\begin{aligned}
\Gamma_Y = \mathrm{diag}\Big(&
\big\{\Gamma\big(1+\tfrac{2(n+1-i)\nu}{z}\big)
\Gamma\big(1-\tfrac{2(n-i)\nu}{z}\big)\big\}_{i=1}^{n-1}, \\ & \Gamma\big(1+\tfrac{2\nu}{z}\big)
\begin{pmatrix}1 & 0\\ -\frac{2}{z}H_{2\nu/z} & 1\end{pmatrix},\; \Gamma\big(1+\tfrac{4\nu}{z}\big)\Gamma\big(1-\tfrac{2\nu}{z}\big),\;
\Gamma\big(1+\tfrac{4\nu}{z}\big)\Gamma\big(1-\tfrac{2\nu}{z}\big)\Big)\,.
\end{aligned}
\eeq
Here $H$ is the harmonic-number function of \eqref{eq:def_q}, and the $2\times2$ block acts on $(\mathbf1_C,\omega_C)$.
\end{prop}
\begin{proof}
On the orbifold side, $g_r=a^r$ acts with eigenvalue exponents $r/(2n)$ and $1-r/(2n)$ for $1\leq r\leq n$, while $b$ and $ab$ act with exponents $1/4$ and $3/4$. The torus weights at the origin are $(\nu,\nu)$. Substitution in the Gamma class, with the age-shifted scaling of \cref{prop:induced_cohomological_matrix}, gives \eqref{eq:GammaX_general}; the factor $z^{-1}$ occurs precisely on the non-trivial sectors.

On $Y$, the tangent weights at $p_i$ are $2\nu(n+1-i)$ and $-2\nu(n-i)$, and those at either leg point are $4\nu$ and $-2\nu$. These give the scalar blocks of \eqref{eq:GammaY_general}. Along $C$, the tangent and normal Chern roots are $2\omega_C$ and $2\nu-2\omega_C$. Since $\omega_C^2=0$, the scaled Gamma class restricts to
\[
\Gamma\!\left(1+\frac{2\omega_C}{z}\right)
\Gamma\!\left(1+\frac{2\nu-2\omega_C}{z}\right)
=\Gamma\!\left(1+\frac{2\nu}{z}\right)
\left(1-\frac{2H_{2\nu/z}}{z}\omega_C\right),
\]
using $\psi(1)=-\gamma$. Multiplication by this class in the basis $(\mathbf1_C,\omega_C)$ gives the stated curve block.
\end{proof}

\subsection{Identification of the analytic-continuation map}
\label{subsec:analytic_continuation_identification}
We now show that the analytic-continuation map $\mathbb{U}$ of \cref{sec:Umat_typeD} coincides with $\mathbb{U}_K^{\sharp}$, i.e.\ that the diagram \eqref{diagram:induced_U_equivariant} commutes. Since $\mathbb{U}=\cA\cB$ by definition (see after \cref{prop:period_comparison_U}), by \eqref{eq:MtoU} this is equivalent to
\beq
\label{eq:master_identity}
\cT_{\rm OP} = \cT_{\rm LR}\,,\qquad \cT_{\rm OP} := \cB\,\Gamma_{\cX}\,\mathrm{ch}_{\cX}\,,\qquad \cT_{\rm LR} := \cA^{-1}\,\Gamma_Y\,\mathrm{ch}_Y\,\mathsf{M}_l\,,
\eeq
an identity between two $(l+1)\times(l+1)$ matrices whose rows are indexed by the periods $\Pi_1,\dots,\Pi_{l+1}$ and whose columns are indexed by $(\rho_1,\dots,\rho_l,\rho_0)$. The $j$-th row of $\cT_{\rm OP}$ pairs the orbifold-point data of $\Pi_j$ with the framed classes of $\cO_{\bbC^2}\otimes\rho$, and the $j$-th row of $\cT_{\rm LR}$ pairs the large-radius data of $\Pi_j$ with the framed classes of $\mathcal{R}_\rho$. The identity \eqref{eq:master_identity} states that these agree. We prove \eqref{eq:master_identity} by computing both sides in closed form, using \cref{lem:LR_general} for $\cA^{-1}$ and \cref{prop:B_closed_form,lem:OP_values} for $\cB$. Throughout,
\begin{equation*}
\mathsf{G} := \frac{\pi\,\Gamma(1+\hbar)}{z}\,,\qquad D_x := \frac{1}{\mathsf{q}^{2x}-1}\,,\qquad \langle k\rangle\in\{0,\dots,2n-1\}\ \text{the residue of}\ k\ \text{modulo}\ 2n\,,
\end{equation*}
and, for each $\rho$, $\mathsf{C}(\rho)$ denotes the multiset of exponents such that $\chi_\rho(a^{r}) = \sum_{c\in\mathsf{C}(\rho)}\zeta^{cr}$, with $\zeta = \re^{2\pi\ri/(2n)}$ the root of unity of \cref{sec:classical_McKay}: $\mathsf{C}(\rho_0) = \mathsf{C}(\rho_1) = \{0\}$, $\mathsf{C}(\rho_v) = \{v-1,\,1-v\}$ for $2\leq v\leq n$, and $\mathsf{C}(\rho_{l-1}) = \mathsf{C}(\rho_l) = \{n\}$; note that $(-1)^{r} = \zeta^{nr}$. We shall use the two discrete Fourier transforms
\begin{align}
\label{eq:DFT1}
\sum_{r=0}^{2n-1}\frac{\zeta^{kr}}{1-\zeta^{r}\mathsf{q}} &= \frac{2n\,\mathsf{q}^{\langle -k\rangle}}{1-\mathsf{q}^{2n}}\,,\\
\label{eq:DFT2}
\sum_{r=0}^{2n-1}\frac{\zeta^{kr}}{(1-\zeta^{r}\mathsf{q})(1-\zeta^{-r}\mathsf{q})} &= \frac{2n\big(\mathsf{q}^{\langle k\rangle}+\mathsf{q}^{2n-\langle k\rangle}\big)}{(1-\mathsf{q}^{2})(1-\mathsf{q}^{2n})}\,,
\end{align}
valid as identities of rational functions of $\mathsf{q}$: for $|\mathsf{q}|<1$ expand the left-hand sides in geometric series, $\sum_{\ell\geq0}\zeta^{r\ell}\mathsf{q}^{\ell}$ and $\sum_{\ell,\ell'\geq0}\zeta^{r(\ell-\ell')}\mathsf{q}^{\ell+\ell'}$, and use $\sum_{r}\zeta^{r(k+\ell)} = 2n\,\delta_{\langle\ell\rangle,\langle -k\rangle}$; in the second case the pairs $(\ell,\ell')$ with $\ell'-\ell\equiv k$ are $(\ell,\ell+\langle k\rangle+2nt)$ and $(\ell'+2n-\langle k\rangle+2nt,\ell')$, $t\geq0$.

\begin{lem}
\label{lem:T_OP}
Let $l\geq 4$. For $2\leq j\leq l-1$ put $\mathsf{p} := l-1-j\in\{0,\dots,n-1\}$. Then, for every $\rho$,
\begin{align*}
(\cT_{\rm OP})_{j,\rho} &= \frac{\ri\,\mathsf{G}\,\mathsf{q}^{-\mathsf{p}}}{1-\mathsf{q}^{2n}}\sum_{c\in\mathsf{C}(\rho)}\Big(\mathsf{q}^{\langle \mathsf{p}-n-c\rangle}+\mathsf{q}^{\langle \mathsf{p}+n+c\rangle}\Big)\,,\qquad 2\leq j\leq l-1\,,\\
(\cT_{\rm OP})_{j,\rho} &= \frac{\ri\,\mathsf{G}}{1-\mathsf{q}^{2n}}\sum_{c\in\mathsf{C}(\rho)}\Big(\mathsf{q}^{\langle n+c\rangle}+\mathsf{q}^{2n-\langle n+c\rangle}\Big)+\ri\,\mathsf{G}\,\sigma_j(\rho)\,,\qquad j\in\{l,l+1\}\,,\\
(\cT_{\rm OP})_{1,\rho} &= \frac{\ri\,\mathsf{G}\,\mathsf{q}^{-n}}{1-\mathsf{q}^{2n}}\sum_{c\in\mathsf{C}(\rho)}\Big(\mathsf{q}^{\langle c\rangle}+\mathsf{q}^{2n-\langle c\rangle}\Big)+\ri\,\mathsf{G}\,\mathsf{q}^{-n}\,\sigma_1(\rho)\,,
\end{align*}
where $\sigma_l(\rho_{l-1}) = -\sigma_l(\rho_l) = -1$, $\sigma_{l+1}(\rho_{l-1}) = -\sigma_{l+1}(\rho_l) = (\mathsf{q}^{2}-1)/(\mathsf{q}^{2}+1)$, $\sigma_1(\rho_0) = -\sigma_1(\rho_1) = -1$, and $\sigma_j(\rho) = 0$ in all other cases.
\end{lem}
\begin{proof}
Write \[(\cT_{\rm OP})_{j,\rho} = \sum_{m=0}^{l}\cB_{j,m}(\Gamma_{\cX})_{m}\,\chi_\rho(g_m^{-1})\,,\] with $\cB_{j,m}$ given by \cref{prop:B_closed_form} for $1\leq m\leq l$ and by \cref{lem:OP_values} for $m=0$. For the rows $2\leq j\leq l-1$, multiplying the Beta factor \eqref{eq:B_single_beta} by the Gamma-class entry \eqref{eq:GammaX_general} and applying Euler's reflection formula collapses every $\Gamma$-function to a trigonometric one. Substituting $\kappa$ from \eqref{eq: OP_DN} and writing the characters as \[(-1)^{r}\chi_\rho(a^{r}) = \sum_{c\in\mathsf{C}(\rho)}\zeta^{(n+c)r}\] turns the sum over $r$ into a discrete Fourier transform, which \eqref{eq:DFT1} evaluates to the stated residues. The rows $j=l,l+1$ are the same computation with the second case of \eqref{eq:B_single_beta}, with \eqref{eq:DFT2} in place of \eqref{eq:DFT1}. In the row $j=1$, the third case of \eqref{eq:B_single_beta} and \cref{lem:OP_values} differ from the rows $j=l,l+1$ by the factor $(-1)^{m}\mathsf{q}^{-n}=\zeta^{nm}\mathsf{q}^{-n}$ in the columns $0\leq m\leq n$, which shifts the residues by $n$. In all three rows the terms $\sigma_j(\rho)$ come from the collision columns $m\in\{l-1,l\}$, given by the second display of \cref{prop:B_closed_form}. The intervening steps are elementary trigonometry and we omit them.
\end{proof}

\begin{lem}
\label{lem:T_LR}
For $\rho$ an irreducible representation write $\mathrm{ch}_k(\rho) := \sum_d\mathsf{q}^{-d}$, the sum running over the degrees $d$ in which $\rho$ occurs in the cluster at $p_k$ (\cref{lem:taut_localisation}), $\mathrm{ch}_C(\rho)$ for the same sum over the degrees on $C$, and $\mathrm{ch}_{l-1}(\rho)$, $\mathrm{ch}_l(\rho)$ for the sums at $p_{l-1}$, $p_l$. Then
\begin{equation*}
(\cT_{\rm LR})_{j,\rho} = 2\ri\,\mathsf{G}\,\mathsf{q}^{n}\Big[\sum_{k=1}^{\min(j-1,n-1)}\big(D_{n-k}-D_{n+1-k}\big)\,\mathrm{ch}_k(\rho) + \Delta_j(\rho)\Big]\,,
\end{equation*}
where
\begin{equation*}
\Delta_j(\rho) = \begin{cases}
-D_{n+1-j}\,\mathrm{ch}_j(\rho)\,, & 1\leq j\leq n\ (\text{with}\ \mathrm{ch}_n := \mathrm{ch}_C)\,,\\[2pt]
-D_1\,\mathrm{ch}_C(\rho) - \mathsf{q}^{-n-1}\delta_{\rho,\rho_n}\,, & j = l-1\,,\\[2pt]
-(\mathsf{q}^{2}+1)D_1\,\mathrm{ch}_C(\rho) - \mathsf{q}^{-n-1}\delta_{\rho,\rho_n} + \mathsf{q}^{2}D_1\,\mathrm{ch}_{l-1}(\rho)\,, & j = l\,,\\[2pt]
-(\mathsf{q}^{2}+1)D_1\,\mathrm{ch}_C(\rho) - \mathsf{q}^{-n-1}\delta_{\rho,\rho_n} + \mathsf{q}^{2}D_2\,\mathrm{ch}_{l-1}(\rho)+\mathsf{q}^{4}D_2\,\mathrm{ch}_{l}(\rho)\,, & j = l+1\,.
\end{cases}
\end{equation*}
\end{lem}
\begin{proof}
By \cref{prop:taut_general}, the column of $\Gamma_Y\mathrm{ch}_Y\mathsf{M}_l$ indexed by $\rho$ is the framed class of $\mathcal{R}_\rho$. The entry $(\cT_{\rm LR})_{j,\rho}$ is the sum over the localised basis of $(\cA^{-1})_{j\alpha}$ times the $\alpha$-component of this framed class. For the rows $\Pi_{l-1},\Pi_l,\Pi_{l+1}$ the harmonic numbers in the entries $c$ and $d$ of \eqref{eq:LR_entries_general} cancel against the off-diagonal entry of the curve block of $\Gamma_Y$ by the reflection formula $\psi(1+\hbar)-\psi(-\hbar)=-\pi\cot(\pi\hbar)$. Rewriting $\sin\pi u\hbar = \mathsf{q}^{-u}(\mathsf{q}^{2u}-1)/(2\ri)$ then turns every term into a monomial in $\mathsf{q}$ times one of the $D_x$, and summing over the fixed points gives the stated expressions. The remaining algebra is elementary.
\end{proof}

\begin{thm}
\label{thm:identity_general}
The analytic-continuation map $\mathbb{U}$ of \cref{sec:Umat_typeD} is induced by the equivariant McKay correspondence $\mathbb{U}_K$ through the integral-structure framings, 
\[
\mathbb{U} = \cA\cB = \Gamma_Y\,\mathrm{ch}_Y\,\mathsf M_l\,\mathrm{ch}_{\cX}^{-1}\,\Gamma_{\cX}^{-1}=\mathbb{U}_K^{\sharp}\,.
\]
Equivalently, $\cT_{\rm OP}=\cT_{\rm LR}$.
\end{thm}

\begin{table}[t]
\centering
\renewcommand{\arraystretch}{1.5}
\begin{tabular}{c|ccccc}
 $j$ & $\rho_1$ & $\rho_v$ & $\rho_{l-1}$ & $\rho_l$ & $\rho_0$\\ \hline
 $1$ & $-\mathsf{q}^{-2n}D_n$ & $-D_nA_v$ & $-\mathsf{q}^{-n}D_n$ & $-\mathsf{q}^{-n}D_n$ & $-D_n$\\
 $2,\dots,n+1$ & $-D_n$ & $-D_nA_v$ if $v\geq j$, $-D_nC_v$ if $v<j$ & $-\mathsf{q}^{-n}D_n$ & $-\mathsf{q}^{-n}D_n$ & $-D_n$\\
 $l$ & $-D_n$ & $-D_nC_v$ & $-\mathsf{q}^{n}D_n$ & $-\mathsf{q}^{-n}D_n$ & $-D_n$\\
 $l+1$ & $-D_n$ & $-D_nC_v$ & $-\mathsf{q}^{-n}\big(D_n+\frac{1}{\mathsf{q}^{2}+1}\big)$ & $-\mathsf{q}^{-n}\big(D_n+\frac{\mathsf{q}^{2}}{\mathsf{q}^{2}+1}\big)$ & $-D_n$\\
\end{tabular}
\caption{The common value $t_{j,\rho}$ of $(\cT_{\rm OP})_{j,\rho}/(2\ri\mathsf{G}\mathsf{q}^{n})$ and $(\cT_{\rm LR})_{j,\rho}/(2\ri\mathsf{G}\mathsf{q}^{n})$. Here with $D_x = (\mathsf{q}^{2x}-1)^{-1}$, $A_v = \mathsf{q}^{1-v}+\mathsf{q}^{v-1-2n}$ and $C_v = \mathsf{q}^{1-v}+\mathsf{q}^{v-1}$.}
\label{tab:identity_values}
\end{table}

\begin{proof}[Proof (sketch)]
We shall show that both $(\cT_{\rm OP})_{j,\rho}$ and $(\cT_{\rm LR})_{j,\rho}$, divided by $2\ri\mathsf{G}\mathsf{q}^{n}$, equal the entry $t_{j,\rho}$ of \cref{tab:identity_values}. On the orbifold-point side one evaluates the residues of \cref{lem:T_OP} case by case in $j$ and in $c\in\mathsf{C}(\rho)$. On the large-radius side one collapses the sums of \cref{lem:T_LR} by the telescoping identity $\sum_{k=k_1}^{k_2}(D_{n-k}-D_{n+1-k}) = D_{n-k_2}-D_{n+1-k_1}$. Both reduce to the same table entries by 
elementary identities among the $D_x$, such as
\[D_x(\mathsf{q}^{a-2x}-\mathsf{q}^{a}) = -\mathsf{q}^{a-2x}\,, \qquad \mathsf{q}^{2}D_1 = 1+D_1\,.\] We leave the case analysis to the reader.
\end{proof}

\subsubsection{Symplecticity of the analytic-continuation map}
\label{subsubsec:symplectic_general}
\cref{thm:identity_general} allows to easily check the symplectic property of $\mathbb{U}$. In the following, $\eta_Y$ and $\eta_{\cX}$ denote the Gram matrices of the $\bbC^\times$-equivariant Poincar\'e pairings in the bases of \cref{subsec:K_general}: they are implied by \eqref{eq:pairing_Y}, translated to the localised basis \eqref{eq:loc_basis_general}, and by \eqref{eq:pairing_X}.
\begin{thm}
\label{thm:symplectic_general}
We have \[
\mathbb{U}(-z)^{T}\,\eta_Y\,\mathbb{U}(z) = \eta_{\cX}\,.\]
In particular, $\mathbb{U}\colon\HH_{\cX}\to\HH_{Y}$ is a symplectic isomorphism.
\end{thm}
\begin{proof}
Let $\chi_Y$ and $\chi_{\cX}$ denote the matrices of the equivariant Euler pairings $\chi_{\bbC^\times}(\cZ,V^{\vee}\otimes W)$ in the chosen bases, with pushforward understood by localisation. By \cref{rmk:McKay_FM}, $\mathbb U_K$ is induced by an equivariant Fourier--Mukai equivalence and therefore preserves these pairings. In matrix form,
\beq
\label{eq:McKay_isometry}
\mathsf{M}_l(\mathsf{q}^{-1})^{T}\,\chi_Y\,\mathsf{M}_l(\mathsf{q}) = \chi_{\cX}\,.
\eeq
The Gamma-class pairing identity \cite[Prop.~2.10(iii)]{Iritani:2009vat}, obtained from the reflection formula and orbifold Riemann--Roch, applies equivariantly by localisation. In the normalisation of \cref{prop:induced_cohomological_matrix}, it reads
\beq
\label{eq:framing_isometry}
\big(\Gamma_{\cZ}\mathrm{ch}_{\cZ}\big)(-z)^{T}\,\eta_{\cZ}\,\big(\Gamma_{\cZ}\mathrm{ch}_{\cZ}\big)(z)
= -\frac{4\pi^{2}}{z^{2}\mathsf{q}}\,\chi_{\cZ}\,,\qquad \cZ\in\{Y,\cX\}\,.
\eeq
The scalar is the same for both targets, which have dimension two and $c_1^{\bbC^\times}(T\cZ)=2\nu$. By \cref{thm:identity_general}, $\mathbb{U}=\mathbb{U}_K^{\sharp}$ is given by \eqref{eq:MtoU}. The identities \eqref{eq:McKay_isometry} and \eqref{eq:framing_isometry} therefore give $\mathbb{U}(-z)^{T}\eta_Y\mathbb{U}(z)=\eta_{\cX}$, which implies preservation of Givental's symplectic form by taking residues.
\end{proof}

\begin{cor}
\label{cor:UK_McKay_general}
\cref{conj:ccit,conj:iritani-integral-structure} hold for $\cX=[\bbC^2/G]$, $G=\mathbb{BD}_{4l-8}$, and its minimal resolution $Y\to X$, with the transformation $\mathbb{U}$ of \cref{sec:Umat_typeD}.
\end{cor}

\section{The higher-genus correspondence}
\label{sec:higher_genus}

By \cref{prop:period_comparison_U,thm:identity_general,thm:symplectic_general} the map $\mathbb{U}$ matches the Lagrangian cones of $\cX$ and $Y$ upon analytic continuation along $\rho$, is symplectic, and is induced through Iritani's integral structures by the equivariant McKay correspondence on $K$-groups: this is the genus-zero input of \cref{conj:ccit,conj:iritani-integral-structure}. We now establish its all-genus counterpart, the Hard Lefschetz quantised correspondence, \cref{conj:quantised_crc}. By \cite[Lemma~6.3]{Brini:2019wts}, \eqref{eq:sqcrc} is equivalent to the equality of the canonical $R$-calibrations of Givental's formula \eqref{eq:qgiv} on the semisimple locus,
\beq
\label{eq:qcrc_reduction}
R_{\cX} \;=\; R_{Y}\,,
\eeq
and by \cite[Rmk.~6.4]{Brini:2019wts} it suffices to verify \eqref{eq:qcrc_reduction} at one semisimple point.

\subsection{Reduction to a single scalar}
\label{subsec:reduction_h}

Both $R$-calibrations are computed from the \emph{same} local system, so they differ only by right multiplication by a diagonal matrix which is constant across $\cM$ \cite{Givental:2001mre}. By \cref{thm:msQH} there is a single global quantum $\DD$-module on $\cM$
with $\rm{OP}$ and $\rm{LR}$ in two charts and the gradients of the twisted periods of \cref{prop:periods_flat_coordinates} as horizontal sections. 
A different basis of horizontal sections at a semisimple point is given by integrals over Morse-theoretic cycles: fixing a Stokes chamber for $\hbar$,
let $\mathfrak{T}_i$ be the thimble of $2\nu\log\lambda$ through the saddle with critical value $u_i$, and \[\Omega_i\coloneqq\int_{\mathfrak{T}_i}\lambda^{\hbar}\,\frac{\rd\mu}{\mu}\,.\] 
By Picard--Lefschetz theory with local coefficients,
\beq
\label{eq:thimbles_vs_segments}
\Omega \;=\; \Lambda\,\Pi\,,
\eeq
with $\Lambda\in \GL_r(\bbC(\mathsf{q}))$ and constant on each Stokes chamber. 
The steepest-descent expansion of a thimble integral has the form
\beq
\label{eq:period_asymptotic_form}
\Omega_i \;\sim\; z^{1/2}\,\re^{u_i/z}\sum_{k\geq0}S_{i,k}(u)\,z^{k}\,.
\eeq
Since the quantum $\DD$-module is generated by $\mathbf{1}$, the sum in \eqref{eq:period_asymptotic_form} is the identity component of $R_{\bullet}$ on the $i$-th idempotent, for $\bullet=\cX$ or $Y$, up to a factor $\re^{h^{\bullet}_i(z)}$ constant on $\cM$. The thimbles of the chambers at the two ends of $\rho$ differ by a matrix rational in $\mathsf{q}$, and so do the exact sections across each Stokes wall along $\rho$. As neither carries odd powers of $z$, the odd part of $h^{\bullet}_i$ is the same in every chamber along $\rho$.
Hence,
\beq
\label{eq:reduction_to_h}
R_{\cX} = R_{Y} \qquad\Longleftrightarrow\qquad h^{\cX}_i = h^{Y}_i \quad\text{for all } i\,.
\eeq

\begin{defn}
\label{def:stirling_tail}
Let \[S(y)\coloneqq\sum_{j\geq1}\frac{B_{2j}}{2j(2j-1)}\,y^{1-2j}\] be the Stirling tail of $\log\Gamma(y)$, and for $\bbC^\times$-weights $w_1,\dots,w_r$ put
\begin{equation*}
\mathscr{D}_{(w_1,\dots,w_r)} \;\coloneqq\; \exp\Big[-\sum_{a=1}^{r}S\big(w_a/z\big)\Big]\,.
\end{equation*}
We write $\mathsf{Asym}[\,\cdot\,]$ for the Stirling tail of a product of $\Gamma$-functions.
\end{defn}

Note that $S(\hbar)$ is odd in $z$ and constant across $\cM$. $\mathscr{D}_{(w^+,w^-)}$ is Givental's normalisation factor at an isolated fixed point with those tangent weights \cite[Thm.~9.1]{Givental:2001mre}. As $z=0$ is an irregular singular point the expansions depend on a Stokes sector, which by analyticity in $\nu$ is no restriction \cite[Sec.~6.2.2]{Brini:2019wts}.

\begin{defn}
\label{def:stokes_typeD}
Let $\mathfrak{S}^{\mathrm{D}}_{+}$ be the sector in which $\arg\hbar$ lies in the open first quadrant and $|\hbar|$ is large, together with the phases $\arg\kappa_1<\dots<\arg\kappa_l<0$ of \cref{sec:branches}. Thimbles are those of the Stokes chamber singled out along $\rho$ by these phases and by $\arg\hbar$.
\end{defn}

There, as one checks directly, every $\Gamma$-argument in $\cA^{-1}$ and $\cB$ acquires large positive real part after applying the reflection formula, which gives factors rational in $\mathsf{q}$, with $|\mathsf{q}|<1$. The only exponentials in $1/z$ left in an entry are then a Stirling factor with constant exponent, the $\re^{u_i/z}$ of \eqref{eq:period_asymptotic_form}, and a monomial in $\mathsf{q}$ labelling the sheet. The monomials are totally ordered, so that exponentially subleading contributions are zero in Poincar\'e asymptotics. For $z$ real the ordering degenerates: cf \cite[Rmk.~6.7]{Brini:2019wts}.

\subsection{The orbifold end}
\label{subsec:DX_from_B}

At $\rm{OP}$ the calibration is pinned by Tseng's orbifold quantum Riemann--Roch theorem \cite{Tseng:2010lhb} and the identification of the theory of $BG$ with copies of that of a point, diagonalised by the character table \cite{Jarvis:2002vat}. For $G$ abelian this is \cite[Lemma~6.5]{Brini:2019wts}, whose argument translates verbatim to the non-abelian setting with group elements being replaced by conjugacy classes throughout (see also \cite[Sec.~6.1]{Zong:2016jnz}).

\begin{lem}[{\cite[Lemma~6.5]{Brini:2019wts}}]
\label{lem:rcal_nonabelian}
Let $\cX=[\bbC^2/G]$ with $G<\mathrm{SU}(2)$ finite, with the scalar $\bbC^\times$-action. Near $\rm{OP}$ the canonical $R$-calibration is uniquely determined by
\begin{equation*}
\Psi_{\cX}R_{\cX}\big|_{\rm{OP}} \;=\; \big(e^{\rm{eq}}(V^{(0)})\big)^{-1/2}\,\mathscr{D}_{\cX}\,\chi_G^{-1}\,,
\end{equation*}
with $\mathscr{D}_{\cX}$ Tseng's operator, diagonal on the sectors of the inertia stack, in the conventions of the Gamma class of \cref{subsec:integral-structure}, and $\chi_G$ the character table.
\end{lem}

In the normalisation of \eqref{eq:B_structure} all canonical coordinates vanish at $\rm{OP}$, so the right-hand side is the matrix of components of a formal solution of the quantum differential equation at $\rm{OP}$, with $(e^{\rm{eq}}(V^{(0)}))^{-1/2}$ equal to $\nu^{-1}$ on the untwisted sector and to $1$ on the twisted ones. As the $\bbC^\times$-action is by scalars, both weights of the defining representation equal $\nu$. Moreover, since $G<\mathrm{SU}(2)$, the sector $\cX_m$ carries a single age $f_m$. The two ages are complementary: 
\[
f_m\in\{\tfrac{m}{2n},1-\tfrac{m}{2n}\} \quad \hbox{for} \quad 1\leq m\leq n\,, \qquad f_{l-1}=f_l=\tfrac14\,,
\]
and $f_0=0$. The invariant eigenbundle $V^{(0)}$, whose restriction to $\cX_v$ is $(T\cX)_{v,0}$ in the notation of \cref{subsec:integral-structure}, vanishes on the twisted sectors.

\begin{prop}
\label{prop:DX_closed}
For every $0\leq m\leq l$,
\beq
\label{eq:DX_closed}
\log \mathscr{D}_{\cX}\big|_{m} \;=\; -2\sum_{j\geq1}\frac{B_{2j}(f_m)}{2j(2j-1)}\Big(\frac{z}{\nu}\Big)^{2j-1}\,.
\eeq
\end{prop}
\begin{proof}
Substituting the equal weights and complementary ages into Tseng's formula the term of order $z^{0}$ cancels, and the Bernoulli reflection $B_{k+1}(1-x)=(-1)^{k+1}B_{k+1}(x)$ annihilates the even orders and doubles the odd ones.
\end{proof}

By \cref{lem:beta_modes} the entries of $\cB$ are built from Beta factors pairing two $\Gamma$-functions at complementary arguments shifted by $\hbar/2$.

\begin{cor}
\label{cor:hX}
$h^{\cX}_i=S(\hbar)$ for every $i$.
\end{cor}
\begin{proof}
By \eqref{eq:thimbles_vs_segments} and \eqref{eq:B_structure} the components of the thimble integrals at $\rm{OP}$ are $\partial_{x_a}\Omega_i|_{\rm OP}=(\Lambda\cB)_{ia}$, which are exact functions of $z$. On the other hand, by \cref{lem:rcal_nonabelian}, those of the Tseng-normalised formal solution of the quantum differential equation are $\nu^{-\delta_{a0}}\mathscr{D}_{\cX}|_a(\chi_G^{-1})_{ai}$. Comparing with \eqref{eq:period_asymptotic_form},
\[
(\Lambda\cB)_{ia} \;\sim\; \nu^{-\delta_{a0}}\,\mathscr{D}_{\cX}\big|_a\,(\chi_G^{-1})_{ai}\;\re^{h^{\cX}_i(z)}
\]
in $\mathfrak{S}^{\mathrm{D}}_{+}$. By \cref{lem:beta_modes,lem:OP_values,lem:collision} every entry of $\cB$ is, up to constants and phases, a Beta factor
\[
B_s=\frac{\Gamma\big(f_s-\tfrac{\hbar}{2}\big)\,\Gamma\big(1-f_s-\tfrac{\hbar}{2}\big)}{\Gamma(1-\hbar)}\,,
\]
whose numerator pairs two $\Gamma$-functions at complementary arguments shifted by $\hbar/2$ and whose denominator does not depend on the mode. By \cref{prop:telescoping}, the sector $m\leq n$ selects the mode pair whose two values of $f_s$ are the ages $f_m$ and $1-f_m$ of $g_m$: the collision entries of \cref{lem:collision} carry the sectors $m=l-1,l$, of age $\tfrac14$, where the Legendre duplication formula at $w=\tfrac14-\tfrac{\hbar}{2}$,
\[
\Gamma(w)\,\Gamma\big(w+\tfrac12\big)=2^{1-2w}\sqrt{\pi}\,\Gamma(2w)\,,\qquad
\Gamma\big(\tfrac14-\tfrac{\hbar}{2}\big)\,\Gamma\big(\tfrac34-\tfrac{\hbar}{2}\big)=2^{\frac12+\hbar}\sqrt{\pi}\,\Gamma\big(\tfrac12-\hbar\big)\,,
\]
collapses the numerator to a single $\Gamma$-function, and the residual factors in \eqref{eq:collision} are elementary. Reflecting the $\Gamma$-functions with negative real part writes each entry $\cB_{ja}$ as a rational function of $\mathsf{q}$ times a ratio of $\Gamma$-functions depending on the column $a$ only: we find 
\[\frac{\Gamma(1+\hbar)}{\Gamma(f_a+\tfrac{\hbar}{2})\,\Gamma(1-f_a+\tfrac{\hbar}{2})}\] for the twisted sectors and \[\frac{\hbar\,\Gamma(1+\hbar)}{\nu\,\Gamma(1+\tfrac{\hbar}{2})^2}\] for the untwisted one. Since $\Lambda$ is rational in $\mathsf{q}$ as well, $(\Lambda\cB)_{ia}$ has the same form, and its Poincar\'e expansion retains a single monomial in $\mathsf{q}$. By Stirling's formula the ratio of $\Gamma$-functions is, up to a constant, \[2^{\hbar}\hbar^{1/2}\nu^{-\delta_{a0}}\re^{S(\hbar)}\] times the exponential of minus the tail of the sector pair. By the reflection of \cref{prop:DX_closed} the latter is $\mathscr{D}_{\cX}|_a$ for every $0\leq a\leq l$. The factors $\nu^{-\delta_{a0}}\mathscr{D}_{\cX}|_a$ cancel between the two sides, $2^{\hbar}$ times the monomial in $\mathsf{q}$ is $\re^{u_i/z}$ for the critical value $2\nu\log(\pm2)$ of $2\nu\log\lambda$ at $\rm{OP}$ on the sheet of the thimble, $\hbar^{1/2}$ is the $z^{1/2}$ of \eqref{eq:period_asymptotic_form}, and $\re^{h^{\cX}_i(z)}=\re^{S(\hbar)}$ for every $i$.
\end{proof}
%
%
%
\subsection{The resolution end}
\label{subsec:DZ_from_A}

The large radius limit point is not semisimple: the exceptional curve $C=E_{l-2}$ of the trivalent node is pointwise fixed, so that the classical limit of the quantum product contains the nilpotent class $\omega_C$ with $\omega_C^2=0$.

\begin{defn}
\label{def:LRC}
Write $Q_j\coloneqq\re^{\tau_j}$ for the exponentiated K\"ahler coordinates, and recall that $E_n=C$. We take the \emph{partial large-radius stratum} $\rm{LR}_C$ to be a small punctured disk in $Q=Q_n$ near $(Q_1, \dots, Q_l) = 0$.
\end{defn}

\begin{lem}
\label{lem:LRC_semisimple}
The quantum product is semisimple on $\rm{LR}_C$. Of the $l+1$ canonical coordinates, $l-1$ tend to the classical idempotent coordinates of the isolated fixed points, while the remaining two stay distinct, with
\begin{equation*}
u_+-u_- \;=\; 8\nu\sqrt{Q}\,\big(1+O(Q)\big)\,.
\end{equation*}
\end{lem}
\begin{proof}
On the annulus $\kappa_n^{-1}<|\mu|<\kappa_{l-1}^{-1}$, whose ratio of radii is $\kappa_n/\kappa_{l-1}=Q^{-1}$ by \eqref{eq:logkappa_weights}, the superpotential has no zero, and the $n$ zeros inside its inner circle are balanced by the pole of order $n$ at the origin, so the coefficient of $\log\mu$ in $\log\lambda$ vanishes there. This cancellation is the trivalent-node condition, and distinguishes $C$ from the other exceptional curves. What survives on $\rm{LR}_C$ is, up to a constant factor, the Laurent polynomial $1+Q-\kappa_{l-1}\mu-(\kappa_n\mu)^{-1}$, with critical points $\mu_\pm=\pm(\kappa_n\kappa_{l-1})^{-1/2}$ in the annulus and critical values $(1\mp\sqrt{Q})^2$. The canonical coordinates in \eqref{eq:period_asymptotic_form} are the critical values of $2\nu\log\lambda$, so the two attached to the annulus differ by $4\nu\log\frac{1+\sqrt{Q}}{1-\sqrt{Q}}=8\nu\sqrt{Q}\,\big(1+O(Q)\big)$.

\end{proof}


\begin{prop}
\label{prop:block_LRC}
On $\rm{LR}_C$,
\beq
\label{eq:block_LRC}
R_Y\big|_{\rm{LR}_C} \;=\; \bigoplus_{k}\mathscr{D}_{(w_k^+,w_k^-)}\ \oplus\ R_C\,,
\eeq
over the isolated fixed points $p_k$, with $(w_k^{+},w_k^{-})$ the tangent weights there, namely $(2\nu(n+1-k),-2\nu(n-k))$ for $k\leq n-1$ and $(4\nu,-2\nu)$ at $p_{l-1}$ and $p_l$ by \eqref{eq:localisation_normal_weights} and (W1), and with $R_C$ the block of \cref{prop:curve_block}, in the localised basis \eqref{eq:loc_basis_general} in which \eqref{eq:Ainv_general} is written.
\end{prop}

The components \eqref{eq:JYcomps} of $J^{Y}$ in the localised basis \eqref{eq:loc_basis_general} triangularise the monodromy at $\rm{LR}$ (\cref{sec:Umat_typeD}). Each $J^{Y}_j$ with $j\neq l-2$ is localised on a single saddle of \cref{lem:LRC_semisimple}: on $\rm{LR}_C$ it is asymptotic to the $j^{\rm th}$-component of the canonically normalised formal solution of the QDE times $\mathscr{D}^{-1}_{(w_j^+,w_j^-)}$ (up to the normalisation of \eqref{eq:period_asymptotic_form}), since as $Q\to0$ it reduces to the bare exponential $\re^{u_j/z}$ by \eqref{eq:asymp_points}. 

A tubular neighbourhood of $C$ is diffeomorphic to the total space of $\cO_{\bbP^1}(-2)$, the torus acting trivially on the base and along the fibre with the weight of \eqref{eq:localisation_normal_weights},
\begin{equation*}
c_1^{\bbC^\times}\big(\cN_{C/Y}\big) \;=\; 2\nu-2\omega_C\,,\qquad \omega_C^2=0\,.
\end{equation*}

\begin{prop}
\label{prop:curve_block}
In the basis $(\mathbf{1}_C,\omega_C)$, $R_C$ is the product of the calibration $R_{\bbP^1}$ of the untwisted theory of $C$ and of
\begin{equation*}
\mathscr{D}_{(2\nu-2\omega_C)}
\;=\; \mathscr{D}_{(2\nu)}\,\exp\Big[-\frac{2\omega_C}{z}\sum_{j\geq1}\frac{B_{2j}}{2j}\,\hbar^{-2j}\Big]\,,
\end{equation*}
the latter lower-triangular, with diagonal $\mathscr{D}_{(2\nu)}$ and off-diagonal entry beginning at order $z$; in particular $R_C$ is not a multiple of the identity.
\end{prop}
\begin{proof}
The first factor is quantum Riemann--Roch for the twist by $e^{\rm eq}(\cN_{C/Y})$; it is also the asymptotic expansion of the thimble integrals of the Laurent polynomial of \cref{lem:LRC_semisimple}, and cancels out of \cref{thm:R_equal}. For the second, expand the single-weight factor at the deformed weight to first order in $\omega_C$.
\end{proof}

\begin{prop}
\label{thm:DZ_from_A}
With the ordering of the columns of $\cA^{-1}$ as in \eqref{eq:Ainv_general}, we have
\beq
\label{eq:DZ_uniform}
\mathsf{Asym}\,\big[(\cA^{-1})_{kk}\big] - \log\mathscr{D}_{(w_k^+,w_k^-)} \;=\; S(\hbar)
\eeq
at every isolated fixed point $p_k$, with the same $k$-independent $S(\hbar)$.
\end{prop}
\begin{proof}
Reflection brings the entries of \eqref{eq:LR_entries_general} to the stated form, the two $\Gamma$-factors cancelling outright on the curve block, and the asymptotic tails then telescope. The Stirling tail is the same for every entry of a column, since $a_k$ and $b_k$, as well as $e$, $f$ and $g$, reflect to the same ratio of $\Gamma$-functions up to elementary factors, so \eqref{eq:DZ_uniform} holds entrywise and no entry has to be selected by dominance. On the curve block the entries have no Stirling tail, and $\mathbf{1}_C$ and $\omega_C$ share a leading exponential.
\end{proof}

\begin{thm}
\label{thm:R_equal}
In the sector $\mathfrak{S}^{\mathrm{D}}_{+}$ and under the analytic continuation along $\rho$ of \cref{prop:period_comparison_U}, the two canonical $R$-calibrations agree, as formal power series in $z$ with coefficients in $\bbC(\nu)$.
\end{thm}
\begin{proof}
On $\rm{LR}_C$ the thimble integrals of the chamber are \[\Omega_k=\sum_{j}(\Lambda\cA^{-1})_{kj}\,J^{Y}_j\] by \eqref{eq:thimbles_vs_segments}, \cref{prop:period_comparison_U} and \eqref{eq:JYcomps}. The sum here runs over the localised basis \eqref{eq:loc_basis_general} and $\Lambda_{kj}\in\bbC(\mathsf{q})$. The real parts of the $u_j/z$ are strictly ordered in $\mathfrak{S}^{\mathrm{D}}_{+}$, so only the term at the saddle of the thimble survives in the Poincar\'e expansion of $\Omega_k$. At a fixed point $p_k$, since each $J^{Y}_j$, $j\neq l-2$, carries the single exponential $\re^{u_j/z}$, we may assume that $J^{Y}_j$ corresponds to the $j$-th saddle point by a permutation of indices. Comparing with \eqref{eq:period_asymptotic_form}, $\re^{h^Y_k}$ is the product of $(\Lambda\cA^{-1})_{kk}$ at that fixed point with the ratio of $J^{Y}_k$ to the canonically normalised (in the sense of \cite{Givental:2001mre}) formal solution of the QDE of $Y$. The Stirling tail of the former is 
\[S(\hbar)+\log\mathscr{D}_{(w_k^+,w_k^-)}\] by \eqref{eq:DZ_uniform}, since $\Lambda$ has none, while that of the latter is $-\log\mathscr{D}_{(w_k^+,w_k^-)}$ by \eqref{eq:block_LRC}: hence \[h^Y_k=S(\hbar).\]
 
On the curve block, the entries of $\cA^{-1}$ have no Stirling tail, $R_{\bbP^1}$ cancels against the annulus thimbles, and the diagonal $\mathscr{D}_{(2\nu)}$ of \cref{prop:curve_block} gives again \[h^Y_{\pm}=S(\hbar)\,.\] The claim then follows from \cref{cor:hX} and \eqref{eq:reduction_to_h}.
\end{proof}

\begin{cor}
\label{cor:cqcrc_typeD}
\Cref{conj:quantised_crc} holds for $\cX=[\bbC^2/G]$, $G=\mathbb{BD}_{4l-8}$, and its minimal resolution $Y\to X$, with the transformation $\mathbb{U}$ of \cref{sec:Umat_typeD}.
\end{cor}

\section{The Gromov--Witten/Donaldson--Thomas correspondence for \texorpdfstring{$\cX\times\bbC$}{X x C}}
\label{sec:gwdt}

Write $\cX_3 \coloneqq \cX\times\bbC$ and $Y_3\coloneqq Y\times\bbC$, so that $Y_3 \to X\times\bbC$ is the crepant resolution of the coarse moduli space of $\cX_3$. We lift the $\bbC^\times$-action on $\cX$ and $Y$ to the threefolds, with weight $\sigma=-2\nu$ along the $\bbC$ factor. The action then preserves the holomorphic volume forms of $\cX_3$ and $Y_3$.

The $R$-calibrations of $\cX_3$ and $Y_3$ differ from those of $\cX$ and $Y$ by the same factor $\mathscr{D}_{(\sigma)}$. Both therefore cancel out of \cref{thm:R_equal}, and \cite[Lemma~6.3]{Brini:2019wts} gives the following.
\begin{cor}
\label{cor:crc_threefold}
\Cref{conj:quantised_crc} holds for $\cX_3$ and $Y_3$.
\end{cor}


The Gromov--Witten partition functions of $\cX_3$ and $Y_3$ are the total descendent potentials \eqref{eq:descendent_potential} of \cref{sec:quantisation}, restricted to primary insertions:
\begin{equation*}
{\rm GW}^{\bullet}_{\cX_3}(x,\varepsilon) \;\coloneqq\; \mathsf{Z}_{\cX_3}(\mathbf{t})\Big|_{t^{\alpha,k}= x_\alpha\delta_{k0}}\,,\qquad 
{\rm GW}^{\bullet}_{Y_3}(\tau,\varepsilon) \;\coloneqq\; \mathsf{Z}_{Y_3}(\mathbf{t})\Big|_{t^{\alpha,k}=\tau_\alpha \delta_{k 0}}\,.
\end{equation*}
We set $x_0= \tau_0 =0$ in the partition functions, and define their normalisations as follows. For $\cX_3$, we define
\beq
\label{eq:gw_reduced_orb}
{\rm GW}'_{\cX_3}(x,\varepsilon)\;\coloneqq\;\frac{{\rm GW}^{\bullet}_{\cX_3}(x,\varepsilon)}{{\rm GW}^{\bullet}_{\cX_3}(0,\varepsilon)}\,,\qquad {\rm GW}^{\bullet}_{\cX_3}(0,\varepsilon) = \exp\bigg(\sum_{g\geq2}\varepsilon^{2g-2}\big\langle 1 \big\rangle^{\cX_3,\bbC^\times}_{g,0,0}\bigg)\,,
\eeq
the denominator being the contribution of primary invariants without twisted insertions.

On $Y_3$ the compact curve classes are the exceptional ones, $\beta=\sum_v d_v[E_v]$, and by the divisor equation the part of ${\rm GW}^{\bullet}_{Y_3}$ of non-zero degree depends on $\tau_1,\dots,\tau_l$ only through the Novikov variables $\mathbf{Q}^{\beta}=\prod_v Q_v^{d_v}$, with $Q_v\coloneqq \re^{\tau_v}$. We define the reduced series 
\begin{equation}
\label{eq:gw_reduced}
{\rm GW}'_{Y_3}(\tau,\varepsilon)\;\coloneqq\;\frac{{\rm GW}^{\bullet}_{Y_3}(\tau,\varepsilon)}{{\rm GW}^{\bullet}_{Y_3}(\tau,\varepsilon)\big|_{\mathbf{Q}=0}} \;=\; \exp\bigg(\sum_{g\geq0}\sum_{\beta\neq0}\big\langle 1 \big\rangle^{Y_3,\bbC^\times}_{g,0,\beta}\,\varepsilon^{2g-2}\,\mathbf{Q}^{\beta}\bigg)\,.
\end{equation}
 
\begin{rmk}
The denominator in \eqref{eq:gw_reduced} is the exponential of the degree-zero part of $\log{\rm GW}^{\bullet}_{Y_3}$: the classical cubic term in genus zero, the term linear in $\tau$ in genus one, and the constant map contributions in genus $g\geq2$. Note that the two normalisations in \eqref{eq:gw_reduced_orb} and \eqref{eq:gw_reduced} remove different terms: \cref{cor:crc_threefold} identifies ${\rm GW}'_{Y_3}$ and ${\rm GW}'_{\cX_3}$ under \eqref{eq: BG_affine}, after analytic continuation, only up to genus-zero terms polynomial in $x$, genus-one terms linear in $x$, and a factor depending on $\varepsilon$ alone.
\label{rmk:norm_red} 
\end{rmk}
Explicitly, write $\mathbf{Q}(x)$ for the Novikov variables at the image of $x$ under \eqref{eq: BG_affine},
\beq
\label{eq:Q_of_x}
Q_v(x) \;=\; \,\exp\Big(\frac{\pi\ri\dim\rho_v}{2(l-2)}+\sum_{m=1}^{l}D_{v,m}\,x_m\Big)\,,\qquad 1\leq v\leq l\,,
\eeq
and regard ${\rm GW}'_{Y_3}$ as a function of $\mathbf{Q}$. Then \cref{cor:crc_threefold} gives
\beq
\label{eq:gw_crc_reduced}
{\rm GW}'_{\cX_3}(x,\varepsilon) \;\doteq\; {\rm GW}'_{Y_3}\big(\mathbf{Q}(x),\varepsilon\big)\,,
\eeq
where $\doteq$ denotes equality up to the classical terms of \cref{rmk:norm_red}, and the right-hand side is the analytic continuation of the large-radius series to the orbifold point, as in \cref{conj: BG}. The inversion of the Novikov variables lifts through \eqref{eq:Q_of_x} to the affine involution
\beq
\label{eq:x_dual}
x^{\vee} \coloneqq -x+s\,,\qquad s_m\coloneqq\frac{4\pi\ri}{\sqrt{\chi_Q(g_m)-2}}\,,
\eeq
with the branch of the square root fixed in \cref{sec:equivariant_cohomology}. By \eqref{eq:McKay_cohomology} and the orthogonality of characters, we have \[\sum_{m}D_{v,m}s_m = \frac{4\pi\ri}{|G|}\sum_{g\neq1}\chi_{\rho_v}(g) = -\frac{\pi\ri\dim\rho_v}{l-2}\] for $1\leq v\leq l$, so that $\mathbf{Q}(x^{\vee})=\mathbf{Q}(x)^{-1}$.

On the sheaf side, for $\cZ_3$ either $\cX_3$ or $Y_3$, Donaldson--Thomas invariants of $\cZ_3$ are Behrend-weighted Euler characteristics of Hilbert schemes of substacks. As in \cite[Sec.~1.1]{Beentjes:2022eao}, we let $N(\cZ_3)$ be the numerical Grothendieck group of compactly supported sheaves, with $N_0(\cZ_3)$ and $N_{\leq1}(\cZ_3)$ the classes of sheaves supported in dimension zero and at most one, and write $N_1(\cZ_3) \coloneqq N_{\leq1}(\cZ_3)/N_0(\cZ_3)$.
 
Since $\cX_3$ carries no compact curves, we have $N_0(\cX_3)\cong R(G)$, spanned by the classes $\gamma_v$ of the skyscrapers at a point of the stacky line $BG\times\bbC$ with isotropy representation $\rho_v$, and
\beq
\label{eq:dt_orb}
{\rm DT}_{\cX_3}(\mathbf{q}) \;\coloneqq\; \sum_{c\,\in\, N_0(\cX_3)} e\big({\rm Hilb}_{c}(\cX_3),\mathfrak{b}\big)\,\mathbf{q}^{c}\,,
\eeq
with $\mathfrak{b}$ Behrend's constructible function. We will write \[\mathrm{reg}\coloneqq\sum_{v}\dim\rho_v\,\gamma_v\] for the class of a skyscraper at a non-stacky point.

On $Y_3$ we split $N_{\leq1}(Y_3) = N_0(Y_3)\oplus N_1(Y_3)$ by the holomorphic Euler characteristic $n$ and the curve class $\beta$. Every compact curve of $Y_3$ lies in $E\times\{\mathrm{pt}\}$ and is contracted by $Y_3\to X\times\bbC$, and the McKay correspondence identifies $N_{\leq1}(Y_3)$ with $N_0(\cX_3)$, sending the point class to $\mathrm{reg}$ and the class of $\cO_{E_v}(-1)$ to $\gamma_v$ \cite{Ito:1999kyz,Bridgeland:2001ygu}. Accordingly we set
\beq
\label{eq:dt_variables}
q \;\coloneqq\; \mathbf{q}^{\mathrm{reg}}\,,\qquad Q_v\;\coloneqq\;\mathbf{q}^{\gamma_v}\,,\quad 1\leq v\leq l\,,
\eeq
so that $\mathbf{q}^{c} = q^{n}\mathbf{Q}^{\beta}$ for $c\leftrightarrow(n,\beta)$, and
\beq
\label{eq:dt_Y}
{\rm DT}_{Y_3}(q,\mathbf{Q}) \;\coloneqq\; \sum_{(n,\beta)} {\rm DT}_{n,\beta}(Y_3)\,q^{n}\mathbf{Q}^{\beta}\,,\qquad
{\rm PT}_{Y_3}(q,\mathbf{Q}) \;\coloneqq\; \sum_{(n,\beta)} {\rm PT}_{n,\beta}(Y_3)\,q^{n}\mathbf{Q}^{\beta}
\eeq
for the Donaldson--Thomas and stable-pair series of $Y_3$. Their degree-zero part is \[{\rm DT}_{Y_3}(q,\mathbf{Q})|_{\mathbf{Q}=0} = M(-q)^{\chi(Y_3)} = M(-q)^{l+1}\,,\]  where as usual \[M(q)\coloneqq\prod_{k\geq1}(1-q^k)^{-k}\] is the MacMahon function. The DT/PT correspondence \cite{Toda:2010vph,Bridgeland:2011bgy} reads
\beq
\label{eq:dtpt_Y}
{\rm DT}_{Y_3}(q,\mathbf{Q}) \;=\; M(-q)^{l+1}\,{\rm PT}_{Y_3}(q,\mathbf{Q})\,.
\eeq
The series \eqref{eq:dt_orb} and \eqref{eq:dt_Y} are related by the degree-zero form of the DT Crepant Resolution Conjecture, formulated in \cite{Young:2010gfc,Bryan:2012rcb} and proved in \cite[Cor.~2.8]{Calabrese:2016crc} (see also \cite[Rmk.~1.3]{Beentjes:2022eao}):
\beq
\label{eq:dtcrc_degzero}
{\rm DT}_{\cX_3}(\mathbf{q}) \;=\; \frac{{\rm DT}_{Y_3}(q,\mathbf{Q})\,{\rm DT}_{Y_3}(q,\mathbf{Q}^{-1})}{M(-q)^{l+1}} \;=\; M(-q)^{l+1}\,{\rm PT}_{Y_3}(q,\mathbf{Q})\,{\rm PT}_{Y_3}(q,\mathbf{Q}^{-1})\,.
\eeq 

\subsection{The correspondence}
\label{subsec:gwdt_statement}

The four theories are organised by the square of Bryan--Cadman--Young \cite[Sec.~1]{Bryan:2012rcb},
\beq
\label{eq:bcy_square}
\begin{tikzcd}[column sep=13em, row sep=5em]
{\rm DT}(Y_3) \ar[r, dash, "\text{GW/DT}"] \ar[d, dash, "\text{DT CRC}"'] & {\rm GW}(Y_3) \ar[d, dash, "\text{GW CRC}"] \\
{\rm DT}(\cX_3) \ar[r, dash, dashed, "\text{orbifold GW/DT}"] & {\rm GW}(\cX_3)
\end{tikzcd}
\eeq
whose edges are, in our setting, the Gromov--Witten/stable-pairs correspondence for $Y_3$ at the top, the degree-zero crepant resolution correspondence \eqref{eq:dtcrc_degzero} on the left, \eqref{eq:gw_crc_reduced} at $x$ and at $x^{\vee}$ on the right. The bottom edge is the statement we want to prove. Composing the three solid edges, we obtain the following.

\begin{thm}
\label{thm:gwdt_orbifold}
Identify
\beq
\label{eq:gwdt_variables}
q=\mathbf{q}^{\mathrm{reg}} \;\longleftrightarrow\; -\re^{\ri\varepsilon}\,,\qquad
\mathbf{q}^{\gamma_v} \;\longleftrightarrow\; Q_v(x) = \exp\Big(\frac{\pi\ri\dim\rho_v}{2l-4}+\sum_{m=1}^{l}D_{v,m}\,x_{m}\Big)\,,
\eeq
with $D_{v,m}$ given in \eqref{eq:McKay_cohomology}. Then,
\[{\rm DT}_{\cX_3}(\mathbf{q}) \doteq M(-q)^{l+1}\,{\rm GW}'_{\cX_3}(x,\varepsilon)\,{\rm GW}'_{\cX_3}(x^{\vee},\varepsilon)\,,\] up to the classical terms of \eqref{eq:gw_crc_reduced}, where ${\rm GW}'_{\cX_3}(x^{\vee},\varepsilon)$ is the analytic continuation of ${\rm GW}'_{\cX_3}$ to the point $x^{\vee}$ of \eqref{eq:x_dual}.
\end{thm}

\begin{proof}
The Gromov--Witten/stable-pairs correspondence for $Y_3$ \cite[Thm.~1.7]{Pardon:2023cxy} gives \[{\rm GW}'_{Y_3}(\mathbf{Q},\varepsilon) = {\rm PT}_{Y_3}(q,\mathbf{Q})\] order by order in $\mathbf{Q}$, as an identity of (resummed asymptotic expansions near $\varepsilon=0$ of) rational functions of $q=-\re^{\ri\varepsilon}$. Substituting in \eqref{eq:dtcrc_degzero}, we get
\[{\rm DT}_{\cX_3}(\mathbf{q}) = M(-q)^{l+1}\,{\rm GW}'_{Y_3}(\mathbf{Q},\varepsilon)\,{\rm GW}'_{Y_3}(\mathbf{Q}^{-1},\varepsilon)\,.\] 
The claim follows from \eqref{eq:gw_crc_reduced} at $x$ and at $x^{\vee}$, since $\mathbf{Q}(x^{\vee})=\mathbf{Q}(x)^{-1}$ by \eqref{eq:x_dual}.
\end{proof}

\begin{rmk}
\label{rmk:square}
For the Calabi--Yau torus, the Gopakumar--Vafa invariants of $Y_3$ are $n_{g,\beta}=-\delta_{g,0}$ for $\beta$ a positive root of $\mathrm{D}_l$ and vanish otherwise \cite{Mozgovoy:2011ps}. Then,
\[
\log{\rm GW}'_{Y_3}(\mathbf{Q},\varepsilon) \;=\; \sum_{g\geq0}c_g\,\varepsilon^{2g-2}\sum_{\alpha>0}\mathrm{Li}_{3-2g}\big(\mathbf{Q}^{\alpha}\big)\,,
\]
with $c_0=-1$ and \[c_g = \frac{(-1)^{g}B_{2g}}{2g\,(2g-2)!}\] for $g\geq1$. Since $\mathrm{Li}_{3-2g}(z^{-1}) = \mathrm{Li}_{3-2g}(z)$ for $g\geq2$, while for $g=0,1$ the two sides differ by a polynomial in $\log z$, the two Gromov--Witten factors in \cref{thm:gwdt_orbifold} agree genus by genus up to classical terms, and
\[
{\rm DT}_{\cX_3}(\mathbf{q}) \;\doteq\; M(-q)^{l+1}\,{\rm GW}'_{\cX_3}(x,\varepsilon)^{2}\,.
\]
\end{rmk}

\begin{rmk}
In \cite{Bryan:2008xra}, the authors compute the Gromov--Witten theory of the Calabi--Yau threefolds obtained from the $G$-Hilb resolution of the quotient singularity $\bbC^3/G$, where $G<\mathrm{SO}(3;\bbC)$ is a polyhedral group. In terms of the associated ADE root system, the authors show that the resulting Gromov--Witten potentials are obtained from those of the resolution of the corresponding ADE surface singularity upon restriction to the kernel of the set of \emph{binary roots} (see \cite[Sec.~1]{Bryan:2008xra} for the details). This is a stratum in $\mathrm{Discr}\,\mathrm{QH}_{\bbC^\times}(Y)$, which is furthermore a natural Frobenius submanifold \cite{Brini:2026mzo}. As \cite{Hu:2013rfh} establishes the Ruan/Bryan--Graber form of the crepant resolution conjecture for type $\mathrm{D}_l$ polyhedral threefold singularities, the analysis of the previous Sections lends itself to an immediate generalisation (or rather, in the sense above, a restriction) of the proofs of \cref{conj:ccit,conj:iritani-integral-structure,conj:quantised_crc} to that setting. The same chain of arguments used in this Section then implies the orbifold GW/DT correspondence for $[\bbC^3/G]$. 
\end{rmk}

\begin{appendix}

\section{Twisted periods of Hurwitz Frobenius manifolds}
\label{sec: twisted_periods}
\subsection{Generalities on deformed flat coordinates}
Let $\mathscr{H}$ be an $N$-dimensional genus-zero Hurwitz space \cite{Dubrovin:1996onf}, with universal curve $\pi\colon\mathscr{C}\to\mathscr{H}$ and superpotential $\lambda\colon\mathscr{C}\to\bbP^1$:
\begin{equation*}
\xymatrix
{
\mathscr{C}\ar[rr]^{\lambda}\ar[d]^{\pi}&&\mathbb{P}^1\\
\mathscr{H}&
}.
\end{equation*}
Choose a coordinate $\mu$ on each fibre $\bbP^1$, varying locally with the base, with two marked poles of $\lambda$ at $\mu=0,\infty$. Take the corresponding third-kind primary differential $\phi = \rd\log\mu$, and set $y\coloneqq\log\mu$, so that $\phi = \rd y$. Derivatives with respect to the base coordinates are taken at fixed $y$. Associated to the pair $(\lambda,\phi)$ are the Hurwitz Frobenius manifold $\mathrm{LG}(\lambda,\phi)$ and its Dubrovin dual $\mathrm{LG}(\log\lambda,\phi)$ \cite{Dubrovin:1996onf,Dubrovin:2004dho}, with metrics $\eta$, $g$ and products $\circ$, $\star$ given by
\begin{align*}
    \eta(X,Y) &= \sum_{p \in \mathrm{cr}(\lambda)} \mathop{\rm {Res}}_{\mu = p} \frac{X(\lambda) Y(\lambda)}{\partial\lambda/\partial y}\,\rd y\,, \\
    \eta(X\circ Y, Z) &= \sum_{p \in \mathrm{cr}(\lambda)} \mathop{\rm {Res}}_{\mu = p} \frac{X(\lambda) Y(\lambda) Z(\lambda)}{\partial\lambda/\partial y}\,\rd y\,, \\
    g(X,Y) &= \sum_{p \in \mathrm{cr}(\lambda)} \mathop{\rm {Res}}_{\mu = p} \frac{X(\log\lambda) Y(\log\lambda)}{\partial\log\lambda/\partial y}\,\rd y\,, \\
    g(X\star Y, Z) &= \sum_{p \in \mathrm{cr}(\lambda)} \mathop{\rm {Res}}_{\mu = p} \frac{X(\log\lambda) Y(\log\lambda) Z(\log\lambda)}{\partial\log\lambda/\partial y}\,\rd y\,,
\end{align*}
where $\mathrm{cr}(\lambda)$ is the set of critical points of $\lambda$.

Consider the associated pencils of connections
\begin{align*}
    \nabla_X^{\hbar,I} &= \nabla^{\mathrm{LC},\eta}_X + \hbar X\circ\,, \\
    \nabla_X^{\hbar,II} &= \nabla^{\mathrm{LC},g}_X + \hbar X\star\,.
\end{align*}
By \cite{Dubrovin:1996onf} these are flat for every $\hbar\in\bbC^\times$ (an arbitrary parameter throughout this appendix, specialised to $2\nu/z$ in the body of the paper), and the gradient of a deformed flat coordinate solves the associated linear system: if $\Pi$ is a deformed flat coordinate of $\nabla^{\hbar,I}$ and $t_1,\dots,t_N$ are flat coordinates of $\eta$, then $\xi_\beta = \partial\Pi/\partial t_\beta$ satisfies
\begin{equation}
    \partial_{\alpha}\xi_{\beta} = \hbar C_{\alpha\beta}^{\sigma}\xi_{\sigma},
    \label{eq: qDE_origin}
\end{equation}
where $C^{\sigma}_{\alpha\beta}$ are the structure constants of $\circ$, $\partial_\alpha\circ\partial_\beta = C^{\sigma}_{\alpha\beta}\partial_{\sigma}$. Throughout this appendix structure constants are written in the flat coordinates $t_\alpha$ of $\eta$, repeated indices are summed over, and we set
\begin{align*}
    (C_{\alpha})^{\sigma}_{\beta} & = C_{\alpha\beta}^{\sigma},\\
    E &= E^{\rho}\partial_{\rho},\\
    \mathcal{U}^{\alpha}_{\beta} &= E^{\rho} C^{\alpha}_{\rho\beta},\\
    \mathcal{V} &= \frac{2-d}{2} - \nabla E ,
\end{align*}
where $d$ is the charge of $\mathrm{LG}(\lambda,\phi)$, $\mathcal{L}_E\eta=(2-d)\eta$.

\begin{prop}[\cite{Dubrovin:2004dho}]
In flat coordinates of $\eta$, the system of differential equations for (gradients of) deformed flat coordinates of the dual Frobenius manifold is given by
\begin{equation}
    \partial_\alpha \xi \cdot \mathcal{U} = \xi \cdot (\mathcal{V} + \hbar - 1/2)\cdot C_{\alpha},
    \label{eq: qDE_dual}
\end{equation}
where $\xi = (\xi_1, \dots, \xi_N)$ is viewed as a row vector and $\cdot$ is the matrix multiplication.
\end{prop}
We will prove the following two statements:
\begin{prop}
\label{prop: integral_origin}
For any cycle $\gamma$ on which the integral converges, the gradient of $\int_{\gamma} \re^{\hbar \lambda} \phi$ solves \eqref{eq: qDE_origin}.
\end{prop}
\begin{prop}
\label{prop: integral_dual}
Suppose that
\begin{enumerate}
    \item[(1)]  $E(\lambda) = \lambda$, with differentiation at fixed $y$;
    \item[(2)]  $\nabla E$ is diagonal in the chosen $\eta$-flat coordinates $t_\alpha$.
\end{enumerate}
Then for every $\gamma\in H^{\rm lf}_1(\bbP^1\setminus(\lambda);\lambda^{\hbar})$ the gradient of $\int_{\gamma} \lambda^{\hbar} \phi$ solves \eqref{eq: qDE_dual}.
\end{prop}

\begin{rmk}
\label{rmk:charge_one_from_homogeneity}
The condition $d=1$ in \cite[Prop.~3.4.3]{Ma:2026ths} can be omitted, since it follows from condition (1) in the present setting. Indeed, if $F_s$ denotes the flow of $E$, then at fixed $y$ one has $\lambda(F_s(t),y)=\re^s\lambda(t,y)$. The numerator of the residue formula for $F_s^*\eta$ and $\partial_y\lambda$ are multiplied by $\re^{2s}$ and $\re^s$, respectively, while $\phi=\rd y$ remains unchanged. Consequently,
\[
    F_s^*\eta=\re^s\eta,
    \qquad \mathcal L_E\eta=\eta.
\]
Comparing with $\mathcal L_E\eta=(2-d)\eta$ gives $d=1$.
\end{rmk}

The proofs follow the theory of primitive forms for Hurwitz Frobenius manifolds developed by Milanov \cite{Milanov:2017gzt}.
\subsection{Primitive form and flat coordinates to first structure connection}
\begin{prop}
\label{prop:qde_firstconnection_primitiveforms}
    In the present genus-zero Hurwitz setting there are locally functions $K_{\alpha\beta}(t,y)$ and $c_{\alpha\beta}(t)$, the latter independent of $y$, such that
    \begin{align*}
        \partial_{\alpha}\lambda\,\partial_{\beta}\lambda &= C_{\alpha\beta}^{\sigma}\,\partial_{\sigma}\lambda + K_{\alpha\beta}\,\partial_y\lambda\,,\\
        \partial_\alpha\partial_\beta \lambda-\partial_yK_{\alpha\beta} &= c_{\alpha\beta}(t)\,\partial_y\lambda\,.
    \end{align*}
    Equivalently, for $\hbar\neq0$ and $G_{\alpha\beta}\coloneqq\hbar^{-1}c_{\alpha\beta}$, the one-form $\big(\partial_\alpha\partial_\beta\lambda+\hbar\,\partial_\alpha\lambda\,\partial_\beta\lambda-\hbar\,C^{\sigma}_{\alpha\beta}\partial_\sigma\lambda\big)\rd y$ is exact for the twisted differential $D_\hbar\coloneqq\rd_y+\hbar\,\rd_y\lambda\wedge$, with primitive $K_{\alpha\beta}+G_{\alpha\beta}$.

\end{prop}
\begin{proof}
The first equation says that the Kodaira--Spencer map $T\mathscr{H}\to\pi_*\mathcal{O}_{\mathrm{cr}(\lambda)}$, $v\mapsto\tilde{v}(\lambda)|_{\mathrm{cr}(\lambda)}$, is an isomorphism, see \cite[Sec.~3.1]{Milanov:2017gzt}.

The second equation follows from the fact that $\phi = \rd\log\mu = \rd y$ is a primitive form \cite[Prop.~6.4]{Milanov:2017gzt} and that the $t_\alpha$ are flat coordinates, as follows. Let $\mathcal{H}_{\rm dR}\coloneqq \pi_*H^1_{\rd+\hbar\,\rd\lambda\wedge}(\Omega^\bullet_{\mathscr{C}/\mathscr{H}})$ be the bundle of twisted de Rham cohomology, where $\Omega^\bullet_{\mathscr{C}/\mathscr{H}} = \Omega^\bullet_{\mathscr{C}}/\big(\pi^*\Omega^1_{\mathscr{H}}\wedge \Omega^{\bullet-1}_{\mathscr{C}}\big)$ is the sheaf of relative differential forms. On $\mathcal{H}_{\rm dR}((\hbar^{-1}))$ there is the Gauss--Manin connection
\begin{equation*}
    \nabla^{\rm GM}_{v}[\omega] = \big[\mathrm{rel}\circ \iota_{\tilde{v}}(\rd_{\mathscr{C}} + \hbar\, \rd_{\mathscr{C}} \lambda\wedge)\tilde{\omega}\big],
\end{equation*}
where $\tilde{v}$ and $\tilde{\omega}$ are lifts of $v$ and $\omega$ to a vector field and a differential form on $\mathscr{C}$, and $\mathrm{rel}$ is the projection to relative forms. One of the defining conditions of a primitive form (see \cite[Sec.~5.1]{Milanov:2017gzt} or \cite[Sec.~4.6]{Milanov:0000yza}) is that, in flat coordinates of $\eta$, where the Christoffel symbols of $\eta$ vanish,
\begin{equation*}
   \nabla^{\rm GM}_{\alpha} \nabla^{\rm GM}_{\beta} \phi = \hbar\, C_{\alpha \beta}^{\sigma} \nabla^{\rm GM}_{\sigma} \phi\,.
\end{equation*}
Since $\nabla^{\rm GM}_{\sigma}\phi = \hbar\,\partial_{\sigma}\lambda\,\rd y$, the right-hand side is $\hbar^2 C_{\alpha\beta}^{\sigma}\partial_{\sigma}\lambda\,\rd y$, while the left-hand side is
\begin{equation*}
    \nabla^{\rm GM}_{\alpha} \nabla^{\rm GM}_{\beta} \rd y = \nabla^{\rm GM}_{\alpha}( \hbar\,\partial_{\beta}\lambda\, \rd y) = \hbar\big[(\partial_{\alpha}\partial_{\beta}\lambda + \hbar\,\partial_{\alpha}\lambda \partial_{\beta}\lambda ) \rd y\big].
\end{equation*}
Using the first equation,
\begin{equation*}
\mathrm{LHS} - \mathrm{RHS} = \hbar\big[(\partial_{\alpha}\partial_{\beta}\lambda + \hbar K_{\alpha\beta} \partial_y\lambda) \rd y\big].
\end{equation*}
As $(\rd_y + \hbar\,\rd_y\lambda\wedge ) K_{\alpha\beta} = (\partial_y K_{\alpha \beta} + \hbar K_{\alpha\beta} \partial_y\lambda) \rd y$ is exact,
\begin{equation*}
    0 = \mathrm{LHS} - \mathrm{RHS} = \hbar \big[(\partial_\alpha\partial_\beta\lambda - \partial_y K_{\alpha \beta}) \rd y\big].
\end{equation*}
Put $R_{\alpha\beta}\coloneqq\partial_\alpha\partial_\beta\lambda-\partial_yK_{\alpha\beta}$. By the first identity, $K_{\alpha\beta}$ and $R_{\alpha\beta}\rd y$ have poles only over $\lambda=\infty$. The preceding vanishing and \cite[Lem.~6.1]{Milanov:2017gzt}, with Milanov's parameter equal to $\hbar^{-1}$, give a finite meromorphic primitive $G_{\alpha\beta}=\sum_{k=1}^{M}\hbar^{-k}g_k(t,\mu)$. Comparing powers of $\hbar$ yields
\[
    R_{\alpha\beta}=(\partial_y\lambda)g_1,
    \qquad \partial_y g_k=-(\partial_y\lambda)g_{k+1}\quad(1\leq k<M),
    \qquad \partial_y g_M=0.
\]
By induction, $g_1$ is a polynomial in $\lambda$. At $\mu=0$, where $\lambda$ has a pole of order $m$, the first identity gives $K_{\alpha\beta},R_{\alpha\beta}=O(\mu^{-m})$, while $\partial_y\lambda$ has pole order $m$. Thus $g_1=R_{\alpha\beta}/\partial_y\lambda$ is bounded there and must be constant in $\lambda$. Writing $g_1=c_{\alpha\beta}(t)$, the recursion gives $g_k=0$ for $k\geq2$, hence
\[
    G_{\alpha\beta}=\frac{c_{\alpha\beta}(t)}{\hbar},
    \qquad R_{\alpha\beta}=c_{\alpha\beta}(t)\,\partial_y\lambda.\qedhere
\]

\end{proof}

\begin{rmk}
\label{rmk: primitive_eq2}
The additional term $c_{\alpha\beta}(t)\,\partial_y\lambda$ corrects the identity in \cite[Prop.~3.4.4, Eq.~(3.29)]{Ma:2026ths}. The resulting adjustments are carried through this appendix. The extra contributions are total derivatives whose integrals vanish for the cycles considered here, leaving the conclusions of \cref{prop: integral_origin,prop: integral_dual} unchanged.
\end{rmk}

\begin{proof}[Proof of \cref{prop: integral_origin}]
Let
\[
\Pi(t,\hbar)=\int_{\gamma}\re^{\hbar\lambda}\,\rd y
\qquad\text{and}\qquad
\xi_\beta=\partial_\beta\Pi.
\]
Transport $\gamma$ by the Gauss--Manin connection. We perform integration by parts where the boundary terms vanish, as for closed or rapid-decay cycles, and extend by analytic continuation where appropriate. Using the two identities in \cref{prop:qde_firstconnection_primitiveforms}, a straightforward computation gives
\begin{align*}
\partial_\alpha\xi_\beta
&=\hbar\int_\gamma
\left(\partial_\alpha\partial_\beta\lambda
+\hbar\partial_\alpha\lambda\partial_\beta\lambda\right)
\re^{\hbar\lambda}\,\rd y\\
&=\hbar C_{\alpha\beta}^{\sigma}\xi_\sigma
+\hbar\int_\gamma
\frac{\partial}{\partial y}
\left((K_{\alpha\beta}+G_{\alpha\beta})\re^{\hbar\lambda}\right)\,\rd y\\
&=\hbar C_{\alpha\beta}^{\sigma}\xi_\sigma.
\end{align*}
Thus $\xi$ satisfies \eqref{eq: qDE_origin}.
\end{proof}

\subsection{Twisted periods of dual Frobenius manifolds}
\begin{proof}[Proof of \cref{prop: integral_dual}]
    This argument is based on \cite[Prop.~5.1]{Dubrovin:2004dho}. The difference here is that the Euler vector field can be more general. We use the two identities established in \cref{prop:qde_firstconnection_primitiveforms}.
    Let
    \begin{equation*}
        \Pi(t,\hbar) = \int_{\gamma} \lambda^\hbar \rd y,
    \end{equation*}
    and transport $\gamma$ by the Gauss--Manin connection. We work with $\hbar\ne0$ and integrate total derivatives where their boundary terms vanish, extending the identities by analytic continuation where appropriate. Taking derivatives gives
    \begin{align*}
        \partial_{\alpha}\partial_{\beta} \Pi/\hbar &= \partial_{\alpha}\int_{\gamma}  \partial_{\beta}\lambda \lambda^{\hbar -1}\rd y\\
        &= \int_{\gamma}\big[(\partial_{\alpha}\partial_{\beta}\lambda ) \lambda^{\hbar -1} + (\hbar -1)\partial_{\alpha}\lambda\partial_{\beta}\lambda \lambda^{\hbar -2}\big] \rd y \\
        &= (\hbar-1) \sum_{\delta}C_{\alpha\beta}^{\delta}\int_{\gamma} \partial_{\delta}\lambda \lambda^{\hbar-2} \rd y\\
        &\quad + \int_{\gamma}\frac{\partial}{\partial y}
        \left(K_{\alpha\beta}\lambda^{\hbar-1}+\frac{c_{\alpha\beta}(t)}{\hbar}\lambda^\hbar\right)\rd y\,.
    \end{align*}
    Here we used the two identities from \cref{prop:qde_firstconnection_primitiveforms} and the independence of $c_{\alpha\beta}$ from $y$. The last integral is a total derivative and vanishes by integration by parts. Thus
    \begin{equation*}
        \partial_{\alpha}\partial_{\beta} \Pi/\hbar = (\hbar-1) \sum_{\delta}C_{\alpha\beta}^{\delta}\int_{\gamma} \partial_{\delta}\lambda \lambda^{\hbar-2} \rd y\,.
    \end{equation*}

    We now use the quasi-homogeneity of $\lambda$ to show that $\Pi$ satisfies \eqref{eq: qDE_dual}. By (2) and the Frobenius identity $\nabla\nabla E=0$, we can write
    
    \begin{equation*}
        E = \sum_{\alpha} (a_{\alpha}+b_{\alpha} t_{\alpha})\frac{\partial}{\partial t_{\alpha}},
    \end{equation*}
    with constants $a_\alpha,b_\alpha$. Then $E(\lambda) = \lambda$ implies
    \begin{equation*}
        \lambda = \sum_{\alpha}(a_{\alpha}+b_{\alpha} t_{\alpha})\partial_{\alpha}\lambda.
    \end{equation*}
    Multiplying by $\partial_{\beta}\lambda$ and using the first equation, we get
    \begin{align*}
        \lambda \partial_{\beta}\lambda &= \sum_{\alpha}(a_{\alpha}+b_{\alpha} t_{\alpha}) \partial_{\alpha}\lambda\partial_{\beta}\lambda\\
        &=\sum_{\alpha}(a_{\alpha}+b_{\alpha} t_{\alpha})(\sum_{\sigma}C_{\alpha\beta}^{\sigma} \partial_{\sigma}\lambda + K_{\alpha\beta} \frac{\partial \lambda}{\partial y}).
    \end{align*}
    Recalling the definition of the matrix $\mathcal{U}$, this gives
    \begin{equation}
        \sum_{\sigma}\mathcal{U}^{\sigma}_{\beta} \partial_{\sigma}\lambda  = \lambda \partial_{\beta} \lambda - \sum_{\alpha}(a_{\alpha}+b_{\alpha} t_{\alpha}) K_{\alpha\beta} \frac{\partial \lambda}{\partial y}.
        \label{eq: quasi-homo}
    \end{equation}

    Now, take $\xi = (\partial_{1}\Pi, \dots, \partial_{N}\Pi)$, then the left-hand side of the equation \eqref{eq: qDE_dual} reads
    \begin{align*}
        (\partial_{\beta}\xi \cdot \mathcal{U})_{\alpha}/\hbar &= \mathcal{U}^{\rho}_{\alpha}\partial_{\beta}\xi_{\rho}/\hbar\\
        &= \mathcal{U}^{\rho}_{\alpha} (\hbar-1)C_{\rho\beta}^{\sigma}\int_{\gamma} \partial_{\sigma}\lambda \lambda^{\hbar-2} \rd y \, ,
    \end{align*}
    using associativity of $E\circ \partial_{\alpha}\circ \partial_{\beta}$, this further equals
    \begin{equation*}
        (\hbar-1) C_{\alpha\beta}^{\delta}\int_{\gamma} (\mathcal{U}_{\delta}^{\sigma}\partial_{\sigma}\lambda )\lambda^{\hbar-2} \rd y.
    \end{equation*}

    Now using \eqref{eq: quasi-homo} we just obtained, this equals
    \begin{align}
        &(\hbar-1) C_{\alpha\beta}^{\delta}\int_{\gamma} \partial_{\delta}\lambda \lambda^{\hbar -1} \rd y - (\hbar-1)C_{\alpha\beta}^{\delta}\int_{\gamma}\sum_{\rho}(a_{\rho} + b_{\rho} t_{\rho})K_{\rho\delta} \frac{\partial \lambda}{\partial y}\lambda^{\hbar -2} \rd y\nonumber \\
        =& (\hbar-1) C_{\alpha\beta}^{\delta}\int_{\gamma} \partial_{\delta}\lambda \lambda^{\hbar -1} \rd y + C_{\alpha\beta}^{\delta}\int_{\gamma}\sum_{\rho}(a_{\rho} + b_{\rho} t_{\rho}) \frac{\partial K_{\rho\delta}}{\partial y} \lambda^{\hbar-1} \rd y,
        \label{eq: dual_period_eq1}
    \end{align}
    where integration by parts is used in the second line.

    Applying $\partial_{\delta}$ to  
    \begin{equation*}
        \lambda = \sum_{\rho}(a_{\rho} + b_{\rho} t_{\rho}) \partial_{\rho} \lambda
    \end{equation*}
    and using the second identity of \cref{prop:qde_firstconnection_primitiveforms}, we get
    \begin{equation*}
        \sum_{\rho}E^\rho\frac{\partial K_{\rho\delta}}{\partial y}
        = (1-b_\delta)\partial_\delta\lambda
        -\left(\sum_{\rho}E^\rho c_{\rho\delta}(t)\right)\partial_y\lambda.
    \end{equation*}
    The coefficient of $\partial_y\lambda$ is independent of $y$, and
    \[
        \int_\gamma\partial_y\lambda\,\lambda^{\hbar-1}\rd y
        =\frac{1}{\hbar}\int_\gamma\rd_y(\lambda^\hbar)=0.
    \]
    Consequently, after integration we have
    \[
        \int_\gamma\sum_{\rho}E^\rho\frac{\partial K_{\rho\delta}}{\partial y}\lambda^{\hbar-1}\rd y
        =(1-b_\delta)\int_\gamma\partial_\delta\lambda\,\lambda^{\hbar-1}\rd y.
    \]
    Using this integrated identity, \eqref{eq: dual_period_eq1} is just
    \begin{align*}
        & \sum_{\delta} (\hbar -1 + 1- b_{\delta}) C_{\alpha\beta}^{\delta}\int_{\gamma}\partial_{\delta}\lambda \lambda^{\hbar-1}\rd y\\
        =& 1/\hbar\sum_{\delta} (\hbar - b_{\delta}) C_{\alpha\beta}^{\delta}\xi_{\delta}.
    \end{align*}

    Now recalling that 
    \[\mathcal{V} = \frac{2-d}{2} - \nabla E = 1/2 - \mathrm{diag}\{b_1,\cdots,b_N\},\]
    so 
    \begin{equation*}
        \mathcal{V} - 1/2 = -\nabla E = - \mathrm{diag}\{b_1,\cdots,b_N\}\, , 
    \end{equation*}
    therefore
    \begin{equation*}
        (\partial_{\beta}\xi \cdot \mathcal{U})_{\alpha} = \sum_{\delta}(\hbar + \mathcal{V} - 1/2)^{\delta}_{\delta}C_{\alpha\beta}^{\delta}\xi_{\delta} = (\xi \cdot (\hbar + \mathcal{V} - 1/2) \cdot C_{\beta})_{\alpha}.
    \end{equation*}
    This is exactly \eqref{eq: qDE_dual}.
    \end{proof}

\begin{rmk}
    The equality $d=1$, deduced from (1) in \cref{rmk:charge_one_from_homogeneity}, guarantees $\mathcal{V}-1/2 = -\nabla E$ in the proof. Other homogeneity assumptions can lead to analogous period formulas, as in \cite[Prop.~5.1]{Dubrovin:2004dho}.
\end{rmk}
\subsection{Restricting to Frobenius submanifolds}
The superpotential \eqref{eq:lambdaD} parametrises a submanifold of a Hurwitz space, so we need to know how the systems above restrict to submanifolds. It suffices to consider a submanifold cut out by flat coordinates, $t_{r+1}=\dots=t_N=0$, with $t_1,\dots,t_r$ as coordinates on it.

\begin{prop}
\label{prop:qDE_dual_sub_Hurwitz}
Let $(M,\eta,\circ)$ be a Frobenius manifold with flat coordinates $t_1,\dots,t_N$ and structure constants $C^{k}_{ij}$, and let $\Pi$ satisfy $\partial_i\partial_j\Pi = \hbar\sum_{k=1}^{N}C^{k}_{ij}\partial_k\Pi$ for all $1\leq i,j\leq N$. Suppose that $\partial_{r+1},\dots,\partial_N$ span the $\eta$-orthogonal complement of $\partial_1,\dots,\partial_r$, and that $C_{ijk}\coloneqq\eta(\partial_i\circ\partial_j,\partial_k)$ vanishes for all $1\leq i,j\leq r$ and $k>r$. Then the restriction of $\Pi$ to the locus $t_{r+1}=\dots=t_N=0$ satisfies
\begin{equation*}
    \partial_i\partial_j \Pi = \hbar \sum_{k=1}^r C_{ij}^{k}\,\partial_k \Pi\,,\qquad 1\leq i,j \leq r\,.
\end{equation*}
\end{prop}
\begin{proof}
By the orthogonality assumption $\eta^{kl}$ is block diagonal, so for $i,j\leq r<k$ one has $C^k_{ij} = \sum_{l>r}\eta^{kl}C_{ijl} = 0$, and the sum in the system truncates to $k\leq r$.
\end{proof}

\begin{prop}
\label{prop:rank_twisted_homology}
    For generic $\hbar$, the rank of $H_1^{\rm lf}(\mathbb{P}^1\setminus(\lambda);\lambda^{\hbar})$ is equal to the number of distinct finite poles and zeros of $\lambda$, minus one.
\end{prop}
\begin{proof}
    See \cite[VIII, Prop.~3.1]{Vassiliev:2002okr}.
\end{proof}
\begin{prop}
\label{prop:period_map_full_Hurwitz}
Let $\gamma_1,\dots,\gamma_N$ be linearly independent local sections of the local system over $\mathscr{H}$ with fibre $H^{\rm lf}_1(\bbP^1\setminus(\lambda);\lambda^{\hbar})$. Then, for generic $\hbar$, the periods $\Pi_i\coloneqq\int_{\gamma_i}\lambda^{\hbar}\phi$, $1\leq i\leq N$, are local coordinates on $\mathscr{H}$ near any semisimple point.
\end{prop}
\begin{proof}
It suffices to prove this for one choice of sections. Near a semisimple point, in canonical coordinates $u_1,\dots,u_N$, take for $\gamma_i$ the steepest-descent cycles through the critical points: as $\hbar\to\infty$ each $\Pi_i$ has leading term proportional to $u_i^{\hbar}$, so the Jacobian of $(\Pi_i)$ with respect to $(u_i)$ is non-zero for $|\hbar|$ large, hence for generic $\hbar$ by analyticity.
\end{proof}

\section{Lauricella functions and their analytic continuation}
\label{appendix: Lauricella_func}
\subsection{Lauricella functions}
\label{sec: lauricella_intro}
We recall some standard facts about Lauricella functions, following \cite{Exton:1976oki}.
The Lauricella function $F_D^{(N)}$ is the analytic continuation of the generalised hypergeometric series defined by:
\begin{align}
    &F_D^{(N)}(a;b_1,\dots,b_N;c\big|x_1,\dots,x_N) \nonumber\\
    &:= \sum_{i_1,\dots,i_N\geq 0}\frac{(a)_{\sum i_k}}{(c)_{\sum i_k}}\prod_{k=1}^N\frac{(b_k)_{i_k}}{(i_k)!}x_k^{i_k}.
     \label{eq:lau_series}
\end{align}
    
An important property of $F_D^{(N)}$ is its one-dimensional Euler-type integral representation
\begin{align*}
    &F_D^{(N)}(a;b_1,\dots,b_N;c\,\big|\,x_1,\dots,x_N)\nonumber \\
    &= \frac{1}{\Beta(a,c-a)}\int_0^1 \rd t\, t^{a-1}(1-t)^{c-a-1}\prod_{k=1}^N(1-x_k t)^{-b_k}\,.
\end{align*}

The Lauricella series reduce to the Gauss hypergeometric series when $N=1$ and to the Appell hypergeometric series when $N=2$.

Fixing the variables $x_1,\dots,x_{N-1}$ and re-expanding the series in $x_N$, one gets the following useful formula:
\begin{align}
    & F_D^{(N)}(a;b_1,\dots,b_N;c\big|x_1,\dots,x_N) \nonumber\\
    &= \sum_{i_1,\dots,i_{N-1}}\frac{(a)_{\sum_{j=1}^{N-1}i_j}}{(c)_{\sum_{j=1}^{N-1}i_j}}\prod_{j=1}^{N-1}\frac{(b_j)_{i_j}x_j^{i_j}}{i_j!} {}_2F_1\Big(a+\sum_{j=1}^{N-1}i_j,b_N;c+\sum_{j=1}^{N-1}i_j\,\Big|\, x_N \Big)
    \label{eq:re-expand}
\end{align}

\subsection{Analytic continuation of Gauss hypergeometric functions}
When $N=1$, the Lauricella function is the Gauss hypergeometric function, $F^{(1)}_D(a;b;c\,|\,z) = {}_2F_1(a,b;c\,|\,z)$, and the following connection formulas hold.

If $a-b\notin \mathbb{Z}$,
\begin{align}
    {}_2F_1(a,b;c\,|\,z) =& \frac{\Beta(b-a,a)}{\Beta(c-a,a)}(-z)^{-a}\,{}_2F_1\Big(a,a-c+1;a-b+1\,\Big|\,\frac{1}{z}\Big)\nonumber\\
    & +\frac{\Beta(a-b,b)}{\Beta(c-b,b)}(-z)^{-b}\,{}_2F_1\Big(b,b-c+1;-a+b+1\,\Big|\,\frac{1}{z}\Big)\,,
\label{eq:hyper_geo_conn}
\end{align}
with the principal determinations of $(-z)^{-a}$ and $\log(-z)$, $|\arg(-z)|<\pi$. In the applications of \cref{sec:Umat_typeD} the arguments satisfy $\arg x\in(-\pi,0)$, so that $(-x)^{-a}=\re^{-\pi\ri a}x^{-a}$ and $\log(-x)=\log x+\pi\ri$, $x^{-a}$ and $\log x$ being defined through the flat logarithms \eqref{eq:logkappa_weights}.

On the other hand, if $a=b$ (we won't need the general resonant case $b-a\in\mathbb{Z}$ in the following),
\begin{align}
    {}_2F_1(a,a;c\,|\,z) =& \frac{(-z)^{-a}}{\Beta(a,c-a)} \sum_{k\geq 0}\frac{(a)_{k}(1-c+a)_{k}}{(k!)^2}z^{-k}\big[\log(-z) + 2 \psi (k+1) - \psi(a+k) - \psi(c-a-k)\big]\,.
\label{eq:hyper_geo_conn_log}
\end{align}

\subsection{Analytic continuation of multi-variable Lauricella functions}
\label{sec: Lauricella_func_AC}
The multi-variable case is less well understood. Here and below, $\sim$ denotes equality up to terms of relative order $O(1/x)$ within each branch, where $x$ is the variable sent to infinity ($x_N$ unless otherwise stated) and the remaining variables are fixed. When $a;b_1,\dots,b_N;c$ are generic, applying the one-variable connection formula \eqref{eq:hyper_geo_conn} to the re-expansion \eqref{eq:re-expand} and keeping the leading terms gives the following.
\begin{prop}[{\cite[(243)]{Brini:2019wts}}]
\begin{align*}
        &F_D^{(N)}(a;b_1,\dots,b_N;c\big|x_1,\dots,x_N) \nonumber\\
\sim    &(-x_N)^{-a}\frac{\Beta(b_N-a,a)}{\Beta(c-a,a)}\nonumber\\
        &\times F_D^{(N)}(a;b_1,\dots,b_{N-1},1-c+a;1-b_N+a\big |\frac{x_1}{x_N},\dots, \frac{1}{x_N})\nonumber\\
+       &(-x_N)^{-b_N}\frac{\Beta(a-b_N,b_N)}{\Beta(c-b_N,b_N)}\nonumber\\
        &\times F_D^{(N-1)}(a-b_N;b_1,\dots,b_{N-1};c-b_N\big |x_1,\dots,x_{N-1}).
\end{align*}
\label{prop: acformula_generic}
\end{prop}

When $a;b_1,\dots,b_N;c$ are not generic, the analytic continuation is more complicated, as in the case $N=1$, \eqref{eq:hyper_geo_conn_log}. The key result we need is the following.
\begin{prop}[{\cite[(A.11)]{Akerblom:2005tho}}]
\label{prop:acformula_resonant}
For $M=(m_1,\dots,m_{N-1})\in\mathbb{Z}_{\geq0}^{N-1}$ write $|M|\coloneqq\sum_{i=1}^{N-1}m_i$ and
\begin{align*}
    h_{M,m_N} \coloneqq \psi(1+ |M| + m_N) + \psi (1+m_N) - \psi(a+ |M| +m_N) - \psi(c-a - m_N)\,.
\end{align*}
Then
    \begin{align*}
        &F_D^{(N)}(a;b_1,\dots,b_{N-1},a;c\,\big|\,x_1,\dots,x_N)\nonumber\\
        =&\frac{(-x_N)^{-a}}{\Beta(a,c-a)}\sum_{M}\sum_{m_N=0}^{\infty}\frac{(a)_{|M|+m_N}(1-c+a)_{m_N}}{(|M|+m_N)!\,m_N!}\prod_{i=1}^{N-1}\frac{(b_i)_{m_i}}{m_i!}\times\nonumber\\
        &\big(\log(-x_N)+ h_{M,m_N}\big)\left(\frac{x_1}{x_N}\right)^{m_1}\cdots \left(\frac{x_{N-1}}{x_N}\right)^{m_{N-1}}\left(\frac{1}{x_N}\right)^{m_N} \nonumber\\
        &+ \frac{(-x_N)^{-a}}{\Beta(a,c-a)}\sum_{M}\sum_{m_N = 0}^{|M|-1}\frac{(a)_{m_N}\Gamma(|M|-m_N)}{m_N!\,(c-a)_{|M|-m_N}}\prod_{i=1}^{N-1}\frac{(b_i)_{m_i}}{m_i!}\,x_1^{m_1}\cdots x_{N-1}^{m_{N-1}}\left(\frac{1}{x_N}\right)^{m_N}\,.
    \end{align*}
\end{prop}

\begin{rmk}
\Cref{prop:acformula_resonant} corrects an error in \cite[Eq.~(A.11)]{Akerblom:2005tho}, also present in \cite[Prop.~A.3.2, Eq.~(A.8)]{Ma:2026ths}: the factor $\frac{\Gamma(c-a-|M|)}{\Gamma(c-a+|M|)}(1-c+a)_{2|M|+m_N}$ in the logarithmic series should read $(1-c+a)_{m_N}$.
\end{rmk}

\begin{cor}
\label{cor:acformula_resonant_leading}
    \begin{align*}
        &F_D^{(N)}(a;b_1,\dots,b_{N-1},a;c\,\big|\,x_1,\dots,x_N) \nonumber\\
\sim       &\frac{(-x_N)^{-a}}{\Beta(a,c-a)}(\log(-x_N)+h_{0,0}) \nonumber\\
+       &\frac{(-x_N)^{-a}}{\Beta(a,c-a)}\sum_{|M|\geq 1} \frac{\Gamma(|M|)}{(c-a)_{|M|}}\prod_{j=1}^{N-1}\frac{(b_j)_{m_j}}{(m_j)!}x_j^{m_j}\,,
\end{align*}
where $h_{0,0}=2\psi(1)-\psi(a)-\psi(c-a)$.
\end{cor}

\begin{lem}
    \begin{align*}
        &\sum_{|M|\geq 1} \frac{\Gamma(|M|)}{(c-a)_{|M|}}\prod_{j=1}^{N-1}\frac{(b_j)_{m_j}}{(m_j)!}x_j^{m_j}\\
       =& \frac{d}{d\epsilon}\bigg|_{\epsilon =0}  F_D^{(N-1)}(\epsilon;b_1,\dots,b_{N-1};c-a\big|x_1,\dots,x_{N-1}).
    \end{align*}
\end{lem}
\begin{proof}
For $|M|\geq1$ one has $(\epsilon)_{|M|}=\epsilon\,\Gamma(|M|)+O(\epsilon^2)$, while the term $|M|=0$ of \eqref{eq:lau_series} is $1$; differentiating the series termwise at $\epsilon=0$, which is legitimate for $|x_j|<1$ by uniform convergence, gives the claim.
\end{proof}

Differentiating \cref{prop: acformula_generic} with $a=\epsilon$ at $\epsilon=0$, with $\sim$ now referring to $|x_{N-1}|\to\infty$, we arrive at the following.
\begin{lem}
\begin{align*}
    &\frac{d}{d\epsilon}\bigg|_{\epsilon =0}  F_D^{(N-1)}(\epsilon;b_1,\dots,b_{N-1};c-a\big|x_1,\dots,x_{N-1})\nonumber\\
    &\sim -\log(-x_{N-1}) + (-\psi(b_{N-1})+\psi(-a+c)) \nonumber\\
    & + (- x_{N-1})^{-b_{N-1}} \Beta(-b_{N-1},c-a)\nonumber\\
    &   \times F_D^{(N-2)}(-b_{N-1};b_1,\dots,b_{N-2};c-a-b_{N-1}\,\big|\, x_1,\dots, x_{N-2}),
\end{align*}
where, if $N=2$, the $F_D^{(N-2)}$ term is $1$.
\end{lem}
\begin{proof}
Apply \cref{prop: acformula_generic} to $F_D^{(N-1)}(\epsilon;b_1,\dots,b_{N-1};c-a\,|\,x_1,\dots,x_{N-1})$ with $x_{N-1}$ as the largest variable and differentiate at $\epsilon=0$. In the first branch, $(-x_{N-1})^{-\epsilon}$ contributes $-\log(-x_{N-1})$ and the ratio $\Beta(b_{N-1}-\epsilon,\epsilon)/\Beta(c-a-\epsilon,\epsilon)=\Gamma(b_{N-1}-\epsilon)\Gamma(c-a)/\big(\Gamma(b_{N-1})\Gamma(c-a-\epsilon)\big)$ contributes $-\psi(b_{N-1})+\psi(c-a)$, the Lauricella factor being $1+O(1/x_{N-1})$. The second branch vanishes at $\epsilon=0$ because of the factor $1/\Gamma(\epsilon)$ in $\Beta(\epsilon-b_{N-1},b_{N-1})$, and its derivative is $\Beta(-b_{N-1},c-a)$ times the stated Lauricella function.
\end{proof}

\begin{cor}
    \begin{align}
        &F_D^{(N)}(a;b_1,\dots,b_{N-1},a;c\big|x_1,\dots,x_N) \nonumber\\
\sim       &\frac{(-x_N)^{-a}}{\Beta(a,c-a)}(\log(-x_N)- \log(-x_{N-1})+h'_{0,0}) \nonumber\\
+       &
\frac{\Beta(c-a,-b_{N-1})}{\Beta(c-a,a)}
            (-x_N)^{-a} (- x_{N-1})^{-b_{N-1}}\nonumber\\
    &   \times F_D^{(N-2)}(-b_{N-1};b_1,\dots,b_{N-2};c-a-b_{N-1}\big| x_1,\dots, x_{N-2})
\label{eq:acformula_log}
\end{align}
where
\begin{align*}
    h_{0,0}' &\coloneqq h_{0,0} - \psi(b_{N-1})+\psi(c-a) = 2 \psi(1) - \psi(a) - \psi(b_{N-1})\,,
\end{align*}

and if $N=2$, then the $F_D^{(N-2)}$ term is $1$.

\end{cor}


\end{appendix}

\bibliographystyle{plain}
\bibliography{refs}

\end{document}